\documentclass[aap]{imsart}

\RequirePackage{amsthm,amsmath,amsfonts,amssymb,lipsum,mathtools,dsfont}
\usepackage{subfigure}
\usepackage{acronym}
\usepackage{latexsym}
\usepackage{paralist}
\usepackage{wasysym}
\usepackage{xspace}
\usepackage[font=small,labelfont=bf]{caption}
\usepackage{tikz}
\usetikzlibrary{calc,patterns}
\tikzset{
	cnd/.style={
		draw,circle,minimum size=0.75cm,
	},
}
\tikzset{edge/.style = {->}}
\usepackage[alpine]{ifsym}
\usetikzlibrary{arrows}
\RequirePackage[numbers,sort&compress]{natbib}
\RequirePackage[colorlinks,citecolor=blue,urlcolor=blue]{hyperref}
\RequirePackage{graphicx}
\numberwithin{equation}{section}

\startlocaldefs

\theoremstyle{plain}

\newtheorem{theorem}{Theorem}[section]
\newtheorem{lemma}[theorem]{Lemma}
\newtheorem{corollary}[theorem]{Corollary}
\newtheorem{proposition}[theorem]{Proposition}
\theoremstyle{remark}
\newtheorem{definition}[theorem]{Definition}
\newtheorem*{remark}{Remark}

\newtheorem*{assumption}{Assumption}

\newcommand{\x}{\boldsymbol{x}}
\newcommand{\y}{\boldsymbol{y}}
\DeclarePairedDelimiter{\ceil}{\lceil}{\rceil}

\newcommand{\si}{\sigma}

\newcommand{\ind}{\mathds{1}}

\newcommand{\cA}{\ensuremath{\mathcal A}} 
\newcommand{\cB}{\ensuremath{\mathcal B}} 
\newcommand{\cC}{\ensuremath{\mathcal C}} 
\newcommand{\cD}{\ensuremath{\mathcal D}} 
\newcommand{\cE}{\ensuremath{\mathcal E}} 
\newcommand{\cF}{\ensuremath{\mathcal F}} 
\newcommand{\cG}{\ensuremath{\mathcal G}} 
\newcommand{\cH}{\ensuremath{\mathcal H}} 
\newcommand{\cI}{\ensuremath{\mathcal I}} 
\newcommand{\cJ}{\ensuremath{\mathcal J}} 
\newcommand{\cK}{\ensuremath{\mathcal K}} 
\newcommand{\cL}{\ensuremath{\mathcal L}} 
\newcommand{\cM}{\ensuremath{\mathcal M}} 
\newcommand{\cN}{\ensuremath{\mathcal N}} 
 
\newcommand{\cP}{\ensuremath{\mathcal P}} 
\newcommand{\cQ}{\ensuremath{\mathcal Q}} 
\newcommand{\cR}{\ensuremath{\mathcal R}} 
 
\newcommand{\cT}{\ensuremath{\mathcal T}}

\newcommand{\E}{\ensuremath{\mathbb{E}}}

\newcommand{\N}{\ensuremath{\mathbb{N}}}
\newcommand{\Z}{\ensuremath{\mathbb{Z}}}
\newcommand{\Q}{\ensuremath{\mathbb{Q}}}
\newcommand{\R}{\ensuremath{\mathbb{R}}}
\renewcommand{\P}{\ensuremath{\mathbb{P}}}

\newcommand{\pl}{\ensuremath{\left\langle}}
\newcommand{\pr}{\ensuremath{\right\rangle}}

\endlocaldefs

\begin{document}
	\begin{frontmatter}
		\title{Meeting, coalescence and consensus time \\ on random directed graphs}
		\runtitle{Meeting, coalescence and consensus on random digraphs}
		
		\begin{aug}
			\author[A,C]{\fnms{Luca}~\snm{Avena}\ead[label=e1]{luca.avena@unifi.it}},
			\author[A]{\fnms{Federico}~\snm{Capannoli}\ead[label=e2]{f.capannoli@math.leidenuniv.nl}},\\
			\author[A]{\fnms{Rajat Subhra}~\snm{Hazra}\ead[label=e3]{r.s.hazra@math.leidenuniv.nl}},
			\and
			\author[B]{\fnms{Matteo}~\snm{Quattropani}\ead[label=e4]{matteo.quattropani@uniroma1.it}}
			\address[A]{Universiteit Leiden, The Netherlands\printead[presep={,\ }]{e2,e3}}
			
			\address[B]{Dipartimento di Matematica ``Guido Castelnuovo'', Sapienza Universit\`a di Roma, Italy\printead[presep={,\ }]{e4}}
			
			\address[C]{Dipartimento di Matematica e Informatica "Ulisse Dini" - Universit\`a degli Studi di Firenze, Italy\printead[presep={,\ }]{e1}}
		\end{aug}
		
		\begin{abstract}
			
			We consider the so-called Directed Configuration Model (DCM), that is, a random directed graph with prescribed in- and out-degrees. In this random geometry, we study the meeting time of two random walks starting at stationarity, the coalescence time for a system of coalescent random walks, and the consensus time of the classical voter model. Indeed, it is known that the latter three quantities are related to each other under certain \emph{mean field conditions} requiring fast enough mixing time and not too concentrated stationary distribution.
			Our main result shows that, for a typical large graph from the DCM ensemble, the meeting time is well-approximated by an exponential random variable for which we provide the first-order asymptotics of its expectation, showing that the latter is linear in the size of the graph, and its pre-constant depends on some explicit statistics of the degree sequence. As a byproduct, we explore the effect of the degree sequence in changing the meeting, coalescence and consensus time by discussing several classes of examples of interest also from an applied perspective. 
			Our approach follows the classical idea of converting meeting into hitting times of a proper collapsed chain, which we control by the so-called First Visit Time Lemma. The main technical challenge is related to the fact that in such a directed setting the stationary distribution is random, and it depends on the whole realization of the graph. As a consequence, a good share of the proof focuses on showing that certain functions of the stationary distribution concentrate around their expectations, and on their characterization, via proper annealing arguments.
		\end{abstract}

		\begin{keyword}[class=MSC]
			\kwd[Primary ]{05C80}
			\kwd[; secondary ]{05C82, 82C22}
		\end{keyword}
		
		\begin{keyword}
			\kwd{Random directed graphs, voter model, meeting times, coalescent random walks}
		\end{keyword}
		
	\end{frontmatter}
	\tableofcontents
	
	\section{Introduction}
	The voter model represents a classical interacting particle system on graphs, which has been used to mathematically model the formation of consensus across a given discrete geometry. In the classical model, each vertex is initially assigned either one of two opinions. Then, at exponential random times, a vertex randomly chooses one of its neighbors to adopt its opinion. Usually, one is interested in understanding the distribution of the so-called \emph{consensus time}, i.e., the first time at which all the vertices share the same opinion. Voter models were introduced in the seminal works \cite{CS73} and \cite{HL75}, and the analysis of the consensus times on finite graphs was first conducted in \cite{DW83} and \cite{Cox89}. In particular, in \cite{Cox89} the author provides the exact first-order approximation of the consensus time on the torus of $\Z^d$ in the limit as the size of the graph tends to infinity. The list of examples in which such precise asymptotics can be provided is not very long. In \cite{CFR09}, the authors compute the exact constant in the case of a random regular graph. Clearly, in the latter case, the law of the voter dynamics depends on the specific realization of the graph. Nevertheless, it is possible to show that the expected consensus time (properly rescaled) converges in probability to a constant as the size of the graph goes to infinity. In the same spirit, in \cite{HSDL22} the authors show that for a configuration model with good expansion properties, the limiting constant can be interpreted in terms of an annealed observable of the random walk on the local weak limit of the random graph under consideration. However, such a constant remains implicit in their work. \cite{FO22} studied the asymptotics of the consensus time on inhomogeneous random graph models, such as the Chung-Lu model and the Norros-Reitu model. In their work, the analysis provides the order of magnitude for the consensus time but not the precise preconstant. More generally, in \cite{Dur07}, the author presents several heuristic arguments showing that in many classical random graph models, the order of magnitude of the expected coalescence time is the same as that of the meeting time of two independent random walks (see also \cite{A13}). The rationale underlying these ideas has been made rigorous by Oliveira in \cite{Oli13}. Indeed, the author determines a set of \emph{mean field conditions} on the underlying graph under which the first-order properties of the consensus time can be reduced to those of the meeting time of two independent stationary random walks. Under such conditions, the consensus time can be shown to converge to an explicit random variable in the Wasserstein-1 sense when the size of the graph grows to infinity. Similarly, in \cite{CCC16}, the authors show that a set of similar conditions is sufficient to show that the density of one of the two opinions converges to a Wright-Fisher diffusion in the Skorohod topology.
	
	A good way to grasp the heuristic arguments behind the ``mean field picture'' is to think in terms of a system of \emph{coalescent random walks}, which is the stochastic dual process of the voter model. Dual systems are commonly used to study various interacting particle systems, and for more information, we refer the reader to the classical books \cite{Lig85, Lig99}. As a consequence of such duality, it is possible to study the behavior of the consensus time by analyzing the coalescence time, which is the first time at which all walks coalesce into one. The precise connection between the voter model and coalescent random walks will be covered in detail in Section \ref{sec: Setting}. The idea behind the study of the coalescence time on general geometries goes back to Kingman's coalescence \cite{Kin81} on finite partitions of $n$ elements and translates into the complete graph as a pure death process that jumps from $k$ particles to $k-1$, with $k\in{2,\dots,n}$, in an exponential time of rate $\binom{k}{2}$. This dynamics reflects the complete absence of geometry of the complete graph.
	
	In \cite{Oli13} Oliveira shows that there is a large class of geometries characterized by some minimal conditions that resemble their \emph{mean field} nature. In this case, one can observe a very similar behavior to that of the complete graph. As a consequence, starting with $2\leq k\leq n$ particles, if the random walk on the graph mixes fast and the stationary distribution is roughly uniform, then it is reasonable to expect that the first coalescence event is approximately exponentially distributed with mean $\mathbf{E}[\tau_{\mathrm{meet}}]/\binom{k}{2}$, where $\tau_{\mathrm{meet}}$ is the meeting time of two independent stationary random walks. By iterating this sort of argument, Oliveira is able to conclude that the coalescence time converges to an infinite sum of exponential random variables with the aforementioned expectations.
	
	\subsection{Our contribution}
	The picture depicted in \cite{Oli13} provides a clear recipe to compute the first-order asymptotic of the consensus time on large graphs: check that the \emph{mean field conditions} are satisfied and then compute the first-order asymptotic of the meeting time of two independent walks. In this paper, we follow such a path in the special case in which the underlying sequence of graphs are random and directed, sampled from the so-called Directed Configuration Model (DCM). The DCM is a natural generalization of the classical configuration model, in which the out- and in-degree of every vertex are prescribed as parameters of the model, and the graph is constructed through a uniform matching between the out- and in-stubs. Initially introduced in \cite{CF04}, in the last few years, a number of works have shed some light on the geometry of these random directed graphs and on the behavior of the random walk on them, see \cite{BCS18,CP20,CP21,CQ20,CQ21a,CQ21b,CCPQ21}. It is crucial to realize that, unlike the classical \emph{undirected} configuration model, the stationary distribution of the random walk on these graphs is a delicate random object, depending on the global realization of the graph, not only on its local features.
	
	By invoking the above-mentioned results, we will show that the \emph{mean field conditions} are satisfied by a typical realization of the directed graph. Doing so, we reduce the analysis of the consensus time to that of the meeting time of two stationary random walks on the graph. A key technical tool for this purpose is the so-called First Visit Time Lemma (see \cite{CFR09,CF08,MQS21}). The latter is a powerful instrument, particularly suited to show that, under certain assumptions very similar in spirit to those in \cite{Oli13}, the hitting time of a target vertex by a stationary random walk is (asymptotically) exponentially distributed. Moreover, the First Visit Time Lemma provides a computationally tractable expression for its expectation. In our setting, as we are interested in the meeting time, the underlying \emph{graph} is the product of a directed graph from the DCM with itself, and the \emph{target} is actually a set, i.e., the diagonal of the product graph. Following the approach of \cite{CFR09,MQS21}, in Section \ref{sec:strategy_of_proof}, we explain how to deal with the First Visit Time Lemma in such a setting.
	
	On the technical side, the major effort of this work lies in computing the exact asymptotic preconstant of the meeting time. To carry out the approach mentioned in the paragraph above, beyond the classical tree approximation of sparse random graphs and the techniques introduced in the aforementioned papers, it is required the analysis of a new process, that we call \emph{random walks with reset}. The process can be described as follows: two walks evolves independently up to their meeting, and when sitting on the same vertex they are ``reset'' somewhere else on the graph. However, such a reset distribution depends on the stationary distribution of the random walk, specifically on its square, $\pi^2$. As mentioned earlier, the stationary distribution is a complex random object that depends on the entire realization of the graph. In particular, resetting the walks according to $\pi^2$ imposes some \emph{a priori} limitations on the use of \emph{annealing} arguments. Roughly speaking, it is impossible to generate the \emph{random walks with reset} together with the realization of the graph by letting the walks create the matching along their exploration. Indeed, when the two walks meet for the first time, in order to perform the reset, they need to know the \emph{entire} realization of the graph. To overcome this difficulty, we study the random walks with reset in the case where the reset distribution, $\mu$, is prescribed and does not depend on the graph. By doing so, we show that the quantities we are interested in only depend on a few features of the distribution $\mu$. \emph{A posteriori}, we demonstrate that such features concentrate for the distribution $\pi^2$ and, by means of some continuity argument, we can translate the case of a prescribed $\mu$ to the case $\mu=\pi^2$. A more detailed discussion on this and the other technical novelties is postponed to Sections \ref{sec:strategy_of_proof}, \ref{sec:rws-mu-reset}, and \ref{sec:random-mu}.

	On the applied side, our study of the voter model on random graphs from the DCM ensemble is mostly motivated by the literature in network science and complex systems. In the last 20 years, physicists and computer scientists have produced an incredible amount of research about multi-agent systems on complex networks. A natural first question in such research programs is related to the effect of \emph{first-order conditions} (i.e., the degrees of the network) on the speed at which information, such as news, infectious diseases, and opinions, travels along the network. These reasons have made the study of the configuration models very popular in that community, and rigorous mathematical results have nowadays confirmed many of the physicists' predictions regarding these models. In this scenario, directed graph models are still less understood compared to their undirected counterparts. As mentioned above, the technical complications introduced as a byproduct of the edge orientations might be among the causes of this lack of results up until very recent times. On the other hand, it is clear that directed graphs constitute the natural model for those real-world networks in which directionality plays a prominent role, such as the World Wide Web or social networks like Twitter or Instagram.
	
	Our main results show an explicit characterization of the expected consensus time as a function of a few simple statistics of the degree distributions. In other words, with our results at hand, it is possible to carry out a complete and rigorous analysis of the effect of the first-order properties on the speed of information diffusion in large directed networks, using the DCM as a natural benchmark. In Section \ref{sec:examples}, we invest some time in discussing the results of this analysis with several relevant examples. Finally, in Section \ref{suse:takehome}, we provide the reader with an easily readable take-home message.
	
	\subsection{Outline of the paper}
	
	Before delving into the core of the paper, we conclude this section with an outline of the paper, presenting the entire structure and how the flow of arguments is articulated throughout the rest of the work.
	
	\begin{enumerate}
		\item Section \ref{sec: Setting} is devoted to the presentation of the models of interest, the required notation, and the aforementioned \emph{mean field conditions}. In particular, in Section \ref{suse:rw}, we introduce the random walk on a directed graph. Section \ref{subsection:votermodel} contains a formal description of the voter model, while the duality with coalescent random walks is introduced in Section \ref{CTRWandCoalescence}. In Section \ref{subsection:meanfield}, we formally introduce the \emph{mean field conditions}, recalling the results in \cite{Oli13}. Finally, in Section \ref{sec:DCM}, we rigorously define the Directed Configuration Model and recall its main properties.
		\item In Section \ref{suse:main-results}, we state our main technical contribution in Theorem \ref{th:main} and its consequence on the consensus time in the subsequent Corollary \ref{coro:meet-coal-cons}. In Section \ref{sec:examples}, we utilize these results to analyze the consensus time as a function of the degree sequences.
		
		\item Section \ref{sec:geometryDCM} describes the geometry of the DCM and the properties of the random walk on such random graphs, presenting the results from the literature that will be useful in the rest of the analysis.
		
		\item In Section \ref{sec:strategy_of_proof}, we provide a complete account of the proof strategy. In particular, after recalling the First Visit Time Lemma, we argue that the main result in Theorem \ref{th:main} boils down to three main claims: Propositions \ref{prop:pi-diag}, \ref{prop:return-diag}, and \ref{prop:tilde-mixing}. These propositions will be proved later in Sections \ref{sec:rws-mu-reset}, \ref{sec:random-mu}, and \ref{sec: Mixing time}. A detailed discussion on the organization of the latter three sections is postponed to Section \ref{suse:organization}.
		
		\item Finally, Section \ref{sec:open} concludes the paper with some open problems and possible future directions of research.
	\end{enumerate}
	
	\section{Models and background} \label{sec: Setting}
	Before introducing formally the models of interest, we describe the general geometric setup that will considered in the whole paper, introducing our notations for directed graphs and random walks on them.
	
	We will be interested in finite directed multigraphs (from now on simply \emph{graphs}) with $n\in\N$ labeled vertices, $G([n],E)$, where $[n]=\{1,\dots,n \}$ and 
	$E$ is a multiset with elements in $[n]^2$. We associate to $G$ its adjacency matrix
	\begin{equation}\label{eq:adjacency matrix}
		A(x,y)\coloneqq|\{e\in E\mid e=(x,y)\}|=\#\text{ edges from }x\text{ to } y\,,
	\end{equation}
	where, in this notation, we often refer to $x$ as the \emph{source} and to $y$ as the \emph{destination} of the edge $e=(x,y)$.
	For each $x\in[n]$, we let 
	$$d_x^+=\sum_{y\in[n]}A(x,y)\,,\qquad d_x^-=\sum_{y\in[n]}A(y,x)\,,$$
	denote the out- and the in-degree of $x$, respectively. Notice that we allow multiple edges with the same source and destination, as well as self-loops, i.e., edges in which the source and the destination coincide.
	
	\subsection{Random walks on directed graphs}\label{suse:rw}
	We let  $(X_t)_{t\ge0}$ denote the \emph{continuous-time random walk} on $G$, which is the Markov process with state space $[n]$  and infinitesimal generator given by
	\begin{equation}\label{CTRW}
		L_{\rm rw}f(x)=\sum_{y\in [n]}\frac{A(x,y)}{d_x^+}\left[f(y)-f(x)\right]\,,\qquad f:[n]\to\R\,.
	\end{equation}
	Let $\cP([n])$ denote the set of probability distributions on $[n]$. Fixed an initial distribution $\mu\in\cP([n])$ for the random walk, we let $\mathbf{P}_\mu$ and $\mathbf{E}_\mu$ denote the law and the expectation on the space of trajectories of $(X_t)_{t\ge0}$ with $X_0\sim\mu$.  Moreover, when $\mu$ is concentrated on a single vertex $x\in[n]$, we simply write $\mathbf{P}_x$ (respectively, $\mathbf{E}_x$).
	
	In the following, we will assume $G$ to be \emph{ergodic}, that is, $G$ admits a strongly connected component such that every vertex that is not in the component has at least a (directed) path leading to it.  The latter requirement immediately implies that the random walk on $G$ admits a unique stationary distribution, which we denote by $\pi$, such that
	\begin{equation}\label{eq:conv-to-eq}
		\lim_{t\to\infty}	\mathbf{P}_\mu(X_t=x)=\pi(x)\,,\qquad\forall x\in[n]\,,\,\mu\in\cP([n])\,.
	\end{equation}
	Notice that, if $G$ is ergodic but not strongly connected, then we have ${\rm supp}(\pi)\subsetneq[n]$. 
	
	With the aim of quantifying the speed of the convergence in \eqref{eq:conv-to-eq}, we define the worst-case total-variation distance at time $t$ as 
	\begin{equation}
		d_{\rm TV}(t)\coloneqq\max_{x\in[n]}	\|\mathbf{P}_x(X_t=\cdot)-\pi \|_{\rm TV}= \frac12 \max_{x\in[n]} \sum_{y\in[n]} \left|\mathbf{P}_x(X_t=y)-\pi(y) \right|\,,
	\end{equation}
	and we consider the \emph{mixing time}
	\begin{equation}
		t_{\rm mix}\coloneqq\inf\bigg\{t\ge0\,:\,d_{\rm TV}(t)\le \frac1{2e} \bigg\}\,.
	\end{equation}
	
	In this work we will be particularly interested in considering a system of two independent random walks, that is, the continuous time Markov process $(X_t,Y_t)_{t\ge 0}$ on $[n]^2$ associated to the generator $$L_{\rm rw}^{\otimes 2}=L_{\rm rw}\otimes {\rm Id} +{\rm Id}\otimes L_{\rm rw}\,.$$
	We will consider the stopping time
	\begin{equation}\label{eq:def-tau-meet}
		\tau_{\rm meet}\coloneqq\inf\{ t\ge 0\mid X_t=Y_t\}\,,
	\end{equation}
	called \emph{meeting time}, representing the first time in which the two independent walks meet. Clearly, the law of $\tau_{\rm meet}$ strongly depends on the initial distribution of the two walks.
	
	Now that we set up the geometric framework and all the required preliminary notations, we are in shape to introduce the two models of interested for this work: the \emph{voter model}, in Section \ref{subsection:votermodel}, and the \emph{coalescent random walks}, in Section \ref{CTRWandCoalescence}.
	\subsection{Voter model \& consensus time}\label{subsection:votermodel}
	
	Let $G([n],E)$ be an ergodic graph and call the  \emph{voter model} on $G$  the continuous-time Markov process $(\eta_t)_{t\geq 0}$ with state space $\{0,1\}^{[n]}$, 
	and infinitesimal generator given by
	\begin{equation}\label{VMGen}
		L_{\rm voter}f(\eta)=\sum_{x\in [n]} \sum_{y\in [n]} \frac{A(x,y)}{d^+_x} \left[f(\eta^{x\to y})-f(\eta)\right]\,,\qquad f:\{0,1\}^{[n]} \rightarrow \mathbb{R}\,,
	\end{equation}
	where
	\begin{equation*}
		\eta^{x\to y}(z) := \begin{cases}
			\eta(y), & \quad \text{if } z=x\,, \\
			\eta(z), & \quad \text{otherwise}\,.
		\end{cases}
	\end{equation*}
	
	In words, the variable $\eta_t(x)$ represents the state of node $x$ at time $t$, being either $0$ or $1$, to be interpreted as the binary \emph{opinion} of the individual $x$ at time $t$. 
	The Markov evolution encoded in~\eqref{VMGen} can be phrased as follows. Each vertex $x\in [n] $ has an exponential clock of rate $1$, when such a clock rings, vertex $x$ chooses one of its out-edges at random and adopts the opinion of the vertex at the other extreme of the edge. See also Figure~\ref{fig:evolution-voter} to help visualization.

	\begin{figure}[ht]
		\centering
		\begin{tikzpicture}
			\tikzset{vertex/.style = {shape=circle,draw,minimum size=1.5em}}
			\tikzset{edge/.style = {->,> = latex'}}
			\node[vertex][style={circle, fill=blue!30, very thick, minimum size=7mm}] (11) at  (0,0) {$x$};
			\node[vertex][style={circle, fill=blue!30, very thick, minimum size=7mm}] (21) at  (2,0) {$y$};
			\node[vertex][style={circle, fill=red!30, very thick, minimum size=7mm}] (31) at  (0,-2) {$z$};
			\node[vertex][style={circle, fill=red!30, very thick, minimum size=7mm}] (41) at  (2,-2) {$w$};
			
			\node[vertex][style={circle, fill=blue!30, very thick, minimum size=7mm}] (12) at  (3.5,0) {$x$};
			\node[vertex][style={circle, fill=blue!30, very thick, minimum size=7mm}] (22) at  (5.5,0) {$y$};
			\node[vertex][style={circle,draw=green!60, fill=red!30, very thick, minimum size=7mm}] (32) at  (3.5,-2) {$z$};
			\node[vertex][style={circle, fill=red!30, very thick, minimum size=7mm}] (42) at  (5.5,-2) {$w$};
			
			\node[vertex][style={circle,draw=green!60, fill=blue!30, very thick, minimum size=7mm}] (13) at  (7,0) {$x$};
			\node[vertex][style={circle, fill=blue!30, very thick, minimum size=7mm}] (23) at  (9,0) {$y$};
			\node[vertex][style={circle, fill=blue!30, very thick, minimum size=7mm}] (33) at  (7,-2) {$z$};
			\node[vertex][style={circle, fill=red!30, very thick, minimum size=7mm}] (43) at  (9,-2) {$w$};
			
			\node[vertex][style={circle,draw=green!60, fill=red!30, very thick, minimum size=7mm}] (14) at  (1.5,-3) {$x$};
			\node[vertex][style={circle, fill=blue!30, very thick, minimum size=7mm}] (24) at  (3.5,-3) {$y$};
			\node[vertex][style={circle, fill=blue!30, very thick, minimum size=7mm}] (34) at  (1.5,-5) {$z$};
			\node[vertex][style={circle, fill=red!30, very thick, minimum size=7mm}] (44) at  (3.5,-5) {$w$};
			
			\node[vertex][style={circle, fill=blue!30, very thick, minimum size=7mm}] (15) at  (5.5,-3) {$x$};
			\node[vertex][style={circle, fill=blue!30, very thick, minimum size=7mm}] (25) at  (7.5,-3) {$y$};
			\node[vertex][style={circle, fill=blue!30, very thick, minimum size=7mm}] (35) at  (5.5,-5) {$z$};
			\node[vertex][style={circle, fill=red!30, very thick, minimum size=7mm}] (45) at  (7.5,-5) {$w$};
			
			\draw[edge](31.80) -- (11.280);
			\draw[edge](11.260) -- (31.100);
			\draw[edge] (11) to (41);
			\draw[edge] (31) to (21);
			\draw[edge] (21) to (41);
			\draw[edge] (41) to (31);
			
			\draw[edge][style={draw=green!60}](32.80) -- (12.280);
			\draw[edge] (12.260) -- (32.100);
			\draw[edge] (12) to (42);
			\draw[edge] (32) to (22);
			\draw[edge] (22) to (42);
			\draw[edge] (42) to (32);
			
			\draw[edge](33.80) -- (13.280);
			\draw[edge] (13.260) -- (33.100);
			\draw[edge][style={draw=green!60}] (13) to (43);
			\draw[edge] (33) to (23);
			\draw[edge] (23) to (43);
			\draw[edge] (43) to (33);
			
			\draw[edge](34.80) -- (14.280);
			\draw[edge][style={draw=green!60}] (14.260) -- (34.100);
			\draw[edge] (14) to (44);
			\draw[edge] (34) to (24);
			\draw[edge] (24) to (44);
			\draw[edge] (44) to (34);
			
			\draw[edge](35.80) -- (15.280);
			\draw[edge] (15.260) -- (35.100);
			\draw[edge] (15) to (45);
			\draw[edge] (35) to (25);
			\draw[edge] (25) to (45);
			\draw[edge] (45) to (35);
		\end{tikzpicture}
		\caption{From left to right, up-to-down, the pictures describe a possible evolution of the voter model with generator as in \eqref{VMGen} on a directed graph with $n=4$ vertices and initial opinions as in the first picture. In particular, at each step, the vertices with green boundary are the ones whose exponential clock rings first, while the corresponding green edges are the randomly selected ones among the out-neighbour on the green vertex. }
		\label{fig:evolution-voter}
	\end{figure}
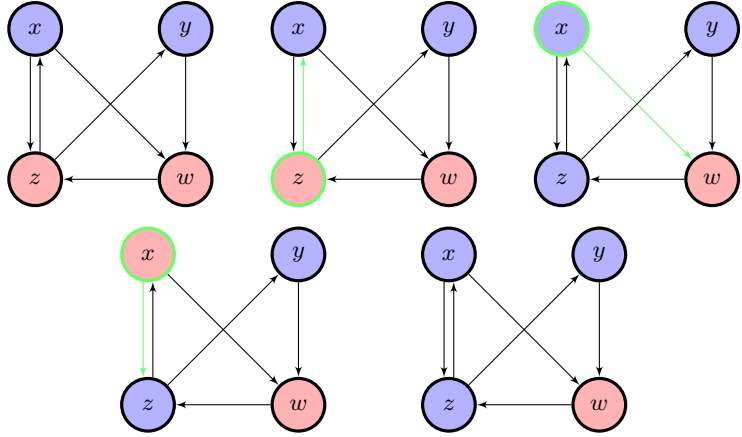
	\begin{figure}[ht]
		\centering
		\subfigure[]{\begin{tikzpicture}[scale=0.64]
				
				\draw (-3,0) -- (6,0);
				\draw[red,very thick] (3,0) -- (3,5);
				\draw[red,very thick] (-3,0) -- (-3,1.5);
				\draw[red,very thick] (0,1.5) -- (0,4);
				\draw[blue,very thick] (0,0) -- (0,2.5);
				\draw[blue,very thick] (-3,1.5) -- (-3,5);
				\draw[blue,very thick] (0,4) -- (0,5);
				\draw[blue,very thick] (6,0) -- (6,5);
				\draw[->](-4,0) -- (-4,5);
				\node[rotate=90] at (-4.3,2.5) {Voter Model};
				
				\filldraw[blue] (0,0) circle (3pt);
				\filldraw[black] (0,0) node[below]{$x$};
				
				\filldraw[blue] (6,0) circle (3pt);
				\filldraw[black] (6,0) node[below]{$y$};
				
				\filldraw[red] (-3,0) circle (3pt);
				\filldraw[black] (-3,0) node[below]{$z$};
				
				\filldraw[red] (3,0) circle (3pt);
				\filldraw[black] (3,0) node[below]{$w$};
				
				\draw [-stealth,very thick](-3,1.5) -- (0,1.5);
				\draw [stealth-,very thick](3,2.5) -- (0,2.5);
				\draw [stealth-,very thick](-3,4) -- (0,4);
				
		\end{tikzpicture}}
		\subfigure[]{\begin{tikzpicture}[scale=0.64]
				
				\draw (-3,0) -- (6,0);
				\draw[black,thick] (-3,0) -- (-3,5);
				\draw[black,thick] (0,0) -- (0,5);
				\draw[black,thick] (3,0) -- (3,5);
				\draw[black,thick] (6,0) -- (6,5);

				\filldraw[purple] (6,5) circle (3pt);
				\filldraw[purple] (3,5) circle (3pt);
				\filldraw[purple] (0,5) circle (3pt);
				\filldraw[purple] (-3,5) circle (3pt);
				
				\filldraw[purple] (6,4) circle (3pt);
				\filldraw[purple] (3,4) circle (3pt);
				\filldraw[purple] (0,4) circle (3pt);
				\filldraw[purple] (-3,4) circle (3pt);
				\filldraw[purple!40] (-1.5,4) circle (3pt);
				
				\filldraw[purple] (6,2.5) circle (3pt);
				\filldraw[purple] (3,2.5) circle (3pt);
				\filldraw[purple] (-3,2.5) circle (3pt);
				
				\filldraw[purple] (6,1.5) circle (3pt);
				\filldraw[purple] (3,1.5) circle (3pt);
				\filldraw[purple] (-3,1.5) circle (3pt);
				\filldraw[purple!40] (-1.5,1.5) circle (3pt);
				
				\filldraw[purple] (6,0) circle (3pt);
				\filldraw[purple] (3,0) circle (3pt);
				\filldraw[purple] (0,0) circle (3pt);
				
				\filldraw[black] (0,0) node[below]{$x$};
				\filldraw[black] (6,0) node[below]{$y$};
				\filldraw[black] (-3,0) node[below]{$z$};
				\filldraw[black] (3,0) node[below]{$w$};
				
				\draw [-stealth,very thick](-3,1.5) -- (0,1.5);
				\draw [stealth-,very thick](3,2.5) -- (0,2.5);
				\draw [stealth-,very thick](-3,4) -- (0,4);
				
				\draw[->](7,5) -- (7,0);
				\node[rotate=90] at (7.4,2.5) {Coalescing RWs};
		\end{tikzpicture}}
		\caption{Picture (a) is the graphical representation of the voter model on the graph shown in Figure \ref{fig:evolution-voter}; while picture (b) represents its dual time reversal consisting of a system of coalescing random walks (CRWs). The two side arrows in the pictures represent the direction of time: in the voter model it runs upwards, while in the CRWs systems it runs downwards. In the voter model, picture (a), when the exponential clock associated to the edge $x\to y$ rings, we attach an arrow from $x$ to $y$. We let  opinions spread vertically until they reach the tail of a arrow. In that case the opinion changes according to the one sitting on the head of the arrow. In order to trace back in time the evolution of the opinions, at the initial time (corresponding to the final time for the voter model) we put independent random walks on each vertex and let them evolve back in time, following the same black arrows, from the tail to the head, with the added coalescing feature. 
			As we can observe in (b), the two walks starting in $x$ and $z$ have coalesced into one particle that ended up in vertex $x$. This means that in the voter model the vertices $x$ and $z$ share the same opinion at the final time, as it can be checked in (a), and such opinion comes from the original opinion of $x$.}
		\label{fig:graphicalRepr}
	\end{figure}
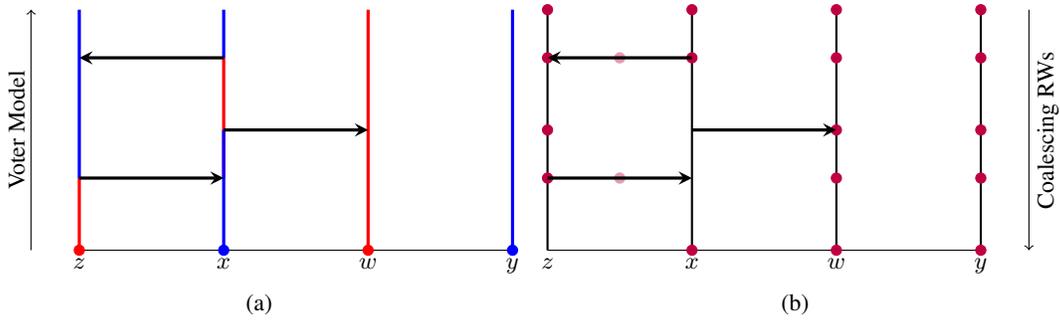
	Under the ergodicity assumption, the voter model is a Markov chain with only two absorbing states, i.e., the monochromatic configurations $\bar{0}$ and $\bar{1}$ consisting of all $0$'s and $1$'s, respectively. As a consequence, regardless of the initial configuration of opinions, almost surely the system reaches in finite time one of these absorbing state. This naturally leads to the question of understanding the distribution of the \emph{consensus time}, defined as
	\begin{equation}\label{Tcons}
		\tau_{\rm cons} := \inf \big\{t \geq 0\,\colon\, \eta_{t} \in \{\bar{1},\bar{0}\} \big\}\,,
	\end{equation}
	as a function of the underlying geometry. 
	
	\subsection{Coalescent random walks}\label{CTRWandCoalescence}
	
	As we will now argue, it is convenient to construct the trajectories of the process $(\eta_t)_{t\ge 0}$ by using a collection of independent Poisson processes indexed by the directed edges, $x\to y$, with corresponding intensities $1/d^+_x$. When the clock associated to an edge rings, say $x\to y$,  the vertex at the origin of the edge, $x$, adopts the opinion of the vertex at the destination of the edge, $y$.
	
	As we will now recall, this Poissonian construction allows to couple the dynamics of the voter model with that of a \emph{dual process}, known as \emph{coalescent random walks}. The latter is a continuous time stochastic process on $[n]^n$ which can be described as follows: the process starts with a random walk sitting on each vertex of the graph. Marginally, each random walk evolves according to the generator \eqref{CTRW}, with the only difference that, when two (or more) walks sit on the  same vertex, then they stick together and continue their trajectory as a single walk. Clearly, under the ergodicity assumption, in finite time all the walks will coalesce on a single walk, which will then continue its trajectory according to the generator in \eqref{CTRW}. To simplify the reading, let $(X_t^x)_{t\ge0}$ denote the trajectory of the walk starting at $x\in[n]$, and define the \emph{coalescence time} as the first time at which all the walks are on the same site, i.e.,
	\begin{equation}\label{Tcoal}
		\tau_{\rm coal}:= \inf\{ t\ge 0\,:\,X_t^x=X_t^y\,,\forall x,y\in[n]\}\,.
	\end{equation}
	Now observe that the graphical representation used for the construction of the voter model dynamics can be used also to sample the trajectory  of the coalescent random walks: reversing the direction of time, when the clock associated to the edge $x\to y$ rings, then the walk(s) sitting at $x$ (if any), move to $y$. In other words, the voter model, the random walk and the system of coalescing random walks can be sampled using the same graphical construction as a source of randomness. Moreover, the same conclusion still holds true for a system of two independent random walks, as soon as we restrict to $t\ge 0$ smaller than the stopping time $\tau_{\rm meet}$ in \eqref{eq:def-tau-meet}. Therefore, there is no ambiguity in denoting all their laws and expectations with the same symbols, $\mathbf{P}$ and $\mathbf{E}$, respectively. 
	
	The beauty of the mentioned graphical construction can be realized by observing that it can be used to describe the ancestral  history of the opinions on our graph. Indeed, for any collection $x_1,\dots,x_k\in[n]$ of distinct vertices and a prescribed time $t$, it is not hard to realize that
	\begin{equation}\label{BackOpinion}
		(\eta_t(x_1),\dots,\eta_t(x_k))\overset{d}{=}(\eta_0\left(X^{x_1}_t),\dots,\eta_0(X^{x_k}_t)\right)\,,
	\end{equation}
	as exemplified in Figure~\ref{fig:graphicalRepr}.
	It is common to refer to the distributional relation in~\eqref{BackOpinion}  by saying that the two models, voter and coalescent random walks, are \emph{dual} of each other.
	
	An immediate consequence of \eqref{BackOpinion} is that the coalescence time in \eqref{Tcoal} is related to the consensus time in~\eqref{Tcons} by the stochastic domination
	$$\tau_{\rm cons} \preceq \tau_{\rm coal}\,,$$
	which holds regardless of the initial configuration $\eta_0$.
	
	Having introduced the models, their properties, and all the required notations, we are now ready to discuss in some detail part of the literature about them. In particular, we will focus on presenting some condition on the underlying sequence of graphs which ensure that the  asymptotic behaviour of the two processes is determined by some easier graph feature, i.e., the meeting time of two independent random walks started at stationarity.

	\subsection{\emph{Mean field conditions} for coalescence and consensus}\label{subsection:meanfield}
	As mentioned in the Introduction, the very first example in which a detailed analysis of the behaviour of the voter and the coalescence dynamics can be carried out is that of a complete graph, which is particularly simple due to the absence of geometry. In such a case, the distribution of $\tau_{\rm coal}$ can be computed explicitly, see \cite[Ch. 14]{AF02} and \cite{Oli13}, obtaining
	\begin{equation}\label{eq:tau-coal-complete}
		\frac{\tau_{\mathrm{coal}}}{(n-1)/2}\overset{d}{=}\sum_{i=2}^nZ_i\,,
	\end{equation}
	where the $Z_i$'s are independent random variables with law
	\begin{equation}\label{eq:def-Z}
		Z_i\overset{d}{=}{\rm Exp}\left(\binom{i}{2}\right)\,,\qquad i\ge 2\,.
	\end{equation}
	In particular, by taking the expectation one can conclude that $\mathbf{E}[\tau_{\mathrm{coal}}]=n-1$. To interpret the denominator in \eqref{eq:tau-coal-complete}, notice that, recalling the definition of $\tau_{\rm meet}$ in \eqref{eq:def-tau-meet},
	$$\mathbf{E}_{\pi\otimes\pi}[\tau_{\rm meet}]\sim \frac12\ (n-1)\,,$$
	where the notation $\pi\otimes\pi$ stands for the fact that we assume the two walks to start independently with law $\pi$, where in this case $\pi$ is uniform over $[n]$.
	In other words, the distributional identity in \eqref{eq:tau-coal-complete} can be rephrased by saying that $\tau_{\rm coal}$ converges in distribution to the sum of exponential random variables when scaled accordingly to $\mathbf{E}_{\pi\otimes\pi}[\tau_{\rm meet}]$.
	
	A natural question is whether this picture is true in more general setups. In fact, one would expect that, if the random walk on the graph mixes quickly and the stationary distribution is not too concentrated, then the coalescence time behaves as in \eqref{eq:tau-coal-complete} after a proper rescaling. Indeed, the above scenario can be made rigorous in the case of $d$-dimensional tori (with $d\ge 3$), \cite{Cox89}, and for \emph{most} regular graphs, \cite{CFR09}.
	In this spirit, in \cite{Oli13} Oliveira determines a set of \emph{mean field conditions} on the underlying graph sequence ensuring the convergence in distribution of the rescaled coalescence time to the sum of random variables in \eqref{eq:tau-coal-complete}. In such a general framework, the scaling factor will be given by $
	\mathbf{E}_{\pi\otimes\pi}[\tau_{\rm meet}]$.
	
	Before stating Oliveira's results we point out that, to ease the reading, we specialize the statements to the setting of random walks on directed graphs introduced above, so that they look simpler (and less general) than in the original paper \cite{Oli13}. We also remark that here and throughout the whole paper,  when considering distances between probability measures, we will adopt the standard abuse of notation of identifying random variables with their law.

	\begin{theorem}[Theorem 1.2 in \cite{Oli13}]\label{MFcoal}
		Consider a sequence of ergodic graphs, $G=G_n([n],E)$, and let $\pi=\pi_n$ be the  stationary distribution of the random walk on $G$.
		Let $\pi_{\max}:=\max_{x\in[n]}\pi(x)$ and assume that \begin{equation}\label{MFcondCoal}
			\lim_{n\to\infty} t_{\rm mix}\ \pi_{\max}\ \log^5(n)=0\,.
		\end{equation}
		Then
		\begin{equation}
			\lim_{n\to\infty}d_W\left(\frac{\tau_{\rm coal}}{\mathbf{E}_{\pi\otimes\pi}[\tau_{\rm meet}]}\,,\,\sum_{k\geq 2}Z_k\right)=0\,,
		\end{equation}
		where $d_W$ denotes the $L^1$ Wasserstein distance, and the $Z_k$'s are independent random variables with  law as in \eqref{eq:def-Z}.
		In particular,
		\begin{equation}\label{doubleTime}
			\lim_{n\to\infty}\frac{\mathbf{E}[\tau_{\rm coal}]}{ \mathbf{E}_{\pi\otimes\pi}[\tau_{\rm meet}]} = 2\,.
		\end{equation}
	\end{theorem}
	In words, the latter theorem tells us that as soon as we are able to verify \eqref{MFcondCoal}, then the behavior of the coalescence time is the same as in the complete graph, and only the scaling factor, $\mathbf{E}_{\pi\otimes\pi}[\tau_{\rm meet}]$, needs to be determined. 
	A similar picture holds for the voter dynamics. Let $u\in(0,1)$ and consider again $G$ to be the complete graph. 
	
	Indeed, it turns out that the asymptotic law of the consensus time can be expressed similarly to that of $\tau_{\mathrm{coal}}$, as pointed out by Oliveira in the following result.
	\begin{theorem}[Theorem 1.3 in \cite{Oli13}]\label{th:mean-cons}
		Under the same assumption of Theorem \ref{MFcoal}, the consensus time of the voter model started with a product of i.i.d. Bernoulli opinions of parameter $u\in(0,1)$ satisfies
		\begin{equation}
			\lim_{n\to\infty}d_W\left(\frac{\tau_{\rm cons}}{\mathbf{E}_{\pi\otimes\pi}[\tau_{\rm meet}]}\,,\,\sum_{k> K}Z_k\right)=0\,,
		\end{equation}
		where
		\begin{equation}\label{eq:def-K-oli}
			K\overset{d}{=} U A + (1-U) B\,,\qquad 	U\overset{d}{=}{\rm Bern}(u)\,,\qquad A\overset{d}{=}{\rm Geom}(1-u)\,,\qquad B\overset{d}{=}{\rm Geom}(u)\,,
		\end{equation}
		and the $Z_k$'s are independent random variables with  law as in \eqref{eq:def-Z}.
		In particular,
		\begin{equation}\label{doubleTime-new3}
			\lim_{n\to\infty}\frac{\mathbf{E}_u[\tau_{\rm cons}]}{ \mathbf{E}_{\pi\otimes\pi}[\tau_{\rm meet}]} = -2\left[(1-u)\log(1-u)+u\log(u) \right]\,.
		\end{equation}
	\end{theorem}
	Notice that the function of $u$ on the right-hand side of \eqref{doubleTime-new3} is symmetric around $\frac{1}{2}$ in the interval $[0,1]$ and it is maximal for $u=\frac12$ where it attains the value $2\log(2)\approx1.38$.  Thanks to Theorems \ref{MFcoal} and \ref{th:mean-cons}, one can conclude that an asymptotic analysis of the consensus time and of the coalescence time of can  be derived, under the \emph{mean field conditions}, by a precise asymptotic of the expected meeting time along the graph sequence. 
	
	\subsection{Directed configuration model (DCM)} \label{sec:DCM}
	In this section we formally introduce the random graph model that we will consider throughout all the paper and state some of its typical properties. The Directed Configuration Model (DCM) is a natural generalization to the directed setting of the classical configuration model introduced by Bollobas \cite{Bol80} in the early '80s . See also \cite{vdH16} for a modern introduction to the topic.
	For each $n\in\mathbb{N}$ fix two finite sequences $\mathbf{d}^+=\mathbf{d}^+_n=(d^+_x)_{x\in[n]}\in \mathbb{N}^n$ and $\mathbf{d}^-=\mathbf{d}^-_n=(d^-_x)_{x\in[n]}\in \mathbb{N}_0^n$ such that 
	\begin{equation*}
		m=m_n\coloneqq\sum_{x\in[n]}d_x^+ = \sum_{x\in[n]}d_x^-\, ,
	\end{equation*}
	and let $d_{\text{min}}^\pm  = \min_{x\in[n]} d^\pm_x $ and $d_{\text{max}}^\pm =\max_{x\in[n]} d^\pm_x $. We will work under the following assumptions on the degree sequences $(\mathbf{d}^+,\mathbf{d}^-)$:
	\begin{assumption}[Degree assumptions]\label{log degree assumptions}
		There exists some constant $C\ge 2$ such that for all $n\in \N$
		\begin{equation}
			\begin{split}
				&\text{(a):}\qquad d_{\text{min}}^+\geq 2\, , \\
				&\text{(b):}\qquad d_{\max}^+\le C\, , \\
				&\text{(c):}\qquad d_{\max}^-\le C\, . 
			\end{split}
		\end{equation}
	\end{assumption}
	
		\begin{remark}
			As we will discuss more extensively in Section \ref{suse:degreeassumptions}, it is worth to notice that most of the proofs in this paper work under weaker assumptions. We decided to state our results under the above degree assumptions so to rely on previous results, \cite{CP20}, on the minimum value of the stationary distribution $\pi$. However, we believe that the result in \cite{CP20} could be extended to much weaker assumptions. Nevertheless, an extension of these results is out of the scope of this work.
		\end{remark}
	
	Notice that assumptions (b) and (c) imply that our graphs are \emph{sparse}, in the sense that $m\asymp n$, and that there are no restrictions on the minimal in-degree, therefore vertices with in-degree 0 are allowed.
	Assign to each vertex $x\in[n]$,  $d^-_x$ labeled \emph{heads} and $d^+_x$ labeled \emph{tails}, denoting respectively the in- and out-stubs of $x$. Call $E_x^-$ and $E_x^+$ the sets of labeled heads and tails of $x$, respectively. Further, let $E^{\pm}=\cup_{x\in [n]}E_x^\pm$. Let $\omega=\omega_n$ be a uniformly random bijection $\omega:E^+ \rightarrow E^-$, viewed as a matching between tails and heads. The latter bijection can be projected to produce a directed graph $G=G_n([n],E)$, obtained by adding a directed edge $x\to y$ for every $f\in E_y^-$ and $e\in E_x^+$ such that $\omega(e)=f$. In what follows we let $\P$ (resp. $\E$) denote the probability law (resp. the expectation) of the sequence of random bijections $(\omega_n)_{n\in\N}$, and we will say that a sequence of graphs is sampled from DCM$(\mathbf{d}^+,\mathbf{d}^-)$ to mean that for every $n$ the graph $G=G_n$ is sampled according to the procedure above.
	We will be interested in studying the asymptotic regime in which $n\to\infty$, and we will say that $G$ has a certain property \emph{with high probability (w.h.p.)}, if the probability that $G_n$ has such a property goes to $1$ as $n$ goes to infinity. 
	
	Being $G$ random, so it is the law of the random walk on it. In particular, the stationary distribution $\pi$ is a non-trivial random variable, and the same holds for $\mathbf{E}_{\pi\otimes\pi}[\tau_{\rm meet}]$. Nevertheless, as we will formalize in Section \ref{sec:rw on DCM}, it is known that under Assumption \ref{degree assumptions} (see \cite{CP21,CCPQ21,BCS19}), w.h.p.,
	
	\begin{itemize}
		\item $G$ is ergodic and $|{\rm supp}(\pi)|\asymp n$, hence there exists a unique stationary distribution, $\pi$, of random walk on $G$,
		\item the invariant distribution $\pi$ is not too concentrated, i.e., $\pi_{\max}=\frac{\log^{O(1)}(n)}{n}$,
		\item the random walk on $G$ mixes fast, i.e., $t_{\rm mix}=\Theta(\log (n))$.
	\end{itemize}
	Hence, the mean field conditions in \eqref{MFcondCoal} are satisfied w.h.p.. To the aim of controlling the expected meeting time, it is worth to introduce the following quantities which will play a key role in our analysis. It is first convenient to define the probability distribution 
	\begin{equation}\label{eq:def-muin}
		\mu_{\rm in}(x)=\mu_{{\rm in},n}(x)\coloneqq \frac{d^-_x}{m}\,,\qquad x\in[n]\,,
	\end{equation}
	that is, the law of a vertex sampled with probability proportional to its in-degree. We will further consider the following functions of the degree sequences
	\begin{equation}\label{eq:def-rho-gamma}
		\begin{split}
			\delta&=\delta_n\coloneqq\frac{m}{n}\,,\qquad\quad\quad\quad\:\:\:\beta=\beta_n\coloneqq\frac1{m}\sum_{x\in[n]}(d_x^-)^2\,, \\
			\rho&=\rho_n\coloneqq\sum_{x\in[n]}\mu_{\rm in}(x)\frac{1}{d^+_x}\,, \quad\quad\gamma=\gamma_n\coloneqq\sum_{x\in[n]}\mu_{\rm in}(x)\,\frac{d^-_x}{d^+_x}\,,\\
		\end{split}
	\end{equation}
	Notice that under Assumption \ref{log degree assumptions} we have that all the four quantities are of order $\Theta(1)$ and, moreover, satisfy the bounds
	\begin{equation}
		\rho\le\frac12\,,\qquad  \gamma\ge 1\,,\qquad  \beta\ge2\gamma\,,\qquad \delta\ge 2\,.
	\end{equation}
	For the sake of intuition, let us provide an interpretation for the quantities in \eqref{eq:def-rho-gamma}. In particular, it is immediate to interpret the first two parameters in a graph theoretical way: $\delta$ is simply the mean degree (equivalently, in- or out-) of the graph, while $\beta$ is the ratio between the second and the first moment of the in-degree distribution. On the other hand, it will be convenient to interpret $\rho$ and $\gamma$ simply as expectations with respect to $\mu_{\rm in}$.
	
	As we will show in the next section, the typical asymptotic behavior of the expected meeting time on a graph $G$ sampled with the above procedure depends on the parameters of the model, $\mathbf{d}=(\mathbf{d}^+,\mathbf{d}^-)$, only through the quantities in \eqref{eq:def-rho-gamma}.

	\section{Main results}\label{suse:main-results}
	We are now ready to state our main results. Throughout this section $\P$ will represent the law of the sequence of graphs sampled, for all $n\in\N$, according to DCM$(\mathbf{d}_n^+,\mathbf{d}^-_n)$, for some prescribed (sequence of) degree sequences $(\mathbf{d}_n^+,\mathbf{d}_n^-)_{n\ge 1}$ satisfying Assumption \ref{log degree assumptions}. Provided that it exists and is unique, let $\pi$ denote the unique stationary distribution of the random walk on $G$. To avoid degeneracy, in what follows we assume that, in case $\pi$ is not well defined, then the same symbol $\pi$ denotes uniform distribution on $[n]$.
	
	In order to state the main result we need to introduce the following quantities which depend on the degree sequences through the functions defined in \eqref{eq:def-rho-gamma}:
	\begin{equation} \label{eq:def-mathfrak-p}
		\begin{split}
			\mathfrak{p}=\mathfrak{p}_n( \mathbf{d}^+, \mathbf{d}^-)\coloneqq
			\frac{1}{\delta}\left(\frac{\gamma-\rho}{1-\rho}+\beta-1 \right)\ge 1\,,
		\end{split}
	\end{equation}
	\begin{equation}\label{eq:def-mathfrak-q}
		\begin{split}
			\mathfrak{q}
			=\mathfrak{q}_n( \mathbf{d}^+, \mathbf{d}^-)\coloneqq\frac{\gamma-\rho}{\gamma-\rho +(\beta-1)(1-\rho)}=
			\left(1+(\beta-1)\frac{1-\rho}{\gamma-\rho} \right)^{-1}\le 1\,,
		\end{split}
	\end{equation}
	and
	\begin{equation} 
		\begin{split}\label{eq:def-mathfrak-r}
			\mathfrak{r}=\mathfrak{r}_n(\mathbf{d}^+,\mathbf{d}^-)\coloneqq\frac{\rho}{\rho -\mathfrak{q}\left(1-\sqrt{1-\rho}\right)}\ge 1\,.
		\end{split}
	\end{equation}
	
	To enhance clarity, we will now offer a heuristic interpretation of the variables $\mathfrak{p}$, $\mathfrak{q}$, and $\mathfrak{r}$, deferring a more comprehensive explanation of their significance to Section \ref{sec:strategy_of_proof}.
	
	As we will show in Section \ref{sec:random-mu}, the first two quantities can be interpreted as expectations in the probability space $\P$ of functions depending on the stationary distribution of the random walk on $G$. In particular, when $n\to\infty$,
	\begin{equation}\label{eq:frak-p-q-emp}
		\mathfrak{p}\approx\E\bigg[n\,\sum_{x\in[n]}\pi(x)^2\bigg]\,,\qquad \mathfrak{q}\approx\E\bigg[\frac{n}{\mathfrak{p}}\sum_{x\in[n]}\pi^2(x)\frac{1}{d_x^+}\bigg]\,.
	\end{equation}
	In words, $\mathfrak{p}$ is the (rescaled) expected sum of the entries of $\pi^2$, while $\mathfrak{q}$ is the expectation of the average inverse-out degree of a vertex sampled with probability proportional to $\pi^2$. As we will show below, the random variables within the parenthesis in \eqref{eq:frak-p-q-emp} concentration around their expectation, that are asymptotically equal to $\mathfrak{p}$ and $\mathfrak{q}$, respectively. As for the quantity $\mathfrak{r}$, we will see in Section \ref{sec:random-mu} that it can be seen as the inverse of the probability that two random walks on an certain infinite Galton-Watson tree never meet, assuming that one walk starts at the root, and the other one at one of its children.
	
	The next theorem, which is the main technical contribution of the paper, shows the weak convergence in probability of the meeting time to an exponential random variable. As a consequence, the random variable $\mathbf{E}_{\pi\otimes\pi}[\tau_{\rm meet}]/n$ converges in probability to a constant, depending only on the degree sequences, which we explicitly characterize.
	\begin{theorem}[{Meeting times on the DCM}]\label{th:main}
		Let $(\mathbf{d}^+,\mathbf{d}^-)$ satisfy Assumption~\ref{log degree assumptions} and $G$ be sampled from DCM$(\mathbf{d}^+,\mathbf{d}^-)$. Then, letting $\tau^{\pi\otimes\pi}_{\mathrm{meet}}$ denote the first meeting time of two independent stationary random walks, it holds
		\begin{equation}\label{eq:meeting time law}
			d_W\left(\frac{\tau^{\pi\otimes\pi}_{\mathrm{meet}}}{\frac12\ \vartheta \times n}, {\rm Exp}(1)\right)\overset{\P}\longrightarrow 0\,,
		\end{equation}
		with 
		\begin{equation}\label{formula}
			\vartheta=\vartheta_n( \mathbf{d}^+, \mathbf{d}^-)\coloneqq\frac{\mathfrak{r}}{\mathfrak{p}}\,,
		\end{equation}
		where $\mathfrak{p}$ and $\mathfrak{r}$ are defined as in \eqref{eq:def-mathfrak-p},  and \eqref{eq:def-mathfrak-r}, respectively. 
	\end{theorem}
	Before discussing in details the above result, we first state its immediate consequences for the voter model and coalescent random walks, based on the discussion of the \emph{mean field conditions} provided in Subsection \ref{subsection:meanfield}.
	\begin{corollary}[{Coalescence and consensus time on the DCM}]\label{coro:meet-coal-cons}
		Let $(\mathbf{d}^+,\mathbf{d}^-)$ satisfy Assumption~\ref{log degree assumptions} and $G$ be sampled from DCM$(\mathbf{d}^+,\mathbf{d}^-)$. Then,
		\begin{itemize}
			\item[(i)] Recalling the definition of $\tau_{\rm coal}$ in \eqref{Tcoal}, 
			\begin{equation}\label{eq:oli-coal-DCM}
				d_W\left(\frac{\tau_{\mathrm{coal}}}{\frac12\ \vartheta \times n},  \sum_{k\ge 2} Z_k\right)\overset{\P}\longrightarrow 0\,.
			\end{equation}
			where the $Z_k$'s are independent random variables with  law as in \eqref{eq:def-Z}.
			\item[(ii)] Consider the voter model on $G$ with $\eta_0=\bigotimes_{x\in[n]}{\rm Bern}(u)$ and $u\in(0,1)$. 
			Using the definitions in Theorem \ref{th:mean-cons},
			\begin{equation}\label{eq:oli-cons-DCM}
				d_W\left(\frac{\tau_{\mathrm{cons}}}{\frac12\ \vartheta \times n},  \sum_{k\ge K} Z_k\right)\overset{\P}\longrightarrow 0\,,
			\end{equation}
			where $K$ is independent from the collection $(Z_k)_{k\ge 2}$ and is defined as in \eqref{eq:def-K-oli}.
		\end{itemize}
	\end{corollary}
	
	\begin{remark}
		It is worth to remark that the result in Theorem \ref{th:mean-cons} can be extend to  an arbitrary number of opinions, as soon as the initial distribution can be expressed a product of i.i.d. random variables. As a consequence, the convergence in Corollary \ref{coro:meet-coal-cons}(ii) extend as well. Although, we prefer to present our results on the simpler case of the two-opinion model, in order to not deviate focus from the general structure and the novelties of our results. 
	\end{remark}
	
	\begin{remark}
		It is known that, beyond the distribution of the coalescence and the consensus time, the precise knowledge of the expected meeting time can be used to determine the scaling limit of certain real valued processes. In particular, in \cite{CCC16} the authors show that rescaling time by $\mathbf{E}_{\pi\otimes\pi}[\tau_{\mathrm{meet}}]$, the weighted average of vertices having opinion $1$ converges to the celebrated Wright-Fisher diffusion (see, e.g., \cite{Lig85}) in the Skorohod topology. Moreover, in \cite{Che13,Che18,Che23,CC18}, similar results are obtained for some modification of the voter model related to the theory of evolutionary games.
	\end{remark}
	\subsection{Discussion and examples} \label{sec:examples}
	In Theorem \ref{th:main} and Corollary \ref{coro:meet-coal-cons} we provide a complete characterization of the distribution of meeting, coalescence and consensus time on a typical random graph as a function of a single quantity, $\vartheta$.
	
	Let us now discuss the property of the function $\vartheta$ in \eqref{formula}, and its dependence on the degree sequences. Being $\vartheta$ a function of only four parameters (those in \eqref{eq:def-rho-gamma}), it is possible to understand how changes in the degree sequences affect $\vartheta$, and thus the distribution of our stopping times. For the sake of comparison, we can for instance fix  $\delta$, namely, the total number of edges, and see how do the other parameters in \eqref{eq:def-rho-gamma} influence the function $\vartheta$. In particular, we are interested in answering the following kind of questions, which are of evident interests for applied network science:
	\begin{itemize}
		\item If the out-degrees (resp. in-degrees) are constant, increasing the variability of the in-degrees (resp. out-degrees) implies a speed-up or a slow-down in the consensus time? Which is of the range $\vartheta$ in these cases?
		\item What is the effect of positive/negative correlation between in- and out-degrees of the vertices on the consensus time? E.g., is consensus reached faster on an Eulerian digraph, or in one in which $d_x^-/d_x^+$ is typically far from $1$?
		\item In the Eulerian setup, does the variability of the degrees speed-up the consensus time?
		\item If some features of the degree sequences are constrained and one is free to choose the exact degree sequence under such constraints, which are the guiding principles to minimize/maximize the consensus time?
	\end{itemize}
	In the rest of this section we will answer such questions and provide some simplified formulas for $\vartheta$ in some special cases.
	
	\begin{figure}
		\centering
		\includegraphics[width=.75\textwidth]{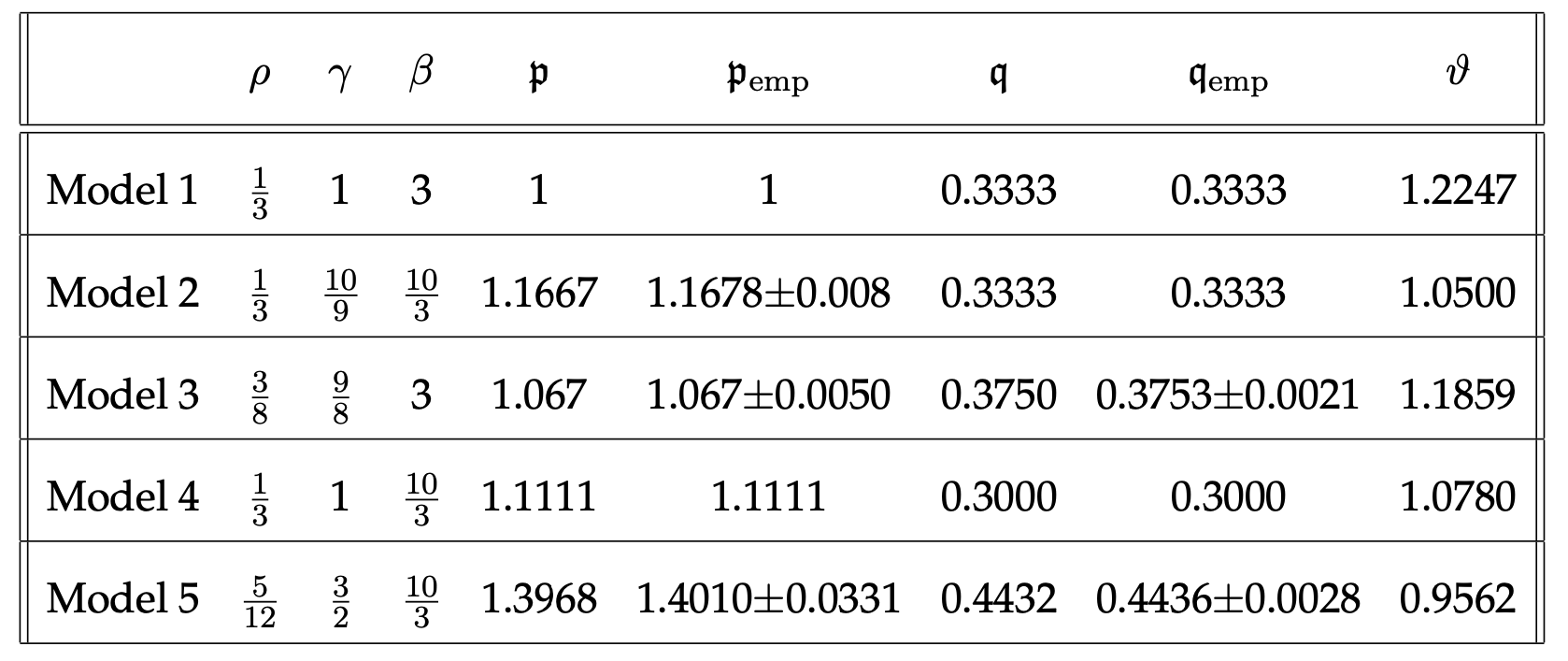}
		\caption{We fix $n=1000$ and consider a sample of $G_n$ with degree distributions of the type $(\mathbf{d}^+,\mathbf{d}^-)$ where $d_x^-=a$ and $d_x^+=b$ for all $x\le 500$ and $d_x^-=c$ and $d_x^+=d$ for all $x\ge 500$. To simplify the reading, we write [$a$,$b$]-[$c$,$d$] to refer to such models. For the sake of comparison,  for all the models, we let $a+c=b+d=6$, i.e., in all the cases there are $3n$ edges in total, hence $\delta=3$. More precisely:\\
			Model 1: [3,3]-[3,3], i.e., the regular case.\\
			Model 2: [3,4]-[3,2], i.e., the out-regular case.\\
			Model 3: [4,3]-[2,3], i.e., the in-regular case.\\
			Model 4: [2,2]-[4,4], i.e., the Eulerian case.\\
			Model 5: [2,4]-[4,2], i.e., the ``alternate'' model discussed in Subsection \ref{suse:alternate}.
			The table reports the value of the quantities  $\beta$, $\rho$ and $\gamma$ in \eqref{eq:def-rho-gamma}, as well as the values of $\mathfrak{p}$, $\mathfrak{q}$ and $\vartheta$ (as defined in \eqref{eq:def-mathfrak-p}, \eqref{eq:def-mathfrak-q} and \eqref{formula}, respectively) associated to the 5 degree sequences. As mentioned in Section \ref{suse:main-results}, the quantities $\mathfrak{q}$ and $\mathfrak{p}$ admit an interpretation in terms of expectation of the random variables (with respect to the generation of the graph) appearing in \eqref{eq:frak-p-q-emp}. Therefore, we report here also the average of such random variables with respect to 100 random generations of the graph. In these cases, the $\pm$ error stands for the empirical standard deviation. To simplify the comparison, we express both the theoretical and the empirical value of $\mathfrak{p}$ and $\mathfrak{q}$ to the first 4 digits. }
	\end{figure}

	\subsubsection{The regular case} The easiest model is the one in which every vertex as in- and out-degree equal to some constant $d\ge 2$. First notice that in the $d$-regular case the following simplifications of the parameters in \eqref{eq:def-rho-gamma} hold
	\begin{equation*}
		\rho=\frac{1}{d},\quad \delta = d,\quad \gamma = 1,\quad \beta =d\ ,
	\end{equation*}
	so that 
	\begin{equation}\label{eq:p-q-reg}
		\mathfrak{p}= 1\quad \text{and} \quad \mathfrak{q}=\frac{1}{d}\ .
	\end{equation}
	Hence, it is immediate to deduce the following result.
	\begin{corollary}\label{coro:regular}
		Fix $d\ge 2$ and let $\vartheta(d)$ denote the quantity in \eqref{formula} when $d_x^+=d_x^-=d$ for all $x\in[n]$.
		Then
		\begin{equation}\label{eq:theta-reg}
			\vartheta(d)= \sqrt{\frac{d}{d-1}}\in\big(1\,,\sqrt{2}\,\big]\,.
		\end{equation}
	\end{corollary}
	
	\begin{remark}
		Notice that, in the $d$-regular \emph{undirected} configuration model, one has (see, e.g., \cite{Che21,ABHHQ22})
		\begin{equation}\label{eq:theta-und-reg}
			\frac{\mathbf{E}_{\pi\otimes\pi}[\tau_{\rm meet}]}{\frac{1}{2}\ \frac{d-1}{d-2}\times  n }\overset{\P}{\longrightarrow}1\,,\qquad d\ge 3\,.
		\end{equation}
		Hence, it is possible to compare the expected meeting time for the undirected $d$-regular case with that of the directed one. As an effect of the directionality, the random walks result to meet faster. Indeed,
		\begin{equation}
			\sqrt{\frac{d}{d-1}}< \frac{d-1}{d-2}\,, \qquad d\ge 3\,.
		\end{equation}
		One may argue that the above comparison is improper, since the total number of neighbors of a vertex in a random $d$-regular directed graph is actually $2d$ ($d$ in-neighbors and $d$-out-neighbors). Nevertheless, it is also true that
		\begin{equation}
			\sqrt{\frac{d}{d-1}}< \frac{2d-1}{2d-2}\,, \qquad d\ge 2\,.
		\end{equation}
		Therefore, the speed-up experienced in the directed setting takes place whatever term of comparison we choose.
	\end{remark}
	\subsubsection{The out-regular case} Another model that is natural to investigate is the one in which all the vertices share the same out-degree, $d$. In this case the following holds.
	\begin{corollary}\label{coro:theta-out-reg}
		Let $d\ge 2$ and $(\mathbf{d}^+,\mathbf{d}^-)$ satisfying Assumption \ref{log degree assumptions} and such that, for all $n\in\N$, 
		$ d_x^+=d$ for all $x$. Call $\vartheta(d,\mathbf{d}^-)$ the quantity in \eqref{formula} in this case. Then
		\begin{equation}
			\vartheta(d,\mathbf{d}^-)=\frac{\sqrt{d(d-1)}}{\beta-1}\in\bigg(0\,,\sqrt{\frac{d}{d-1}}\,\bigg]\,.
		\end{equation}
	\end{corollary}
	\begin{proof}
		In the out-regular case we have
		\begin{equation}
			\delta=\frac1\rho=d\,,\qquad \beta=d\gamma\,.
		\end{equation}
		After some algebraic manipulation we get
		\begin{equation}\label{eq:p-in-reg}
			\mathfrak{p}=\frac{\beta-1}{d-1}\,,
			\qquad
			\mathfrak{q}=\frac1d\,,
			\qquad
			\mathfrak{r}=\sqrt{\frac{d}{d-1}}\,,
		\end{equation}
		from which the result follows.
	\end{proof}
	It is worth noting that, as soon as $d\ge 3$ it is possible to choose $\mathbf{d}^-$ so to have $\beta$ arbitrarily large. This implies that, for a typical random out-regular directed graph, the larger is the variance in the in-degree distribution, the smaller is the value of the expected meeting time. In other words, among the out-regular random digraphs, the regular is the one in which the meeting time is the largest.
	\subsubsection{The in-regular case}
	We now show that, for a typical in-regular random  directed graph the situation looks much more similar to the regular case than to the out-regular one.
	\begin{corollary}\label{coro:theta-in-reg}
		Let $(\mathbf{d}^+,\mathbf{d}^-)$ satisfying Assumption \ref{log degree assumptions} and such that, for all $n\in\N$,
		$ d_x^-=d$ for all $x$. Call $\vartheta(\mathbf{d}^+,d)$ the quantity in \eqref{formula} in this case. Then, fixed $d\ge 3$, it holds
		\begin{equation}\label{eq:theta-in-reg}
			\vartheta(\mathbf{d}^+,d)=\frac{d\,\sqrt{1-\rho}}{d-1}\in\bigg(\frac{1}{\sqrt{2}}\,\frac{d}{d-1}\,,\sqrt{\frac{d}{d-1}}\,\bigg]\,,
		\end{equation}
			while
			\begin{equation}
				\vartheta(\mathbf{d}^+,2) = \sqrt{2}\,.
		\end{equation}
	\end{corollary}
	
	\begin{proof}

		Since $\rho\ge \frac1d$ by Jensen inequality, we have
		$$\vartheta(\mathbf{d}^+,d)\le \sqrt{\frac{d}{d-1}}.$$
		On the other hand, for $d\ge 3$, 
		$$\vartheta(\mathbf{d}^+,d)> \frac1{\sqrt{2}}\frac{d}{d-1} $$
		since $\rho\le \frac12$. Moreover, such a lower bound is sharp. To see this, fix $C=d_{\rm max}^+$ and consider the degree sequence in which the first $(1-\varepsilon) n$ vertices have  $d^+=2$,
		and the last $\varepsilon n$ vertices have degree $C$.
		Then, $\varepsilon$ is fixed by the equation
		$$2(1-\varepsilon)+\varepsilon C=d$$
		from which it follows that
		\begin{equation}
			\varepsilon=\frac{d-2}{C-2}.
		\end{equation}
		Then
		\begin{equation}
			\rho=(1-\varepsilon)\frac12 +\frac{\varepsilon}{C}=\frac{C^2-dC+2(d-2)}{2C^2-4C}.
		\end{equation}
		In conclusion, for any $d\ge 3$ and $\eta>0$ one can get $C$ large enough so to have $\rho>\frac{1}{2}-\eta$.
	\end{proof}
	Roughly speaking, the previous result shows that, for a fixed value of the average degree, the meeting time on a typical out-regular graphs (regardless of the choice of the in-degrees) is smaller than in a regular graph with the same number of edges, regardless of the out-degrees. Nevertheless, contrarily to the out-regular case, the enhancement obtained by taking the out-degrees non-regular is bounded and the value of the constant $\vartheta$ cannot go below the threshold $1/\sqrt{2}$.
	\subsubsection{The Eulerian case}
	Random Eulerian graphs are a particularly relevant subclass of the DCM, corresponding to the case in which the in-degree and the out-degree of each vertex coincide. Indeed, it is the only class of directed graphs with the following property: as soon as the graph is strongly connected, $\pi$ is proportional to the degree of the vertices, exactly as in an undirected graph. For this reason, it is the model that it is easier to compare to the classical undirected Configuration Model. In this setting, it is worth to realize that the quantity $\beta/\delta$ represents the ratio between the second moment and the squared first moment of the degree sequence, and it is therefore a signature of variability of the degrees. As the next corollary shows, the value of $\vartheta$ in this case depends essentially just on such a ratio, and higher variability implies a faster meeting time.
	\begin{corollary}\label{coro:theta-eulerian}
		Let $(\mathbf{d}^+,\mathbf{d}^-)$ be such that, for all $n\in\N$,
		$ d_x^+=d_x^-$ for all $x$. Call $\vartheta(\mathbf{d})$ the quantity in \eqref{formula} in this case. Then
		\begin{equation}\label{eq:theta-eulerian}
			\vartheta(\mathbf{d})=\left({\frac{\beta}{\delta}-1+\sqrt{1-\frac{1}{\delta}}}\right)^{-1}\in\bigg(0,\sqrt{\frac{\delta}{\delta-1}}\,\bigg]\,.
		\end{equation}
	\end{corollary}
	\begin{proof}
		In this case we have
		\begin{equation}
			\rho=\frac1\delta\,,\qquad \gamma=1\,,
		\end{equation}
		from which we get
		\begin{equation}\label{eq:p-out-reg}
			\mathfrak{p}=\frac\beta\delta\,,
			\qquad
			\mathfrak{q}=\frac{1}{\beta}\,,
			\qquad
			\mathfrak{r}=\frac{\beta/\delta}{\beta/\delta-1+\sqrt{1-1/\delta}}\,,
		\end{equation}
		thus the validity of \eqref{eq:theta-eulerian}.
	\end{proof}
	It is worth to remark that, also in this case, the random regular graph is the one with the slowest meeting time among all the random Eulerian graphs with the same number of edges.
	
	\subsubsection{The ``alternate'' case}\label{suse:alternate}
	
	\begin{figure}[ht]
		\centering
		\includegraphics[width=.45\textwidth]{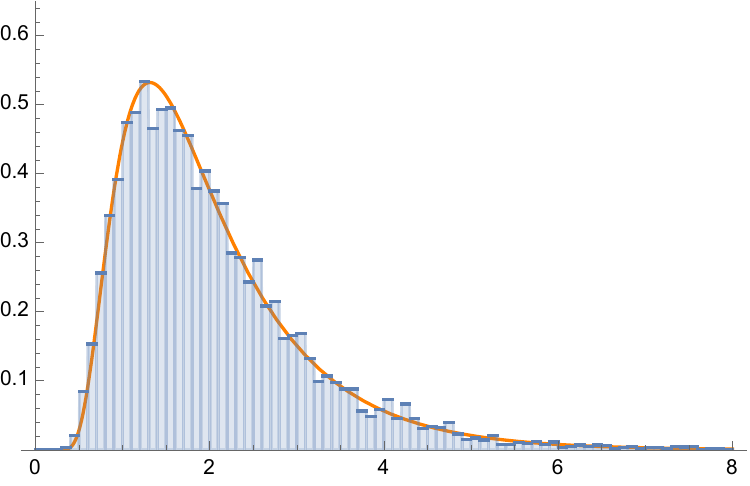}\qquad\qquad
		\includegraphics[width=.45\textwidth]{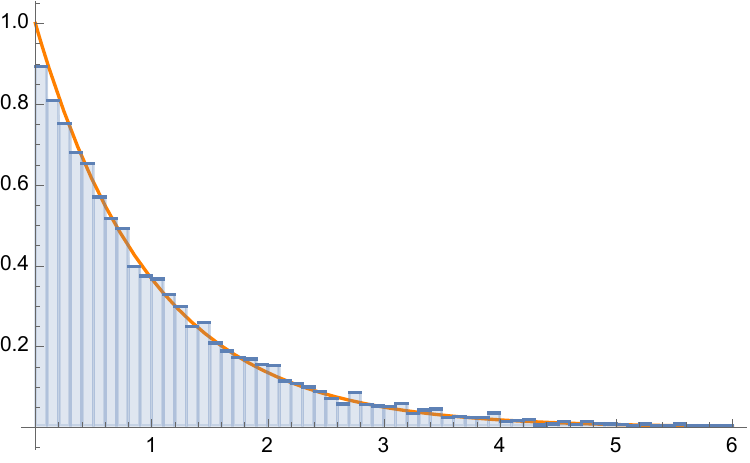}\\ \vspace{0.6cm}
		\includegraphics[width=.45\textwidth]{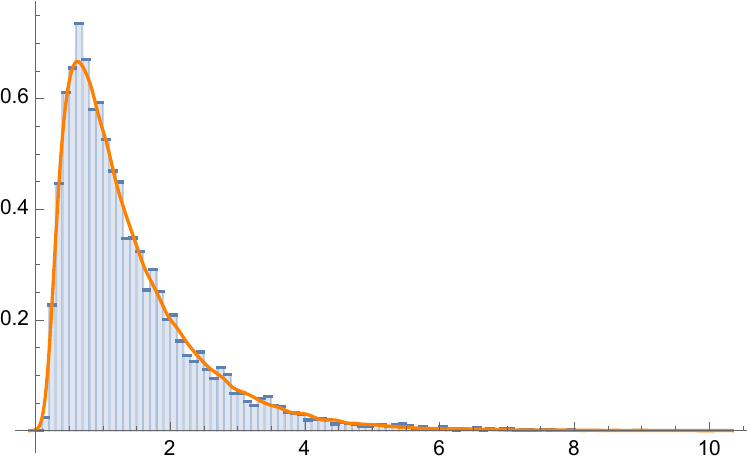}\qquad\qquad
		\includegraphics[width=.45\textwidth]{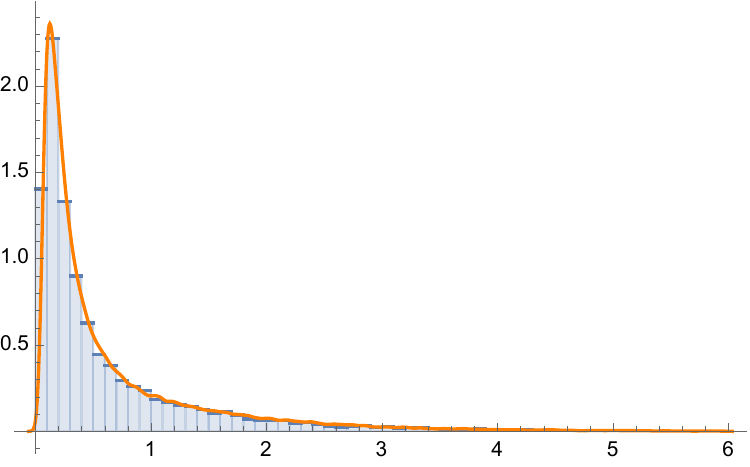}
		\caption{The pictures refer to the statistics of $\tau_{\rm coal}$ and $\tau_{\rm meet}$ for $10^5$ runs of the  dynamics on the same (quenched) realization of the graph, with $n=1000$,  for the ``alternate'' model described in Subsection \ref{suse:alternate}, with $a=2$ and $b=4$. \\ Top left: 
			in blue, the discretized empirical density of $\frac{2\tau_{\rm coal}}{n\vartheta}$.
			In orange, the numerical approximation of the PDF of the infinite sum in \eqref{eq:oli-coal-DCM}.\\
			Top right:  
			in blue, the discretized empirical density of $\frac{2\tau_{\rm meet}}{n \vartheta}$ where the initial distribution is $\pi\otimes\pi$.
			In orange, the PDF of an exponential distribution of mean 1.
			\\ Bottom left: 
			in blue, the discretized empirical density of $\frac{2\tau_{\rm cons}}{n\vartheta}$ with initial density $u=1/2$.
			In orange, the numerical approximation of the PDF infinite sum in \eqref{eq:oli-cons-DCM}.\\
			Bottom right:  
			same as in the bottom left picture, but with $u=1/10$.}
		\label{fig:simulations}
	\end{figure}
	A first toy model to check for the effect of correlations between in- and out-degrees is what we call \emph{alternate} model, that is, the vertices are divided in two even groups: in one group the degrees are $d^+=a$ and $d^-=b$, and the vice versa for the other group. Indeed, in this model it is natural to see that the farther $a/b$ is from $1$, the more anti-correlated are the degrees.
	\begin{corollary}
		Let $b\ge a\ge 2$ and  $(\mathbf{d}^+,\mathbf{d}^-)$ be such that for all $n\in2\N$
		$ d_x^+=a$ and $d_x^-=b$ for all $x\in[n/2]$, while   $ d_x^+=b$ and $d_x^-=a$ for  $x\in[n]\setminus[n/2]$. Let $d=\frac{a+b}{2}$ and call $\vartheta(d,a)$ the quantity in \eqref{formula} in this case. Then, for all $n\in2\N$ and any fixed $d\ge 2$ the function 
		$$\vartheta(d,\cdot):\{2,\dots,d\}\to \bigg(0\,,\sqrt{\frac{d}{d-1}}\,\bigg]$$
		is monotone increasing.
	\end{corollary}
	\begin{proof}
		Rewrite
		\begin{equation}
			\delta=d\, ,\qquad  \beta=\frac{1}{d}\frac{a^2+b^2}{2}\,,\qquad \rho=\frac{1}{d}\frac{a^2+b^2}{2ab}\,,\qquad\gamma=\frac{1}{d}\frac{a^3+b^3}{2ab}\,,
		\end{equation}
		with $b=2d-a$. It is immediate ti check that, fixed $d$, $\beta$, $\rho$ and $\gamma$ are decreasing in $\{2,\dots,d\}$. Similarly, 
		\begin{equation}\label{eq:frakp-alternate}
			\mathfrak{p}=\frac{1}{d}\frac{\gamma-\rho}{1-\rho}+\frac{\beta-1}{d}\,,
		\end{equation}
		is decreasing. On the other hand, it is possible to check that $\frac{(\beta-1)(1-\rho)}{\gamma-\rho}$ is increasing in $a$, writing
		\begin{equation}
			\mathfrak{q}=\left(1+\frac{(\beta-1)(1-\rho)}{\gamma-\rho} \right)^{-1}\,,
		\end{equation}
		we deduce that $\mathfrak{q}$ is monotone decreasing in $a$. Similarly, writing 
		\begin{equation}\label{eq:frakr-alternate}
			\mathfrak{r}=\left(1-\frac{\mathfrak{q}(1-\sqrt{1-\rho})}{\rho} \right)^{-1}\,,
		\end{equation}
		by taking derivatives one can conclude that $\mathfrak{r}$ is monotone decreasing in $a$, hence the desired result follows from \eqref{eq:frakp-alternate}, \eqref{eq:frakr-alternate} and the definition of $\vartheta$ in \eqref{formula}.
	\end{proof}
	\subsubsection{Some further considerations about network design}\label{suse:takehome}
	Configuration models are particularly well studied in network science, due to the fact that they are \emph{maximum likelihood} ensembles under the degree constraints. To see this in practice, suppose that a network designer is asked to design a large network in which the expected consensus time is in a certain target window. If the designer has control only on the degree sequence but no prior on the way in which the network will eventually be constructed, then the (Bayesian) designer will try to optimize the degrees so that a typical DCM with those degrees has an expected consensus time as close a possible to the given target.
	The reader should be convinced at this point that our results provide to the designer the opportunity to solve such a task. Going beyond the sub-models investigated above, we conclude this section with a take-home message concerning the effect of degree variability and degree correlations in affecting our stopping times.
	
	On the one hand, it is clear from the special cases investigated above that the variability in the in-degree sequence plays a crucial role in determining the value of $\vartheta$. On the other hand, with the general result of Theorem \ref{th:main} at hand, we propose as a measure of correlation between  in- and out-degrees the quantity
	\begin{equation} \label{eq:correlation-measure-in-out-deg}
		\alpha\coloneqq\frac{\gamma-\rho}{1-\rho}\ge 1\,.
	\end{equation}
	To see that the latter can be thought of as a measure of correlation it is helpful to  note that it equals one in the Eulerian case and diverges when $\frac1n\sum_x\frac{d_x^-}{d_x^+}$ diverges. It is worth to introduce the increasing function
	\begin{equation} \label{eq:increasing-fun-network-design}
		\epsilon:[0,1/2]\to [0.5,0.59)\,,\qquad x\mapsto\frac{1-\sqrt{1-x}}{x}\,,
	\end{equation}
	Indeed, with this notation we get
	\begin{equation}
		\mathfrak{p}=\frac{\alpha+\beta-1}{\delta}\,,\qquad \mathfrak{q}=\left(1+\frac{\beta-1}{\alpha} \right)^{-1}\,,\qquad
		\mathfrak{r}=\big(1+ \epsilon(\rho)\,\mathfrak{q} \big)^{-1}\,,
	\end{equation}
	which result in
	\begin{equation}\label{eq:formula-theta-design}
		\vartheta=\frac{\delta}{\big(1-\epsilon(\rho))\alpha+\beta-1}\,.
	\end{equation}
	The advantage of \eqref{eq:formula-theta-design} is that it immediately shows the effect of $\alpha$ and $\beta$ on the function $\vartheta$. Indeed, fixed the average degree $\delta$, and realizing that the effect of the $\epsilon$-term is somehow negligible due to its small range, we see that $\vartheta$ is essentially inverse proportional to $\alpha$ and $\beta$. In other words, the larger  the variability of the in-degree distribution, and the more anti-correlated are the out- and in-degree distributions, the faster is the voter model in reaching consensus. 
	It is worth to point out that in \cite{SAR08} the authors predict, by means of some non-rigorous computation, that in the case of an \emph{undirected} random graph from the configuration model one should have $\vartheta\approx \delta/\beta$. Notice that in the undirected setting (by Eulerianity) one has $\alpha=1$ and therefore the formula in \eqref{eq:formula-theta-design} coincides with the prediction in \cite{SAR08} up to the quantity $\epsilon(\rho)$ at the denominator. Notice also that their prediction actually fails to capture the exact constant in the case of regular undirected graphs, in which $\vartheta$ is given by \eqref{eq:theta-und-reg} while $\delta/\beta=1$ in that case.

	\section{Geometry of the DCM}\label{sec:geometryDCM}
	In this section we provide the prerequisite knowledge on the typical feature of a graph sampled from the DCM and of the random walk on it. In particular in Section \ref{sec: Local structures} we discuss the classical Breadth First (BF) construction of the graph and its coupling with the construction of a Galton-Watson tree. In Section \ref{sec:rw on DCM} we introduce the discrete-time random walk on a typical realization of $G$, and recall some recent result about its mixing time and the shape of its stationary distribution.
	\subsection{Locally tree like vertices} \label{sec: Local structures}
	Recall from Section \ref{sec: Setting} the sequential construction used for generating the environment $\omega$ with $n$ vertices and degree sequences $\mathbf{d}=\mathbf{d}_n=(\mathbf{d}^-,\mathbf{d}^+)$. In the following section we propose a well-known coupling describing the locally-tree-like structure of the (sparse) directed configuration model. \\
	For any fixed $v\in[n]$ and any $h=h_n>0$, define $\mathcal{B}^+_v(h)$, the $h$-out-neighborhood of vertex $v$, to be the set of paths starting from $v$ of length at most $h$; where a path is a sequence of directed edges $(e_1f_1,\dots,e_\ell f_\ell)$, $\ell\leq h$, such that $v_{f_i}=v_{e_{i+1}}$ for all $i=1,\dots,\ell-1$, and $v_f$ (resp. $v_e$) is the vertex incident to the head $f\in E^-$ (resp. tail $e\in E^+$). In order to generate $\mathcal{B}^+_x(h)$ we use the breadth-first procedure (BF) starting from $v$ as priority rule, iterating the following steps:
	\begin{enumerate}
		\item \label{item: 1}pick the first available unmatched tail $e\in E^+$ according to BF starting from $v$;
		\item \label{item: 2}pick uniformly at random an \emph{unmatched} head $f\in E^-$;
		\item \label{item: 3} draw the resulting directed edge $ef$. Continue until the graph distance from an unmatched tail in Item \ref{item: 1} to $v$ exceeds $h$.
	\end{enumerate}
	We want to compare the exploration process of a neighborhood of $G$ with an exploration process of a marked Galton-Watson tree. To this aim, for any fixed $v\in[n]$, let us define a marked (out-directed) random tree $\mathcal{T}^+_v$ rooted at $v$ as follows: the root is assigned mark $v$, and all other vertices an independent mark $x\in[n]$ with probability $\frac{d_x^-}{m}$. Each vertex with mark $x\in[n]$ has $d_x^+$ children. Note that $\mathcal{T}^+_v$ is obtained by gluing together $d_v^+$ independent Galton-Watson trees with offspring distribution
	\begin{equation}
		\label{eq: size-biased offspring distribution}
		\lambda^{\rm{biased}}_{n}(k)= \lambda^{\rm{biased}}(k) \coloneqq \sum_{x\in[n]} \frac{d_x^-}{m} \ind(d^+_x = k)\ , \qquad k\in\N\ .
	\end{equation}
	Let $\mathcal{T}^+_v (h)$ be a subtree of $\mathcal{T}^+_v$ given by its truncation up to generation $h$. We give a classical description of the coupling between $\mathcal{B}^+_v(h)$ and $\mathcal{T}^+_v (h)$ that can also be found e.g. in \cite[Sec. 2.2]{CCPQ21}.\\
	Consider the steps \ref{item: 1}-\ref{item: 3} for $\mathcal{B}^+_v(h)$, and change Item \ref{item: 2} into
	\begin{enumerate}
		\item[(2')] Pick uniformly at random a head $f\in E^-$ among \emph{all possible} ones, rejecting the proposal if it was already matched, and resampling the head.
	\end{enumerate}
	The subtree $\mathcal{T}^+_v (h)$ can be generated by iterating essentially the same steps with some minor differences. In item (2') we never reject the proposal. The head chosen in item (2') belongs to $E^-_x$ for some $x\in[n]$. In item (3) we add a new leaf and provide the mark $x$ to it. Note that this implies that the new leaf will have $d^+_x$ many children.
	\begin{lemma}\label{lemma:LTL-structure}
		Assume the degree sequence satisfy Assumption \ref{log degree assumptions}. Let $v\in[n]$ and let $\hat{\P}$ be the law of the coupling between $\mathcal{B}^+_v(h)$ and $\mathcal{T}^+_v (h)$ for any $h>0$. It holds that
		\begin{equation*}
			\hat{\P}\left(\mathcal{B}^+_v(\hbar) \neq  \mathcal{T}^+_v (\hbar)\right) =o(1)\,, 
		\end{equation*}
		where
		\begin{equation}\label{eq:def-h-bar}
			\hbar =\hbar_n\coloneqq \frac{\log(n)}{5\,\log(d_{\rm max}^+)}\,.
		\end{equation}
	\end{lemma}
	\begin{proof}
		The coupling fails only in one of the following two cases: either the sampled head $f\in E^-$ coincides with one of the heads already used in a previous sample (the ``rejection'' condition of $\mathcal{B}^+_v(h)$), or $f$ is not sampled yet, but it is incident to a vertex already present in the tree. Let $\tau$ be the first time such that an uniform random choice among all heads gives $f\in E^-_x$, for some mark $x$ which is already present in the tree. By construction, the out-neighborhood and the tree coincide up to time $\tau$, and it holds
		\begin{equation*}
			\hat{\P}(\tau = t) \leq \frac{t\, d_{\rm{max}}^-}{m-t}, \quad t\ge 0\ .
		\end{equation*}
		Therefore
		\begin{equation*}
			\hat{\P}(\tau \leq t) \leq \frac{t^2\, d^-_{\rm{max}}}{m-t}\ .
		\end{equation*}
		Note that, a.s., $t=(d^+_{\rm{max}})^{\hbar+1}$ steps are sufficient to explore the whole $\mathcal{B}^+_v(\hbar)$. It follows that 
		\begin{equation}
			\label{eq: cupling failiure tree}
			\hat{\P}\left(\mathcal{B}^+_v(\hbar) \neq  \mathcal{T}^+_v (h)\right) \leq  \hat{\P}(\tau \leq (d^+_{\rm max})^{h+1})\leq \frac{(d^+_{\rm{max}})^{2\hbar+2}\, d^-_{\rm max}}{m-(d^+_{\rm{max}})^{\hbar+1}} = o\left(\frac{(d^+_{\rm max})^{2\hbar}}{\sqrt{n}}\right)\ ,
		\end{equation}
		and the conclusion follows by the definition of $\hbar$.
	\end{proof}
	In what follows, we will to refer to the following corollary of Lemma \ref{lemma:LTL-structure}.
	\begin{corollary}[LTL vertices]\label{coro:LTL}
		Assume the degree sequence satisfies Assumption \ref{log degree assumptions}. Let
		\begin{equation} \label{eq:def Vstar}
			V_\star=V_{\star,n}\coloneqq\{v\in [n]\mid \cB_v^+(\hslash)\text{ is a tree}\}, 
		\end{equation}
		where $\hbar$ is defined as in \eqref{eq:def-h-bar}, then
		\begin{equation}
			\frac{|V_\star|}{n}\overset{\P}{\longrightarrow}1\,.
		\end{equation}
	\end{corollary}
	
	\subsection{(Discrete-time) Random walk on the DCM} \label{sec:rw on DCM}
	For a given realization of the \emph{environment}, $\omega$, we will consider the discrete time simple random walk $(X_t)_{t\ge 0}$ described by the transition matrix
	\begin{equation}\label{DTRW}
		P(x,y)=\frac{|\{e\in E^+_x\mid \omega(e)\in E^-_y\}| }{d_x^+}\ .
	\end{equation}
	Recall here that given a tail $e\in E^+$, $\omega(e)\in E^-$ represents the head at which the tail $e$ have been matched in the uniform matching generated by the environment $\omega$ .
	The latter describes a Markov chain on the vertex set, in which the walker chooses one of the out-going edges of the vertex it is currently visiting and moves to the vertex attached to the matched head.

	Despite $\pi$ being random, it is possible to show a very precise result for the mixing time of the random walk on the DCM under Assumption \ref{log degree assumptions}. In particular, in \cite{BCS18} the authors show that the mixing time is logarithmic and the total-variation distance decays abruptly to zero at a precise spot on the time line, an instance of the so-called \emph{cutoff at the entropic time} (see also \cite{BCS19,CCPQ21}). More precisely, their result reads as follows.
	\begin{theorem}[Mixing time]\label{th:cutoff}
		Assume the degree sequence satisfy Assumption \ref{log degree assumptions}, and let
		\begin{equation*}
			H=H_n\coloneqq \sum_{x\in[n]}\frac{d^-_x}{m}\log(d_x^+)\quad \text{and} \quad t_{\rm ent}=t_{{\rm ent},n} \coloneqq \frac{\log(n)}{H}\ .
		\end{equation*}
		For all $\varepsilon\neq1$, it holds that
		\begin{equation}
			\max_{x\in[n]}| \|P^{\lfloor\varepsilon t_{\rm ent}\rfloor}(x,\,\cdot) - \pi\|_{\rm TV} - \mathds{1}_{\varepsilon<1} | \overset{\P}{\longrightarrow} 0 \ .
		\end{equation}
	\end{theorem}
	Beyond the mixing time result, we will need to establish a control over the maximal and minimal values of $\pi$ within its support as part of our proof. In \cite{CQ20,CP20,CCPQ21} the authors deal exactly with this problem under different assumptions on the degree sequence. We collect  their results in the following theorem.
	\begin{theorem}[Extremal values of $\pi$]\label{th:extremal-pi}
		Assume the degree sequence satisfy Assumption \ref{log degree assumptions}. Then there exists some $\varepsilon>0$ and $C> 1$ such that
		\begin{equation}
			\frac{\max_{x\in[n]}\pi(x)}{n^{{-1+\varepsilon}}}\overset{\P}{\longrightarrow}0\,,\qquad
			\frac{n^{-C}}{\min_{x\in{\rm supp}(\pi)}\pi(x)}\overset{\P}{\longrightarrow}0\,.
		\end{equation}
	\end{theorem}
	
	\section{Strategy of proof}\label{sec:strategy_of_proof}
	As mentioned in the Introduction, the core contribution of this work lies in showing that the expected meeting time of two independent random walks evolving on a typical random directed graph is well concentrated around a deterministic quantity (which we provide explicitly)  depending only on the parameters of the model, that is, on the degree sequences. To this aim, we follow the strategy depicted by \cite{CFR09}, which consists in using the so-called \emph{First Visit Time Lemma (FVTL)}. Essentially, the FVTL states that, given a mixing chain and a target state $\partial$, as soon as the chain mixes sufficiently fast compared to the stationary mass of $\partial$, $\pi(\partial)$, then the hitting time of the target is well approximated (uniformly in time) by a geometric random variable whose parameter depends only on $\pi(\partial)$ and on the ``local geometry'' around $\partial$. The FVTL  has been introduced by Cooper and Frieze in a series of works (see, among others, \cite{CF04,CF05,CF07,CF08}) in which the authors characterize the first order of the cover time of random walks in various random graphs. Recently, a simplified probabilistic proof has been provided by \cite{MQS21}. For convenience, we refer here to the version in \cite[Appendix A]{QS23}, in which the authors adopt the same notation. Nevertheless, the general idea that hitting times are essentially exponential under some \emph{fast mixing} condition has a long history, tracing back to the work of Aldous and Brown, see, e.g., \cite{AB92,B99,AldPCH,FL14}.
	\begin{theorem}[First Visit Time Lemma]\label{fvtl}
		For every $N\in \N$ consider a discrete time ergodic Markov chain $(X_t)_{t\in\N}$ on $[N]$ with transition matrix $Q=Q_N$ and stationary distribution $\pi=\pi_N$. 
		Consider further a target state $\partial\in[N]$. Call 
		\begin{equation}
			t_{\rm mix}= t_{{\rm mix},N}\coloneqq \inf\left\{t\ge 0\mid  \max_{x\in[N]}\|Q^t(x,\cdot)-\pi \|_{\rm TV}\le \frac1{2e}\right\},
		\end{equation}
		and
		\begin{equation} \label{eq:mixing_sequence_FVTL}
			\mathfrak{T}=\mathfrak{T}_N\coloneqq  t_{\rm mix} \times \log\left(\bigg(\min_{x\in{\rm supp}(\pi)}\pi(x)\bigg)^{-1}\right),
		\end{equation}
		and assume that
		\begin{equation}\label{eq:fvtl-hp-tmixsmall-vs-pi-delta}
			\lim_{N\to\infty} \pi(\partial)\,{\mathfrak{T}}=0\, .
		\end{equation}
		Fix any $T=T_N$ such that
		\begin{equation}\label{eq:fvtl-def-T}
			T \ge 2\ \mathfrak{T}\ ,\qquad  \limsup_{N\to \infty}\pi(\partial)\ T=0\,,
		\end{equation}
		and call
		\begin{equation}\label{eq:fvtl-def-R}
			R_T(\partial)=\sum_{t=0}^{T}Q^t(\partial,\partial)\,.
		\end{equation}
		Then, there exists a sequence $\lambda=\lambda_N$ such that
		\begin{equation}\label{eq:fvtl-convergence-lambda}
			\lim_{N\to\infty} \frac{\lambda}{\pi(\partial)/R_T(\partial)}=1\,,
		\end{equation}
		and 
		\begin{equation}\label{eq:fvtl-convergence-exp}
			\lim_{N\to\infty}\sup_{t\ge 0}\left| \frac{\mathbf{P}_\pi(\tau_\partial>t)}{(1-\lambda)^t}-1\right|=0\,.
		\end{equation}
		In particular,
		\begin{equation}\label{eq:fvtl-convergence-mean}
			\lim_{N\to\infty} \frac{\mathbf{E}_\pi[\tau_\partial]}{R_T(\partial)/\pi(\partial)}=1\,.
		\end{equation}
	\end{theorem}
	\begin{remark}[Continuous vs discrete time]\label{rmk:discrete-time}
		Despite the fact that our main results in Section \ref{suse:main-results} are written for continuous-time processes, the First Visit Time Lemma above is written for a general discrete time chain. Let us clarify now that there is no difference between the two, due to the fact that the exit rates are normalized to $1$, that is
		\begin{equation}
			L_{\rm rw}=P-{\rm Id}\,.
		\end{equation}
		Therefore, by a Poissonization argument, it follows immediately that neither $\pi$ nor the order of $t_{\rm mix}$ depend on the fact that time is discrete or continuous. Similarly, the convergence in \eqref{eq:fvtl-convergence-exp} can be rephrased in continuous time as
		\begin{equation}\label{eq:fvtl-convergence-exp-2}
			\lim_{N\to\infty}\sup_{t\ge 0}\left| \frac{\mathbf{P}_\pi(\tau_\partial>t)}{e^{-\lambda t}}-1\right|=0\,.
		\end{equation}
	\end{remark}
	We aim to compute the first order asymptotic of the meeting time of independent random walks on a graph $G$. To do so, the idea is to apply the FVTL to the product chain on $[n]\times[n]$ associated to the transition matrix
	\begin{equation}\label{eq:def-P-otimes2}
		P^{\otimes 2}=\frac12 \left(P\otimes I+I\otimes P \right).
	\end{equation}
	Notice that the latter corresponds to the evolution of two independent asynchronous random walks on $G$ where at each step a walk is selected u.a.r. to perform a single random walk step. Clearly, the stationary distribution of $P^{\otimes2}$ is given by $\pi\otimes \pi$. It follows that, in our setting, the role of the \emph{target} in the FVTL is played by the diagonal set
	\begin{equation} \label{eq:def_of_Delta}
		\Delta\coloneqq\{(x,y)\in [n]\times [n]\mid x=y\}.
	\end{equation}
	Nonetheless, being the FVTL suited for a single target, an extra step is needed when the target is a non-trivial set. As one can imagine, this extra step lies in considering a chain in which the target set is merged into a single state. In other words, we consider the state space
	\begin{equation}
		\tilde V \coloneqq ([n]^2 \setminus \Delta) \cup \{\partial\},
	\end{equation}
	where the state $\partial$ represents the merged set $\Delta$. Then, we are interested in constructing a process $\tilde P$ on $\tilde V$ having the following properties:
	\begin{enumerate}
		\item The transitions from $\tilde V\setminus \{\partial \}$ to $\tilde V\setminus \{\partial \}$ should coincide with those from $[n]^2\setminus \Delta$ to $[n]^2\setminus \Delta$ for the product chain $P^{\otimes 2}$.
		\item The transitions from any $\x\in\tilde V\setminus \{\partial \}$ to $\partial$ should coincide with the cumulative transitions from $\x\in[n]^2\setminus \Delta$ to $\Delta$ for the product chain $P^{\otimes 2}$.
		\item The stationary distribution of the new chain, say $\tilde \pi$, satisfies
		\begin{align}
			\pl \pi^2\pr\coloneqq\sum_{v\in[n]}\pi(v)^2=\tilde \pi(\partial),\qquad \pi(x)\pi(y)=\tilde\pi((x,y)),\quad  x,y\in[n]\text{ s.t. } x\neq y.
		\end{align}
	\end{enumerate}
	Indeed, if the conditions (1), (2) and (3) are satisfied, then we have the identity
	\begin{equation}\label{eq:identity}
		\tilde{\mathbf{E}}_{\tilde\pi}[\tau_\partial]=\mathbf{E}^{\otimes 2}_{\pi\otimes\pi}[\tau_{\rm meet}],
	\end{equation}
	thus it is enough to check that the assumption of the FVTL apply to the chain $\tilde P$ to conclude that the meeting time occurs at a geometric time of rate given by \eqref{eq:fvtl-convergence-lambda}.
	As pointed out in \cite{MQS21}, it is immediate to check that all the three conditions can be satisfied by defining $\tilde P$ as follows:
	call
	\begin{equation}\label{eq:def-tilde-mu}
		\tilde\mu(v)\coloneqq \frac{\pi(v)^2}{\pl \pi^2\pr},\qquad v\in[n]
	\end{equation}
	and for every $\x,\y \in\tilde V$ define
	\begin{equation}\label{eq:def-tilde-P}
		\tilde P(\x,\y)=\begin{cases}
			\frac12 P(x_1,y_1)&\x=(x_1,z),\,\y=(y_1,z),\,z\neq x_1,y_1,\\
			\frac12 P(x_2,y_2)&\x=(z,x_2),\,\y=(z,y_2),\,z\neq x_2,y_2,\\
			\frac12\left(P(x_1,x_2)+P(x_2,x_1)\right)&\x=(x_1,x_2),\,\y=\partial\,,x_1\neq x_2,\\
			\frac12\left(\tilde\mu(y_1)P(y_1,y_2)+ \tilde\mu(y_2)P(y_2,y_1)\right)&\x=\partial,\,\y=(y_1,y_2),\,y_1\neq y_2,\\
			\sum_{z\in[n]}\tilde\mu(z)P(z,z)&\x=\partial,\,\y=\partial.
		\end{cases}
	\end{equation}
	In other words, $\tilde P$ can be thought as the product chain $P^{\otimes 2}$ with the addition of a \emph{reset step} at $\Delta$: when the random walks on $G$ meet, they are instantaneously ``reset'' to another vertex (possibly the same) sampled according to the probability distribution $\tilde\mu$.

		As it will become clear later, one of the challenges in the application of the First Visit Time Lemma lies in the explicit computation of the truncated Green function $R_T(\partial)$. The main idea, that will be fully developed in Section \ref{sec:rws-mu-reset}, is based on the intuition that the graph of the Markov chain $\tilde Q$, \emph{seen from $\partial$}, looks like a infinite transient (and directed) graph. Based on this intuition, we will couple the quantity $R_T(\partial)$ with the Green function of the random walk on this infinite graph rooted at $\partial$. By transience, the latter can be computed as the escape probability of the random walk, i.e., the probability that the walk never returns to the root. More precisely, let us call $\cT_{\tilde\mu}$ an out-directed Galton-Watson tree in which the root has $k$
		children with probability $\sum_{x\in V}\tilde{\mu}(x)\ind_{d^+_x=k}$, while every other vertex has offspring distribution $\lambda^{\rm biased}$.
		In view of the construction of the process $\tilde P$, the escape probability described above will be shown to coincide with the (annealed) probability that two walks starting at the root of $\cT_{\tilde\mu}$ never meet anymore after time $t=0$.
		
		As the reader can imagine, such an approach can be worked out similarly also in the undirected setup, as already heuristically explained, e.g., in \cite{AF02,Dur07}. Nevertheless, the explicit computation of the escape probability requires a detailed combinatorial analysis of the possible trajectories on the random limiting graph, which is much easier in our directed setting where backtracking is not allowed.

	\begin{remark}\label{rmk:general-mu}
		If we desire that only conditions (1) and (2) are satisfied by $\tilde P$, then the reset distribution $\tilde \mu$ can be replaced by any probability distribution $\mu$ on $[n]$. In what follows we will need to adopt such a general perspective. In fact, one should keep in mind that we are interested in \emph{random} directed graphs. Therefore, the distribution $\tilde\mu$ in \eqref{eq:def-tilde-mu}, which depends on $\pi$, is a \emph{random} probability distribution depending on the realization of the graph $G$. This adds a further level of complication, which distinguishes the directed setting from the classical \emph{undirected} configuration model. For this reason, as will be explained in Section \ref{suse:organization}, we will start by analyzing chains with a fixed distribution $\mu$ (which does not depend on the graph), and subsequently prove the result for $\tilde\mu$ as in \eqref{eq:def-tilde-mu} by means of some concentration arguments.
	\end{remark}
	
	\begin{remark}[Continuous vs discrete time -- continued]\label{rmk:discrete-time-2}
		It is worth to stress that the discrete time processes in \eqref{eq:def-P-otimes2} and \eqref{eq:def-tilde-P} are, in expectation, twice as slow as their continuous time counterparts. That is to say, the expected meeting time of two independent discrete-time random walks evolving accordingly to $P^{\otimes2}$ is twice the expected meeting time of the same two particles evolving independently according to $L_{\rm rw}$. 
		It will be technically convenient in what follows to study the discrete-time processes introduced in this section, and then recall that the value of the constant $\lambda$ in Theorem \ref{fvtl} has to be doubled to recover the same result in the continuous time setting. Notice indeed that, a posteriori, $\lambda\sim(\vartheta \,n)^{-1}$, from which, in the continuous time setting of Theorem \ref{th:main}, we get the factor $1/2$ in the denominator of  \eqref{eq:meeting time law}.
	\end{remark}
	In light of Theorem \ref{fvtl} and \eqref{eq:identity}, the proof of Theorem \ref{th:main} relies on three main propositions, ensuring that conditions of Theorem \ref{fvtl} are satisfied w.h.p. by the (random) chain $\tilde P$ in \eqref{eq:def-tilde-P}:

	\begin{proposition}[Mass at the diagonal]\label{prop:pi-diag}
		Assume the degree sequences satisfy Assumption \ref{log degree assumptions}. Let $\mathfrak{p}$ as in \eqref{eq:def-mathfrak-p}. Then it holds that
		\begin{equation}
			\frac{  n \pl \pi^2 \pr}{\mathfrak{p}} \overset{\P}{\longrightarrow}1.
		\end{equation}
	\end{proposition}
	
	\begin{proposition}[Returns to the diagonal]\label{prop:return-diag}
		Assume the degree sequences satisfy Assumption \ref{log degree assumptions}. Let $\mathfrak{q}$ and $\mathfrak{r}$ be as in \eqref{eq:def-mathfrak-q} and \eqref{eq:def-mathfrak-r} respectively. Then, with $T:= \lfloor\log^5(n)\rfloor$, it holds
		\begin{equation}
			\frac{R_T(\partial)}{\mathfrak{r}} \overset{\P}{\longrightarrow}1\ .
		\end{equation}
	\end{proposition}
	
	\begin{proposition}[Mixing time for the collapsed chain]\label{prop:tilde-mixing}
		Assume the degree sequences satisfy Assumption \ref{log degree assumptions}. Then, the mixing time of the chain $\tilde P$ defined in \eqref{eq:def-tilde-P}, i.e.,
		\begin{equation}
			\tilde{t}_{\rm mix}\coloneqq \inf\left\{ t\ge 0\mid \max_{\x\in\tilde{V}}\|\tilde{P}^t(\x,\cdot)-\tilde\pi\|_{\rm TV}\le \frac1{2e} \right\} ,
		\end{equation}
		satisfies 
		\begin{equation}
			\frac{\tilde{t}_{\rm mix}}{\log(n)^3}\overset{\P}{\longrightarrow}0.
		\end{equation}
	\end{proposition}
	
	At this point, our main results follow at once.
	\begin{proof}[Proof of Theorem \ref{th:main} and Corollary \ref{coro:meet-coal-cons}]
		To see the validity of Theorem \ref{th:main} it is enough to realize that, by Propositions \ref{prop:pi-diag}, \ref{prop:return-diag}, \ref{prop:tilde-mixing} and  Theorem \ref{th:extremal-pi}, the convergences in \eqref{eq:fvtl-convergence-lambda}, \eqref{eq:fvtl-convergence-exp} and \eqref{eq:fvtl-convergence-mean} hold in probability for the process $\tilde{\mathbf{P}}$. Hence, by the identity in \eqref{eq:identity}, they hold in probability for the product chain $\mathbf{P}^{\otimes2}$. The passage to continuous time follows immediately by Remarks \ref{rmk:discrete-time} and \ref{rmk:discrete-time-2}. Then, the convergence in probability of the Wasserstein-1 distance in \eqref{eq:meeting time law} follows by the continuous time version of \eqref{eq:fvtl-convergence-exp} by working under the appropriate high-probability events.
		
		As for Corollary \ref{coro:meet-coal-cons}, notice that the condition in \eqref{MFcondCoal} is satisfied w.h.p. thanks to Theorems \ref{th:cutoff} and \ref{th:extremal-pi}. At this point, the  convergences in probability in \eqref{eq:oli-coal-DCM} and \eqref{eq:oli-cons-DCM}  follow by Theorems \ref{MFcoal} and \ref{th:mean-cons}.
	\end{proof}
	
	\subsection{Heavy-tailed in-degrees}\label{suse:degreeassumptions}
	We presented our main the result under the strong constraint of uniform boundedness of out- and in-degrees. Nevertheless, it is natural to conjecture that the result still holds true in a more general setup. In particular, in \cite{CCPQ21} the authors show that Theorem \ref{th:cutoff} holds under the following weaker assumption
	\begin{assumption}[Relaxed degree assumptions]\label{degree assumptions}
		Assume the same constraints (a) and (b) as in Assumption \ref{log degree assumptions}, with (c) replaced by: there exists some $\epsilon>0$ such that for all $n\in\N$
		\begin{equation}
			\begin{split}
				&\text{(c'):}\qquad \sum_{x\in[n]}(d_x^-)^{2+\epsilon}\le C\ n\ . 
			\end{split}
		\end{equation}
	\end{assumption}
	Notice that under Assumption \ref{degree assumptions} the quantities in \eqref{eq:def-rho-gamma} are still $\Theta(1)$, even though one can have any
	$d_{\max}^-\asymp n^{1/2-o(1)}$. Notice indeed that the latter condition still ensures that the quantities in \eqref{eq:def-rho-gamma} are all of order 1. In fact, our technical results are proved under such weaker assumptions. The only technical lemma in which we need to impose a stronger assumption is Lemma \ref{prop:sec-mom-y}, where we essentially show the concentration of the quantities in \eqref{eq:frak-p-q-emp}. Therein, our techniques require the following bound
	\begin{equation}\label{eq:extra-degree-assumption}
		\begin{split}
			&\text{(d):}\qquad d_{\max}^-\le C\ n^{\frac13-\epsilon}\ . 
		\end{split}
	\end{equation}
	
	In particular, our proof of Proposition \ref{prop:tilde-mixing}, works in the more general framework of Assumption \ref{degree assumptions}, whereas Proposition \ref{prop:pi-diag} and \ref{prop:return-diag} require in addition \eqref{eq:extra-degree-assumption}. 
	
	Nevertheless, to formally extend the result in Theorem \ref{th:main} to any degree sequences satisfying Assumption \ref{degree assumptions} and \eqref{eq:extra-degree-assumption}, it is further required a control on the minimum of $\pi$. In fact, the second result in Theorem \ref{th:extremal-pi} has been proved in \cite{CP20} only in the bounded degree setting. Nevertheless, it is natural to expect that the effect of large in-degrees cannot impact the fact that $\pi_{\min}=n^{-O(1)}$. 
	Investigating the extent to which the heuristic argument above can be made precise is not the focus of this work.

	\subsection{Organization of the proof}\label{suse:organization}
	Due to the point raised in Remark \ref{rmk:general-mu}, we start  by analyzing independent random walks undergoing the reset according to any fixed distribution $\mu$. In particular, in Section \ref{sec:rws-mu-reset} we show that the \emph{annealed version} of such random walks can be successfully coupled with certain random rooted forests. In Section \ref{sec: computation of R_Delta} we exploit such a computational tool to control the quantity $R_T(\partial)$ in Theorem \ref{fvtl}, again in the case in which the graph $G$ is random but the auxiliary chain $\tilde P$ has a deterministic reset distribution $\mu$ in place of $\tilde \mu$. In doing so, we show that the dependence of $R_T(\partial)$ on $\mu$ is weak, meaning that it depends only on the expectation with respect to $\mu$ of certain bounded functions of the out-degrees. Subsequently, in Section \ref{sec:random-mu}, we show that the expectation of  such bounded functions with respect to the random distribution $\tilde \mu$ concentrates around a function of the degree sequence, see Proposition \ref{prop:general-formula}. In particular, as a consequence of such a general result, we extend the result obtained  for deterministic reset distributions in Section \ref{sec: computation of R_Delta} to the case of $\tilde \mu$-resets, thus deducing the validity of Propositions \ref{prop:pi-diag} and \ref{prop:return-diag}. Finally, Section \ref{sec: Mixing time} is devoted to the proof of Proposition \ref{prop:tilde-mixing}.
	
	\section{Random walks with $\mu$-reset}\label{sec:rws-mu-reset}
	Throughout this section, we fix a probability distribution $\mu=\mu_n$ on $[n]$ and consider, for a fixed realization of the environment $\omega$, the discrete time Markov chain $\tilde P=\tilde P_\mu^{\omega}$ on $\tilde V$, defined as in \eqref{eq:def-tilde-P} with $\mu$ in place of $\tilde\mu$. It is not hard to see that for each choice of $\mu$ the chain $\tilde P_\mu$ will have a unique stationary distribution $\tilde \pi_{\mu}$. 
	For a prescribed distribution $\mu$, recall that the transition probabilities from $\partial$ are as follows: 
	\begin{equation*}
		\tilde{P}_\mu(\partial,\x)=
		\begin{cases}
			\frac12\left(\mu(x) P(x,y)+\mu(y) P(y,x)\right) &  \x=(x,y)\ ,\\
			\sum_{z\in[n]}\mu(z) P(z,z) &  \x = \partial\ ,
		\end{cases}\ \qquad \x\in\tilde{V}\ .
	\end{equation*}
	Moreover, we will use the symbols $\tilde{\mathbf{P}}^{\mu}_\partial=\tilde{\mathbf{P}}^{\omega,\mu}_{\partial}$ and  $\tilde{\mathbf{E}}^\mu_\partial=\tilde{\mathbf{E}}^{\omega,\mu}_{\partial}$ to denote the law and the expectation of the trajectories of the Markov chain $(\tilde X_t)_{t\ge 0}$ on $\tilde V$, with $\tilde X_0=\partial$ and transition matrix $\tilde P=\tilde P_\mu^{\omega}$, on the quenched realization $G=G^\omega$ of the graph. 
	Finally, we remind the reader that the symbols $\P$ and $\E$ are concerned uniquely with the probability space of the generation of the environment $\omega$.

	\subsection{Annealed random walks with $\mu$-reset}\label{suse:annealed-forest}
	In what follows we will be interested in computing expectations, with respect to $\P$, to the transition probabilities of the chain $\tilde P$. A very useful way to compute this quantities (see Remark \ref{rmk:multiple-annealed-4} below for references) is to rewrite such expectations as probabilities of a non-Markovian process that we will call \emph{annealed walks}. In other words, the annealed walks process is the one in which the environment $\omega$ is realized together with the random walks evolving on it.
	More precisely, for every $A\subseteq\tilde V$ and $t\ge 0$ one can write
	\begin{equation}
		\E\left[\tilde{\mathbf{P}}^{\mu}_\partial\left(\tilde{X}_t\in A\right)\right] = \tilde\P^{\text{an},\mu}_\partial\left(\tilde{W}_t\in A\right),
		\label{eq: annealed law}
	\end{equation}
	where the law $\tilde\P^{\text{an},\mu}_\partial$ is the law of the random variable $(\omega_t,\tilde W_t)_{t\ge 0}$, where $\omega_t$ is a partial matching of $E^+$ and $E^-$ with at most $t$ matchings, and $\tilde W_t\in \tilde V$ is the location of the annealed walk at time $t$. 
	
	The equality in \eqref{eq: annealed law} is a consequence of the exchangeability property of the DCM that translates into a sort of \emph{spatial Markov property}. It resembles the fact that we can compute the left-hand side of \eqref{eq: annealed law} using a local exploration process.
	More explicitly, samples according to $\tilde\P^{\text{an},\mu}_\partial$ can be obtained as  result of the following randomized algorithm:
	\begin{enumerate}
		\item set $\tilde W_0=\partial$;
		\item for all $s\in \{0,\dots,t\}$, if $\tilde{W}_s=\partial$, sample a vertex $x$ according to $\mu$. Select one of the tails of $x$, $e\in E^+_x$ uniformly at random:
		\begin{itemize}
			\item[(2a)] if $e$ is already matched to some $f\in E^-$, call $v_f$ the vertex incident to the head $f$. If $v_f\neq x$, set either $\tilde{W}_{s+1}=(x,v_f)$ or $\tilde{W}_{s+1}=(v_f,x)$  w.p. $1/2$. Whereas, if $v_f= x$, set $\tilde{W}_{s+1}=\partial$.
			\item[(2b)] if $e$ is unmatched, choose u.a.r. some $f\in E^-$ which is still unmatched, set $\omega(e)=f$. Call $v_f$ the vertex incident to the head $f$. If $v_f\neq x$ set either $\tilde{W}_{s+1}=(x,v_f)$ or $\tilde{W}_{s+1}=(v_f,x)$  w.p. $1/2$. Whereas, if $v_f=x$, set $\tilde{W}_{s+1}=\partial$.
		\end{itemize}
		\item  if instead $\tilde{W}_s=(x,y)$ with $x\neq y$, select the coordinate to move with uniform probability. For the sake of illustration, let us as assume that this the first coordinate. Select u.a.r. one of the tails of vertex $x$ (which is associated with the selected coordinate) and call it $e\in E_x^+$:
		\begin{itemize}
			\item[(3a)] if $e$ was already matched at a previous step to some head $f\in E^-$, let $v_f \in [n]$ denote the vertex incident to the head $f$. Then, if $v_f\neq y$ let $\tilde{W}_{s+1}=(v_f,y)$, otherwise let $\tilde{W}_{s+1}=\partial$;
			\item[(3b)] if $e$ is still unmatched, select uniformly at random a head, $f\in E^-$, among the unmatched ones, match it to $e$, and  let $v_f \in [n]$ the associated vertex. As in the previous case,  if $v_f\neq y$ let $\tilde{W}_{s+1}=(v_f,y)$, otherwise let $\tilde{W}_{s+1}=\partial$.
		\end{itemize}
	\end{enumerate}
	\begin{remark}[Multiple annealed random walks]\label{rmk:multiple-annealed}
		In what follows, it will be convenient to use the \emph{annealed philosophy} to compute higher moments of the transition probabilities associated to the law $\tilde{\mathbf{P}}^\mu_\partial$. Indeed, for all $\kappa\in\N$, $A_1,\dots,A_\kappa\subset\tilde{V}$ and $t_1,\dots,t_\kappa\ge 0$ one can write
		\begin{equation}
			\E\left[\prod_{i=1}^\kappa\tilde{\mathbf{P}}_\partial^{\mu}\left(\tilde{X}_{t_i}\in A_i\right) \right] = \tilde\P^{\kappa\text{-an},\mu}_\partial\left(\tilde{W}_{t_i}^{(i)}\in A_i\,,\,\forall i\le \kappa\right),
		\end{equation}
		where the random variables $(\tilde{W}_{s}^{(i)})_{i\le \kappa\,,\,s\le t_i}$ can be sampled by means of the same construction described above, constructing the $\kappa$ walks sequentially, with the $i$-th one evolving in the partial environment constructed by the first $i-1$.
	\end{remark}
	
	\begin{remark}[Annealed random walks without reset]\label{rmk:multiple-annealed-3}
		As should be clear to the reader, an annealing formula is available also for the simple random walk, with law $\mathbf{P}$, and for the independent random walks with law $\mathbf{P}^{\otimes2}$. Moreover, an annealing formula can be written also when conditioning on a partial realization of the underlying graph. 
		
		In general, called $\sigma$ a partial matching of $E^-$ and $E^+$ and fixed an initial distribution $\nu\in\cP([n]^2)$ depending only on $\sigma$, any $\kappa\in\N$, a collection of times $t_1,\dots,t_\kappa\ge 0$ and some events $\cE_1(X^{(1)},X^{(2)},t_1),\dots,\cE_\kappa(X^{(1)},X^{(2)},t_\kappa)$ depending only on the trajectory of $(X^{(1)},X^{(2)})$ up to time $t_i$, respectively, one can write
		\begin{equation}
			\E\left[\prod_{i=1}^\kappa{\mathbf{P}}_\nu^{\otimes 2}\left(\cE_i(X^{(1)},X^{(2)},t_i)\right) \mid\sigma\right] = \P^{\otimes2,\,\kappa\text{-an}|\sigma}_\nu\big(\cap_{i\le \kappa}\cE_i(W^{(i,1)},W^{(i,2)},t_i) \big),
		\end{equation}
		where $\P^{\otimes2\,,\kappa\text{-an}|\sigma}_\nu$ is the joint law of $(\omega_{s})_{s\le \kappa T}$ and $(W_{s}^{(i,1)},W_{s}^{(i,2)})_{i\le \kappa\,,\,s\le t_i}$ can be sampled by constructing the $\kappa$ couples of walks sequentially, with the first couple starting at $\nu$ with $\omega_0=\sigma$ and, in general, the $i$-th couple starting independently at $\nu$ and evolving in the partial environment given by $\sigma$ and the additional matchings constructed by the first $i-1$ couples. Differently from the construction of annealed walks with $\mu$-reset, when a couple of annealed walks meet, they do not experience any reset but rather the construction continues as for any other point in $[n]^2$. 
	\end{remark}
	
	\begin{remark}[Annealed simple random walks]\label{rmk:multiple-annealed-4}
		The construction of the annealed random walk with $\mu$-reset presented in this section---as well as its variant without reset introduced in Remark \ref{rmk:multiple-annealed-3}---is a  natural variation of the original \emph{annealed simple random walk}, which is a classical tool for the analysis of random walks on sparse random graphs. In particular, in the context of random directed graphs, it has been used extensively in    \cite{BCS18,BCS19,CQ20,CQ21a,CQ21b,CCPQ21}.
		
		In the \emph{annealed simple random walk} construction, there is a single annealed walk $W$ that replaces either $\tilde{W}$ or $(W^{(1)},W^{(2)})$.  Nevertheless, we postpone a detailed definition of the algorithm to the proof of Lemma \ref{lemma:exp-lambda}.  In order to introduce the notation that will be used later, call $\sigma$ a partial matching of $E^-$ and $E^+$ and fix an initial distribution $\mu\in\cP([n])$ depending only on $\sigma$, any $\kappa\in\N$, a collection of times $t_1,\dots,t_\kappa\ge 0$ and some events $\cE_1(X,t_1),\dots,\cE_\kappa(X,t_k)$ depending only on the trajectory of $X$ up to time $t_i$, respectively. We write
		\begin{equation}
			\E\left[\prod_{i=1}^\kappa{\mathbf{P}}_\mu(\cE_i(X,t_i))\mid\sigma\right] = \P^{\kappa\text{-an}|\sigma}_\nu\big(\cap_{i\le \kappa}\cE_i(W^{(i)},t_i) \big),
		\end{equation}
		where the  $\P^{\kappa\text{-an}|\sigma}_\nu$ is the joint law of $(\omega_{s})_{s\le \kappa T}$ and $(W_{s}^{(i)})_{i\le \kappa\,,\,s\le t_i}$ can be sampled by constructing the $\kappa$ walks sequentially, with the first one starting at $\mu$ with $\omega_0=\sigma$ and, in general, the $i$-th walk independently starting at $\nu$ and evolving in the partial environment given by $\sigma$ and the additional matchings constructed by the first $i-1$ walks.
	\end{remark}
	
	In the same spirit of Section \ref{sec: Local structures} we now define the generation process of a random forest which will be coupled with the construction of the annealed walk just described. 
	More precisely, our aim is to show that, under some weak assumption on the distribution $\mu$, the part of graph explored by $\tilde W$ within time $t=t_n$, for $t$ not too large, can be coupled at a polynomially small TV-cost to a random rooted forest $\mathcal{F}$, i.e. a collection of independent, rooted, multi-type Galton-Watson trees, the construction of which we now describe. 
	
	The forest we construct is made by a random number of (out)-trees. Every vertex in each tree has a mark $v\in[n]$.  Moreover, each vertex in each tree has at most two children and in particular at most one vertex in the whole forest has exactly two children. The construction of the forest goes as follows: 
	\renewcommand{\theenumi}{\roman{enumi}}
	\begin{enumerate}
		\item At time $s=0$:
		\begin{itemize}
			\item\label{item:step1} select the mark of the root according to $\mu$, say $r\in[n]$, put a red flag on the root, set $\tilde R_0=r$ and sample $\texttt{red}_0\in\{1,2\}$ u.a.r.;
			\item sample $e\in E^+_r$ and $f\in E^-$ u.a.r.; create a new edge of the forest from the root to a new vertex with mark $v_f$, and give a label $(e,f)$ to such a new edge; put a blue flag on the new vertex of the forest and set $\tilde{B}_0=v_f$.
		\end{itemize}
		\item Given the construction up to time $s\ge0$, construct $(\cF_{s+1},\tilde{R}_{s+1},\tilde{B}_{s+1},\texttt{red}_{s+1})$ as follows:
		\begin{enumerate}
			\item If the red flag and the blue flag are on the same vertex, end the construction of the current tree and construct another tree of the forest starting as in Item \ref{item:step1}.
			\item If the red flag and the blue flag are in different spots in the tree,  select one of the two colors with uniform probability:
			\begin{itemize}
				\item\label{item:blue moves} If blue is selected, and $v$ is the type of the vertex with the blue flag attached, sample $e\in E^+_v$ and $f\in E^-$ u.a.r., create a new edge of the forest from the vertex with the blue flag to a new vertex with mark $v_f$, and give label $(e,f)$ to such a new edge. Move the blue flag to the new vertex of the forest, and set $\tilde{R}_{s+1}=\tilde{R}_s$,  $ \tilde B_{s+1}=v_f$ and $\texttt{red}_{s+1}=\texttt{red}_s$.
				\item If red is selected, and $v$ is the mark of vertex with the red flag on top proceed as follows:
				\begin{itemize}
					\item If $v$ has a (unique) child: sample $e\in E^+_v$ u.a.r.
					\begin{itemize}
						\item[(I)] if the edge connecting $v$ to its child has label $(e,f)$ for some $f\in E^-$, move the red flag to the $v_f$ and set $\tilde{R}_{s+1}=v_f$, $\tilde{B}_{s+1}=\tilde{B}_s$ and $\texttt{red}_{s+1}=\texttt{red}_s$.
						\item[(II)] otherwise, sample $f\in E^-$ u.a.r., create a new edge of the forest from the vertex with the red flag to a new vertex with mark $v_f$. Label the new edges as $(e,f)$, and move the red flag to such a new vertex. Set $\tilde{R}_{s+1}=v_f$, $\tilde{B}_{s+1}=\tilde{B}_s$ and $\texttt{red}_{s+1}=\texttt{red}_s$.
					\end{itemize}
					\item If $v$ has no children, sample $e\in E^+_v$ and $f\in E^-$ u.a.r., create a new edge of the forest from the vertex with the red flag to a new vertex with mark $v_f$, and give a label $(e,f)$ to such a new edge. Move the red flag to the newly created vertex. Set $\tilde{R}_{s+1}=v_f$, $\tilde{B}_{s+1}=\tilde{B}_s$ and $\texttt{red}_{s+1}=\texttt{red}_s$.
				\end{itemize}
			\end{itemize}
		\end{enumerate}
	\end{enumerate}
	
	\begin{figure}
		\centering
		\begin{tikzpicture}
			\foreach \ang [count=\n from 1] in {90,30,...,-210}
			\node[] at (-7.5,2.8) {(0)};
			\node[] at (-2.5,2.8) {(1)};
			\node[] at (2.5,2.8) {(2)};
			\node[] at (-7.5,-2.3) {(3a)};
			\node[] at (-2.5,-2.3) {(3b)};
			\node[] at (2.5,-2.3) {(3c)};
			
			\node (r1) [label={[red]left:\VarFlag}] [cnd] at (-6,2.5) {$r$};
			\node (x1) [label={[blue]left:\VarFlag}] [cnd] at (-5,1) {$x$};
			
			\node (r2) [label={[red]left:\VarFlag}] [cnd] at (-1,2.5) {$r$};
			\node (x2)  [cnd] at (0,1) {$x$};
			\node (y1)  [label={[blue]left:\VarFlag}]  [cnd] at (-0.5,-0.5) {$y$};
			
			\node (r3)  [cnd] at (4,2.5) {$r$};
			\node (x3) [label={[red]left:\VarFlag}] [cnd] at (5,1) {$x$};
			\node (y2)  [label={[blue]left:\VarFlag}]  [cnd] at (4.5,-0.5) {$y$};
			
			\node (r4)  [cnd] at (-6,-2.5) {$r$};
			\node (x4) [label={[red]left:\VarFlag}] [cnd] at (-5,-4) {$x$};
			\node (y3)  [cnd] at (-5.5,-5.5) {$y$};
			\node (w)  [label={[blue]left:\VarFlag}]  [cnd] at (-5.2,-7) {$w$};

			\node (r5)  [cnd] at (-1,-2.5) {$r$};
			\node (x5) [cnd] at (0,-4) {$x$};
			\node (y4)  [label={[blue]left:\VarFlag}]  [cnd] at (-0.5,-5.5) {$y$};
			\node[] at (-1.8,-5.5) {\textcolor{red}{\VarFlag}};
			
			\node (r6)  [cnd] at (4,-2.5) {$r$};
			\node (x6) [cnd] at (5,-4) {$x$};
			\node (y5)  [label={[blue]left:\VarFlag}]  [cnd] at (4.5,-5.5) {$y$};
			\node (z)  [label={[red]right:\VarFlag}]  [cnd] at (5.8,-5.5) {$z$};

			\draw [-Bar] (r1.-90) -- +(-90:0.3);
			\draw [-Bar] (r1.-120) -- +(-120:0.3);
			\draw [<-Bar] (r1.60) -- +(60:0.3);
			\draw [<-Bar] (r1.90) -- +(90:0.3);
			\draw [<-Bar] (r1.120) -- +(120:0.3);
			
			\draw [-Bar] (x1.-70) -- +(-70:0.3);
			\draw [-Bar] (x1.-120) -- +(-120:0.3);
			\draw [<-Bar] (x1.60) -- +(60:0.3);
			\draw [<-Bar] (x1.90) -- +(90:0.3);
			
			\draw[edge] (r1.-60) to (x1.120);
			
			\draw [-Bar] (r2.-90) -- +(-90:0.3);
			\draw [-Bar] (r2.-120) -- +(-120:0.3);
			\draw [<-Bar] (r2.60) -- +(60:0.3);
			\draw [<-Bar] (r2.90) -- +(90:0.3);
			\draw [<-Bar] (r2.120) -- +(120:0.3);
			
			\draw [-Bar] (r2.-90) -- +(-90:0.3);
			\draw [-Bar] (r2.-120) -- +(-120:0.3);
			\draw [<-Bar] (r2.60) -- +(60:0.3);
			\draw [<-Bar] (r2.90) -- +(90:0.3);
			\draw [<-Bar] (r2.120) -- +(120:0.3);
			
			\draw [-Bar] (x2.-70) -- +(-70:0.3);
			\draw [-Bar] (y1.-60) -- +(-60:0.3);
			\draw [-Bar] (y1.-90) -- +(-90:0.3);
			\draw [-Bar] (y1.-120) -- +(-120:0.3);
			\draw [<-Bar] (x2.60) -- +(60:0.3);
			\draw [<-Bar] (x2.90) -- +(90:0.3);
			\draw [<-Bar] (y1.120) -- +(120:0.3);
			
			\draw[edge] (r2.-60) to (x2.120);
			\draw[edge] (x2.-120) to (y1.90);
			
			\draw [-Bar] (r3.-90) -- +(-90:0.3);
			\draw [-Bar] (r3.-120) -- +(-120:0.3);
			\draw [<-Bar] (r3.60) -- +(60:0.3);
			\draw [<-Bar] (r3.90) -- +(90:0.3);
			\draw [<-Bar] (r3.120) -- +(120:0.3);
			
			\draw [-Bar] (x3.-70) -- +(-70:0.3);
			\draw [-Bar] (y2.-60) -- +(-60:0.3);
			\draw [-Bar] (y2.-90) -- +(-90:0.3);
			\draw [-Bar] (y2.-120) -- +(-120:0.3);
			\draw [<-Bar] (x3.60) -- +(60:0.3);
			\draw [<-Bar] (x3.90) -- +(90:0.3);
			\draw [<-Bar] (y2.120) -- +(120:0.3);

			\draw[edge] (r3.-60) to (x3.120);
			\draw[edge] (x3.-120) to (y2.90);
			
			\draw [-Bar] (r4.-90) -- +(-90:0.3);
			\draw [-Bar] (r4.-120) -- +(-120:0.3);
			\draw [<-Bar] (r4.60) -- +(60:0.3);
			\draw [<-Bar] (r4.90) -- +(90:0.3);
			\draw [<-Bar] (r4.120) -- +(120:0.3);
			
			\draw [-Bar] (x4.-70) -- +(-70:0.3);
			\draw [-Bar] (y3.-60) -- +(-60:0.3);
			\draw [-Bar] (y3.-120) -- +(-120:0.3);
			\draw [<-Bar] (x4.60) -- +(60:0.3);
			\draw [<-Bar] (x4.90) -- +(90:0.3);
			\draw [<-Bar] (y3.120) -- +(120:0.3);
			
			\draw[edge] (r4.-60) to (x4.120);
			\draw[edge] (x4.-120) to (y3.90);
			
			\draw [-Bar] (r5.-90) -- +(-90:0.3);
			\draw [-Bar] (r5.-120) -- +(-120:0.3);
			\draw [<-Bar] (r5.60) -- +(60:0.3);
			\draw [<-Bar] (r5.90) -- +(90:0.3);
			\draw [<-Bar] (r5.120) -- +(120:0.3);
			
			\draw [-Bar] (x5.-70) -- +(-70:0.3);
			\draw [-Bar] (y4.-60) -- +(-60:0.3);
			\draw [-Bar] (y4.-90) -- +(-90:0.3);
			\draw [-Bar] (y4.-120) -- +(-120:0.3);
			\draw [<-Bar] (x5.60) -- +(60:0.3);
			\draw [<-Bar] (x5.90) -- +(90:0.3);
			\draw [<-Bar] (y4.120) -- +(120:0.3);
			
			\draw[edge] (r5.-60) to (x5.120);
			\draw[edge] (x5.-120) to (y4.90);
			
			\draw [-Bar] (r6.-90) -- +(-90:0.3);
			\draw [-Bar] (r6.-120) -- +(-120:0.3);
			\draw [<-Bar] (r6.60) -- +(60:0.3);
			\draw [<-Bar] (r6.90) -- +(90:0.3);
			\draw [<-Bar] (r6.120) -- +(120:0.3);
			
			\draw [-Bar] (y5.-60) -- +(-60:0.3);
			\draw [-Bar] (y5.-90) -- +(-90:0.3);
			\draw [-Bar] (y5.-120) -- +(-120:0.3);
			\draw [<-Bar] (x6.60) -- +(60:0.3);
			\draw [<-Bar] (x6.90) -- +(90:0.3);
			\draw [<-Bar] (y5.120) -- +(120:0.3);
			
			\draw[edge] (r6.-60) to (x6.120);
			\draw[edge] (x6.-120) to (y5.90);
			
			\draw [-Bar] (w.-65) -- +(-65:0.3);
			\draw [-Bar] (w.-35) -- +(-35:0.3);
			\draw [-Bar] (w.-95) -- +(-95:0.3);
			\draw [-Bar] (w.-125) -- +(-125:0.3);
			
			\draw[edge] (y3.-90) to (w.110);
			
			\draw [-Bar] (z.-60) -- +(-60:0.3);
			\draw [-Bar] (z.-90) -- +(-90:0.3);
			\draw [-Bar] (z.-120) -- +(-120:0.3);
			\draw [<-Bar] (z.60) -- +(60:0.3);
			\draw [<-Bar] (z.90) -- +(90:0.3);
			\draw [<-Bar] (z.150) -- +(150:0.3);
			
			\draw[edge] (x6.-70) to (z.120);
		\end{tikzpicture}
		\caption{An example of a realization of the forest. As shown in picture (0), $\tilde{R}_0=r$ is the mark, chosen according to $\mu$, of the first root and $\tilde{B}_0=x$. The forest at time $s=1$ is presented in picture (1), while $\tilde{R}_1=\tilde{R}_0=r$, $\tilde{B}_1=y$. Similarly, the forest at time $2$ is presented in picture (2), where $\tilde{R}_2=\tilde{B}_0=x$ and $\tilde{B}_2=\tilde{B}_1=y$. Given $(\cF_2,\tilde{R}_2,\tilde{B}_2)$, there are three possible scenarios that can happen: if the blue flag is selected, then it will be moved to a vertex having a random mark, say $w$, so that $\tilde{B}_3=w$ and $\tilde{R}_3=\tilde{R}_2=x$, as in picture $(3a)$; whereas, if the red flag is selected there are two possible scenarios: in the first one, described in picture $(3b)$, the random selected tail from $E^+_x$ coincides with the one connecting the vertex with mark $x$ to its unique child, thus the two flags will be in the same vertex and a new root needs to be sampled at the forthcoming step. It follows that $\tilde{R}_3=\tilde{B}_3=\tilde{B}_2=y$. Instead, picture $(3c)$ describes the situation in which the chosen tail does not coincide with the unique edge leaving $x$, therefore we need to select another random mark, say $z$, create a new labeled edge, and move the red flag to $z$. Hence the vertex with mark $x$ will have two children and the two flags will never meet again. It follows that $\tilde{R}_3=z$, and $\tilde{B}_3=\tilde{B}_2=y$.}
		\label{fig:annealing forest}
	\end{figure}
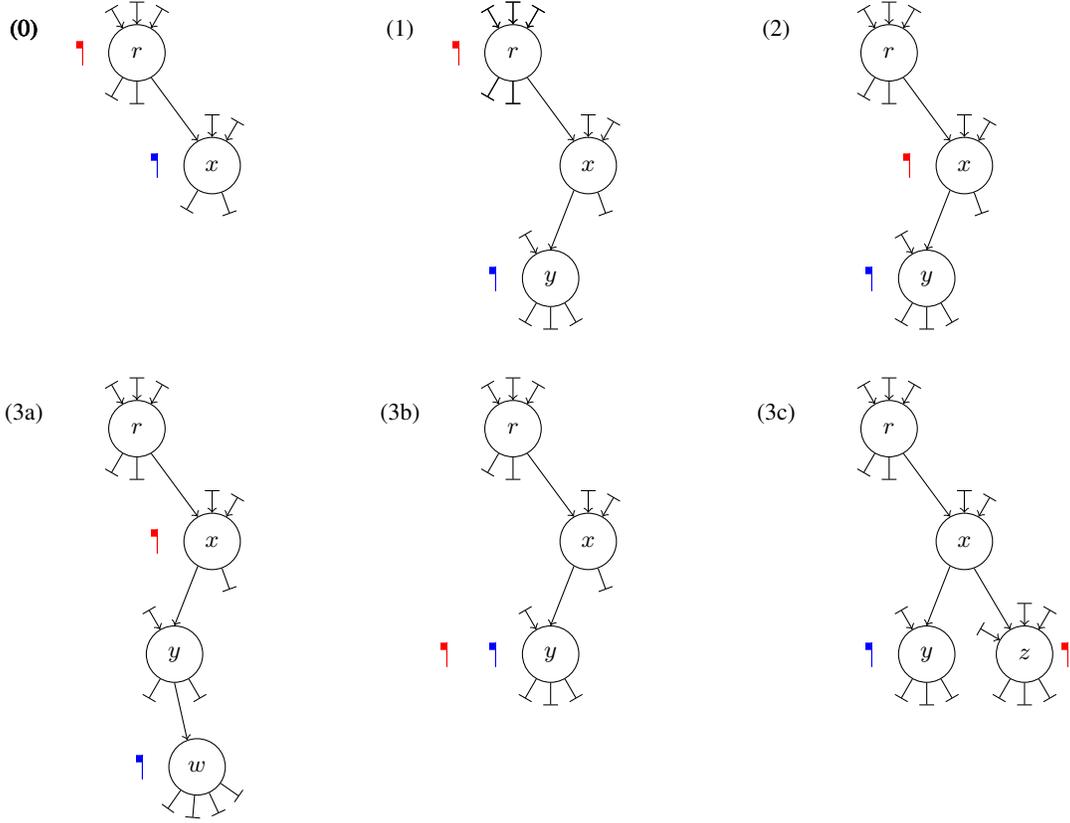
	Let us try to explain in plain English how the construction of the forest is performed: there are two particles sitting on the same spot $r$ with mark distributed according to $\mu$. As a result of a fair coin toss, one of the particles is given color red and the other color blue. The blue particle is the first to move, and creates a directed edge (of the forest) having the root as origin and as destination some vertex with mark $v$, sampled according to $\mu_{\rm in}$. At each subsequent step, a fair coin is tossed to decide which of the two colors has to move. The blue particle always creates a new directed edge as in the first step. On the other hand, the red particle, which sits on some vertex with mark $v_r$, can either follow the unique out-going edge already present at the vertex it is currently visiting (if any) w.p. $1/d_{v_r}^+$; or it can create a new edge from the vertex it is currently visiting to some new vertex in the forest which will have a mark sampled accordingly to $\mu^{\rm in}$. Notice that, in the latter scenario, the red and blue particles will never meet again. If the two particles meet (that is, the red particles reaches the blue one), then the process restarts as follows: the two particles will be placed on a new vertex that will be the root of a new tree in the forest; the mark of such vertex will be chosen accordingly to $\mu$ and the colors of the two particles are reinitialized. Notice also that the auxiliary alea needed to define $(\texttt{red}_t)_{t\ge 0}$, plays essentially no role in the construction, but it will be useful in a while to define the coupling. Indeed, as one can imagine at this point, the quantity $\texttt{red}_t$ will determine which of the two coordinates of $\tilde W$ is ``escaping'' and which is ``chasing'' at time $t$. Clearly such a construction leads to a collection $(\cF_t,\tilde{R}_t,\tilde{B}_t,\texttt{red}_t)_{t\ge 0}$, where  $\tilde{R}_t$ (resp. $\tilde{B}_t$) is the mark at the vertex in which lies the red (resp. blue) particle at time $t$, and $\cF_t$ is the rooted forest realized up to step $t$, which we identify with the collection of its labeled edges (with multiplicities).
	The next result shows that it is possible to couple the two constructions at a small TV-cost, as soon as $t$ is sufficiently small and $\mu$ is not too concentrated.
	
	\begin{proposition}\label{prop:coupling-forest}
		Assume that, for some $\varepsilon> 0$
		\begin{equation}
			\max_{x\in[n]}\mu(x)\le n^{-\varepsilon}\ . 
		\end{equation}
		Call $(\tilde \cL_s)_{s\le t}$ the law of the annealed walks construction of $(\omega_s, \tilde{W}_s )_{s\le t}$ and $(\tilde \cL_s^\cF)_{s\le t}$ the law of the process $(\cF_s, W(\tilde{R}_s, \tilde{B}_s,\texttt{red}_s),\texttt{red}_s)_{s\le t}$ 
		where $(\cF_s, \tilde{R}_s, \tilde{B}_s, \texttt{red}_s)_{s\le t}$ are sampled via the forest construction and $W:[n]^2\times\{1,2\}\to \tilde V$ is defined as
		\begin{equation}
			W(R,B,\texttt{red})\coloneqq\begin{cases}
				(R,B)&\text{if }\texttt{red}=1\text{ and }R\neq B\ ,\\
				(B,R)&\text{if }\texttt{red}=2\text{ and }R\neq B\ ,\\
				\partial&\text{if }R= B\ .
			\end{cases}
		\end{equation}
		Then
		\begin{equation}\label{eq:coupling-forest-bound}
			\max_{t\le n^{\varepsilon/3}} \|\tilde\cL_t- \cL^{\cF}_t \|_{\rm TV}= o(n^{-\varepsilon/3})\ .
		\end{equation}
	\end{proposition}
	\begin{proof}
		In the same vein of Section \ref{sec: Local structures}, consider the coupling between the two constructions in which in Step (3b) the sample of an head $f$ is made \emph{u.a.r. among all heads}, and rejecting the sample if $f$ is already matched. 
		In case of a rejection, we say that the coupling has \emph{failed}, and the two constructions will then continue independently. Similarly, if the vertex $x$ sampled in Step (2) of the construction has already a matched head or a matched tail, declare the coupling as failed and continue the two constructions independently. Let $\hat \P$ denote the law of such coupling. We now show that, if $t$ is not too large, the probability that the coupling fails goes to zero as $n\to\infty$ at the speed in \eqref{eq:coupling-forest-bound}, which immediately implies the desired result. For a fixed $s\leq t$, consider the following events:
		\begin{align*}
			A_s &= \left\{\substack{\text{$s$ is the first time in which at step (2b) or (3b) is sampled some }f\in E^- \\\text{that was already sampled in step (2b) or (3b) at some time $<s$}}\right\}\,, \\
			B_s &= \left\{\substack{\text{ $s$ is the first time in which a vertex }x\in [n] \text{ which has }\\\text{already a matched tail or head is sampled in step (2)}}\right\},
		\end{align*}
		and define, similarly to the proof of Lemma \ref{lemma:LTL-structure}, the hitting time $\tau$ to be the the first time such that the coupling fails. Note that the only ways in which the coupling can fail at a fixed time $s\leq t$ are the ones stated in the events $A_s$ and $B_s$. Therefore, by the union bound we get
		\begin{equation}
			\label{eq: bounds on failiure forest coupling}
			\begin{split}
				\hat{\P}(\tau\leq t) &\leq \hat{\P}\left(\bigcup_{s\leq t} (A_s \cup B_s)\right) \leq \sum_{s\leq t} \hat{\P}\big(A_s \cup B_s\big) \\
				&\leq t \ \max_{s\leq t} \hat{\P}(A_s) + t \ \max_{s\leq t} \hat{\P}(B_s),
			\end{split}
		\end{equation}
		The first term on the right-hand side of the latter inequality can be bounded as in \eqref{eq: cupling failiure tree}, so that 
		\begin{equation}\label{eq:bound1}
			t \ \max_{s\leq t} \hat{\P}(A_s) \leq t^2 \, \frac{d_{\rm{max}}^-}{m-t}\ .
		\end{equation}
		As for the second term in \eqref{eq: bounds on failiure forest coupling}, we have that
		\begin{equation}\label{eq:bound2}
			t \ \max_{s\leq t} \hat{\P}(B_s) \le  t \ \max_{s\leq t} \sum_{x\in [n]} \mu(x)\mathds{1}_{x\in \Gamma_s} \leq  t^2  \ \max_{x\in[n]}\mu(x) \ ,
		\end{equation}
		where $\Gamma_s$ denotes the number of vertices explored by time $s$, thus $|\Gamma_s|\leq s$.
		Plugging the bounds in \eqref{eq:bound1} and \eqref{eq:bound2} into \eqref{eq: bounds on failiure forest coupling} we get
		\begin{equation}
			\label{eq: bound forest coupling}
			\hat{\P}(\tau\leq t) \leq t^2 \, \frac{d_{\rm{max}}^-}{m-t} + \ t^2 \max_{x\in[n]}\mu(x)\ .
		\end{equation}
		Under our assumptions on $t$ and $\mu$ the right-hand side of \eqref{eq: bound forest coupling} vanishes as $n\rightarrow \infty$, and the desired result follows.
	\end{proof}
	\begin{remark}[Coupling of multiple annealed random walks]\label{rmk:multiple-annealed-2}
		Recall the multiple annealed walk $(\tilde W^{(1)}_t,\dots,\tilde{W}^{(\kappa)}_t)_{t\le T}$ introduced in Remark \ref{rmk:multiple-annealed} and assume that $\kappa\ge 2$ is bounded. The reader should be convinced at this point that the law of such iterated construction can be coupled with the iterated construction of $\kappa$ i.i.d. forests (of polylogarithmic size) at a polynomially small TV-cost.
	\end{remark}

	\subsection{Return to the diagonal for random walks with $\mu$-reset} \label{sec: computation of R_Delta}
	Throughout this section, we set 
	\begin{equation}\label{eq:def-T}
		T=T_n\coloneqq\lfloor \log(n)^6\rfloor\ .
	\end{equation}
	For a fixed probability distribution $\mu=\mu_n$ on $[n]$ we aim at understanding the asymptotic behavior of the random variable with respect to the generation of $\omega$ given by
	\begin{equation}\label{eq: mean- R}
		R^{\mu}=R_{T}^{\omega,\mu}(\partial)\coloneqq \sum_{t=0}^T \tilde{P}_\mu(\partial,\partial)\ .
	\end{equation}
	Start by noting that its expectation can be expressed in terms of the annealed walks defined as in \eqref{eq: annealed law}
	\begin{equation}
		\E[R^\mu]=\sum_{t=0}^T\E[\tilde{P}_\mu(\partial,\partial)]=\sum_{t=0}^T\P^{{\rm an},\mu}(\tilde W_t=\partial)\ .
	\end{equation}
	In what follows, it will turn out to be useful to introduce the following stopping times with respect to the annealed process
	\begin{equation}\label{eq:def-tau-j}
		\begin{split}
			\tau^{(1)}_{\partial} &\coloneqq\inf\{t>0 : \tilde{W_t}=\partial\}\ , \\
			\tau^{(j)}_{\partial} &\coloneqq\inf\{t>\tau^{(j-1)}_{\partial} : \tilde{W_t}=\partial\}\ , \qquad j\geq 1\ .
		\end{split}
	\end{equation}
	Notice that, under the success of the coupling in Proposition \ref{prop:coupling-forest} the stopping times  $\tau^{(j)}_{\partial}$ can be interpreted at the successive times at which the red and the blue particles are on the same vertex.
	Exploiting the coupling with the random forest introduced in Proposition \ref{prop:coupling-forest}, we will start by showing the following quantitative estimate.
	\begin{lemma}\label{lemma: expected first return to delta}
		For every $n$ fix a degree sequence satisfying Assumption \ref{log degree assumptions}. Fix $\varepsilon>0$ and for each $q\in(0,1/2]$ consider the compact set
		\begin{equation}\label{eq:def-cM}
			\cM(q)=\cM_n(q_n,\varepsilon)\coloneqq\left\{\mu\in\cP([n])\mid \sum_{x\in[n]}\mu(x)\frac{1}{d_x^+}=q\ ,\: \max_{x\in[n]}\mu(x)\le n^{-\varepsilon}\right\}\ . 
		\end{equation}
		Then,
		\begin{equation}
			\sup_{\mu\in\cM(q)}\sup_{t\le \frac{T}2}\tilde{\P}^{\rm an,\mu}_\partial(\tau^{(1)}_{\partial}=2t+1) =o(n^{-\varepsilon/3})\ ,
			\label{expected first return to delta-odd}
		\end{equation}
		and
		\begin{equation}
			\sup_{\mu\in\cM(q)}\sup_{t\le \frac{T}2}\left|\tilde{\P}^{\rm an,\mu}_\partial(\tau^{(1)}_{\partial}=2t) - 2^{-2t+1}\,C_{t-1}\, \rho^{t-1}\, q \right|=o(n^{-\varepsilon/3})\ ,
			\label{expected first return to delta}
		\end{equation}
		where 
		\begin{equation}\label{eq:def-catalan}
			C_s \coloneqq \frac{1}{s+1}\binom{2s}{s}\ ,\qquad s\in\N\ ,
		\end{equation}
		is the s-th Catalan number and $\rho$ 
		and $T$ are defined as in \eqref{eq:def-rho-gamma} and \eqref{eq:def-T}, respectively.
	\end{lemma}
	\begin{proof}
		Thanks to Proposition \ref{prop:coupling-forest} it is enough to consider the same probabilities under the event that the coupling succeeds. For this reason we restrict the focus on the construction of the random forest, and let $\hat\P$ denote its law. Accordingly, the stopping time $\tau_\partial^{(1)}$ will be interpreted as explained below \eqref{eq:def-tau-j}. Fix $t\leq T/2$.
		Recall that on the forest process, the two particles are always on the same tree and their distance at time $t$, which we will refer to as $d(t)$, starts at $d(0)=0$ and can only reduce by one unity (if the red particle is in the same ray as the blue particle and does a step in its direction), increase by a unity (if the red particle is in the same ray as the blue particle and the latter does a step forward), or become infinity and keep this value for the rest of the time (if the red particle is in the same ray as the blue particle and creates a new branch).
		Therefore, the two particles can meet only at even times, so that \eqref{expected first return to delta-odd} follows immediately.
		
		Recalling that $d(0)=0$, we now need to compute the probability of the event
		\begin{equation}
			\cE_t=\{d(s)>0\ , \forall s\in\{1,\dots,2t-1\} \}\cap \{d(2t)=0\}\,,\qquad t\le \frac{T}2\,.
		\end{equation}
		Notice that, each simple event in $\cE_t$ corresponds to an evolution in which the red particle follows the blue one up to reaching it at distance $t$ from the root. Therefore, each simple events in $\cE_t$ can be associated uniquely to a couple $(\mathfrak{D}_t,E_t)$ where
		\begin{itemize}
			\item $\mathfrak{D}_t$ is a Dyck path with $t-1$ upstrokes and $t-1$ downstrokes having $\pm1$ increments, representing the evolution of $d(s)$ for $s\le t$;
			\item $E_t=(x_0,\dots,x_{t-1})\in [n]^t$ is the sequence of marks of all the vertices along the path of length $t$ from the root to the vertex at which the particles meet (not included).
		\end{itemize}
		Notice that two simple events in $\cE_t$ associated to the same sequence of marks but to different Dyck paths are equiprobable. Therefore, it is enough to take a representative Dyck path $\mathfrak{D}_t$ and write
		\begin{equation}\label{eq:expression1}
			\hat\P(\tau^{(1)}_{\partial}=2t)=C_{t-1}\sum_{E_t\in[n]^t}\hat{\P}(\mathfrak{D}_t,\ E_t)=
			2^{-2t+1}C_{t-1}\sum_{E_t\in[n]^t}\hat{\P}( E_t\mid \mathfrak{D}_t)\ ,
		\end{equation}
		where we exploited the fact that all Dyck paths are equiprobable, since they only depend on the color of the moving particle at each step, and the fact that the number of Dyck paths of length $2s$ is $C_s$, as in \eqref{eq:def-catalan}. 
		A natural choice for the representative Dyck path consists in choosing $\cD_t$ to by the Dyck path in which the walks move alternately, with the exception of the first two steps, in which the blue particle moves, and the last two step, in which the red particle moves twice. In such a way the distance between them oscillates between 1 and 2 until $\tau_\partial^{(1)}$. Recall the definition of $\lambda^{\rm biased}$ in \eqref{eq: size-biased offspring distribution} and
		consider the collection of independent random variables $D_0,D_1,\dots,D_{t-1}$ where $D_i\sim\lambda^{\rm biased}$ for $i\ge 1$ while $D_0$ is distributed as 
		\begin{equation}\label{eq:def-lambda-mu}
			\lambda^{\mu}(d)\coloneqq \sum_{x\in[n]}\mu(x)\ind(d_x^+=d)\ ,\qquad d\in\N\ .
		\end{equation}
		the out degree of a vertex sample according to $\mu$. Then
		\begin{equation}\label{eq:expression2}
			\begin{split}
				\sum_{E_t\in[n]^t}\hat\P\left(E_t\mid \mathfrak{D}_t\right) 
				&= \sum_{d_0\geq 2}\dots \sum_{d_{t-1}\geq 2}\; \prod_{j=0}^{t-1} \frac{1}{d_j} \,\hat\P(D_0=d_0,D_1=d_1,\dots,D_t=d_{t-1}) \\
				&= \sum_{d_0,d_1,\dots,d_{t-1} \geq 2} \;\prod_{j=0}^{t-1}\frac{1}{d_j}\,\hat\P(D_j=d_j)\\
				&= \left(\sum_{d\geq 2} \frac{1}{d}\,\lambda^\mu(d)\right)\;\left(\sum_{d\geq 2} \frac{1}{d}\,\lambda^{\text{biased}}(d)\right)^{t-1}\ .
			\end{split}
		\end{equation}
		More explicitly, 
		\begin{equation}\label{eq:expression4}
			\sum_{d\geq 2} \frac{1}{d}\,\lambda^{\text{biased}}(d) = \sum_{d\geq 2} \frac{1}{d}\,\sum_{x\in [n]} \frac{d^-_x}{m}\mathds{1}_{d^+_x =d} = \frac1m\sum_x \frac{d_x^-}{d_x^+}=\rho\ ,
		\end{equation}
		and for all $\mu\in\cM(q)$
		\begin{align}\label{eq:expression3}
			\sum_{d\geq 2} \frac{1}{d}\,\lambda^\mu(d) &= \sum_{d\geq 2} \frac{1}{d}\,\sum_{x\in[n]} \mu(x) \mathds{1}_{d_x^+=d}=\sum_{x\in[n]}\mu(x)\frac{1}{d_x^+} = q\ .
		\end{align}
		
		Thus the desired result follows by plugging \eqref{eq:expression2}, \eqref{eq:expression3} and \eqref{eq:expression4} into \eqref{eq:expression1}.
	\end{proof}

		Before presenting our result on the behavior of $\E[R^\mu]$ it is convenient to recall the following classical identity.
		\begin{lemma}\label{lemma:identity}
			For all $\rho\in(0,1/2]$
			\begin{equation}
				\sum_{t\ge 1} 2^{-2t+1}\,C_{t-1}\, \rho^{t-1}=\frac{1-\sqrt{1-\rho}}{\rho}\,.
			\end{equation}
		\end{lemma}
		\begin{proof}
			For all $\rho\in(0,1/2]$, it holds that
			\begin{equation*}
				\sum_{t\ge 1} 2^{-2t+1}\,C_{t-1}\, \rho^{t-1} = \frac{1}{2}   \sum_{t\ge 0} C_{t}\, \bigg(\frac{\rho}{4}\bigg)^{t} = G(\rho/4) =\frac{1-\sqrt{1-\rho}}{\rho}\,,
			\end{equation*}
			where $$ G(x) =\frac{1-\sqrt{1-4x}}{2\,x}, \quad x\in(0,1),$$ is the generating function of the Catalan numbers, (see, e.g., \cite[Ch. 5.4]{GKP89} or \cite[Eq. 24]{FL03}).    
		\end{proof}
		At this point we have essentially all the technical ingredients needed to compute $\E[R^\mu]$.
		\begin{lemma}\label{coro:mean-R-new}
			In the same setting of Lemma \ref{lemma: expected first return to delta}, defined $\Phi:[0,1/2]^2\to \R$ as
			\begin{equation}\label{eq:def-mathfrak-r-old-new}
				\Phi(\rho,q)\coloneqq   \frac{\rho}{\rho -q\left(1-\sqrt{1-\rho}\right)}-1\,,
			\end{equation}
			it holds
			\begin{equation}\label{eq:coro-mean-R}
				\lim_{n\to\infty}\sup_{\mu\in\cM(q)}\left|\E[R^\mu]-(1+\Phi(\rho,q))\right|=0\ ,
			\end{equation}
			where 
			$\rho=\rho_n$ is defined as in \eqref{eq:def-rho-gamma}.
		\end{lemma}
		\begin{proof}
			Call $\hat R^\mu(s,t)$ the expected number of ``reset'' of the forest process, i.e., of the number of meetings of the red and the blue flag, in the interval $[s,t)$. We claim that it is enough to show that
			\begin{equation}\label{eq:goal-forest}
				\lim_{n\to\infty}\hat{R}^\mu(0,T+1)=1+\Phi(\rho,q)\,.
			\end{equation}
			Indeed, thanks to Proposition \ref{prop:coupling-forest},
			\begin{equation}
				\E[R^\mu]=\hat R^\mu(0,T+1)+O(Tn^{-\varepsilon/3})\sim \hat R^\mu(0,T+1)\,,
			\end{equation}
			where we also used that $\E[R^\mu]\le T$,  $T=o(n^{\varepsilon/3})$ and $\hat R^\mu(0,T+1)\ge 1$.
			Recalling that $\hat\tau^{(i)}_\partial$ denotes the time of the $i$-th reset, we notice that process satisfies the renewal property
			\begin{equation}
				\hat\P(\hat\tau^{(i+1)}_\partial=\infty\mid \cF_{i})=1-\sum_{t\ge 1}\hat\P(\hat\tau^{(1)}_\partial=2t)\,,\qquad \forall i\ge 1\,,
			\end{equation}
			where $\cF_{i}$ is the sigma-field generated by the history of the forest process up to time $\hat\tau^{(i)}_\partial$. Therefore,
			\begin{equation}\label{eq:R0infty}
				\hat R^\mu(0,\infty)=\frac{1}{1-\sum_{t\ge 1}\hat \P(\hat\tau^{(1)}_\partial=2t)}=1+\Phi(\rho,q)\,,
			\end{equation}
			where the last equality follows from Lemmas \ref{lemma: expected first return to delta} and \ref{lemma:identity}. We claim that, there exists some universal constant $C>0$ such that for all $t,s\ge 1$ 
			\begin{equation}\label{eq:hatR}
				\begin{split}
					\hat R^\mu(t,t+s)&\le s\, \bigg( 2^{-t/3}+\hat\P\big({\rm Bin}(t,1/2)<\tfrac{t}3   \big)\bigg)\le s\, \big( 2^{-t/3}+e^{-C t^2}  \big) \,.
				\end{split}
			\end{equation}
			The second inequality in \eqref{eq:hatR} is a trivial consequence of the Chernoff bound. To see the validity of the first inequality, observe that the number of resets in $(t,t+s)$ is at most $s$ almost surely and, in order to be different from zero, it must be the case that the chasing walk never deviates from the escaping walk up to time $t$. At each step the chasing walk is selected with probability $1/2$ and, given that it did not deviate yet, it deviates with probability at least $1/2$, regardless of $\mu$. Hence, either the chasing walk is selected less than $t/3$ times or, otherwise, the chance that it never deviates up to time $t$ is at most $2^{-t/3}$; thus \eqref{eq:hatR} follows.
			Therefore, for $n$ sufficiently large,
			\begin{equation}\label{eq:RTinfty}
				\hat R^\mu(T,\infty)=\sum_{k\ge 1}\hat R(kT,(k+1)T)\le \sum_{k\ge 1}T2^{-kT/4}=O\big(T2^{-T/4}\big)=o(1)\,,
			\end{equation}
			and \eqref{eq:goal-forest}
			follows from \eqref{eq:R0infty} and \eqref{eq:RTinfty}. \end{proof}

	At this point, it is enough to use Lemma  \ref{coro:mean-R-new} and some concentration result for the random variable $R^\mu$ to conclude the validity of the following proposition, which is the main result of this section.
	\begin{proposition} \label{prop:R convergence mu}
		In the same setting of Lemma \ref{lemma: expected first return to delta},
		\begin{equation}
			\sup_{\mu \in \cM(q)}\Big| R^{\mu} -  \left(1 + \Phi(\rho,q)\right)\Big|\overset{\P}{\longrightarrow}0\, ,
		\end{equation}
		where $\rho=\rho_n$ and $\Phi$ are defined as in  \eqref{eq:def-rho-gamma} and \eqref{eq:def-mathfrak-r-old-new}, respectively.
	\end{proposition}
	
	\begin{proof}
		We only need to study the second moment of $R^\mu$. We can rewrite
		\begin{align*}
			\E\left[ \left( R^\mu \right)^2 \right] =& \sum_{s,t \le T} \E\left[\tilde{\mathbf{P}}_\partial(\tilde{X}_{2t}=\partial) \tilde{\mathbf{P}}_\partial(\tilde{X}_{2s}= \partial)\right] 
			= \sum_{s,t \le T}  \tilde\P^{2\text{-an},\mu} \left(\tilde{W}^{(1)}_{2t}=\partial, \ \tilde{W}^{(2)}_{2s}=\partial \right)\ ,
		\end{align*}
		where the law $\tilde\P^{2\text{-an},\mu}$ is defined as in Remark \ref{rmk:multiple-annealed}. As pointed out in Remark \ref{rmk:multiple-annealed-2}, we can couple at a polynomially small TV-cost the construction of the two annealed walks with that of two i.i.d. forests. Therefore,
		\begin{equation*}
			\begin{split}
				\sum_{s,t \le T}  &\tilde\P^{2\text{-an},\mu} \left(\tilde{W}^{(1)}_{2t}=\partial, \ \tilde{W}^{(2)}_{2s}=\partial \right)\\
				&\le o(T^2 n^{-\frac{\varepsilon}3})+\sum_{s,t \le T} 2^{-(2t+2s)} \sum_{k_1=1}^t \sum_{k_2=1}^s  q^{k_1+k_2}\ \rho^{t+s-k_1-k_2}\ B(t,k_1)\ B(s,k_2)\\
				&=(1+o(1))\ \E[R^\mu]^2\ .
			\end{split}
		\end{equation*}
		It follows that $$\E\left[ \left( R^\mu\right)^2 \right] \sim \E\left[ R^\mu \right]^2 \ ,$$
		thus $\rm{Var}(R^\mu) = o(1)$. Therefore, by Chebyshev inequality, we can conclude that for any $\delta>0$
		\begin{align*}
			\P\left( | R^\mu -  \E [R^\mu] | > \delta \right) &\leq\frac{ \rm{Var}(R^\mu)}{\delta^2}
			=o(1)\ ,
		\end{align*}
		Hence, $\left|R^\mu - \E [R^\mu]\right|\overset{\P}{\longrightarrow}0$ and the desired conclusion follows by Lemma \ref{coro:mean-R-new}.
	\end{proof}

	\section{Return to the diagonal for random walks with $\tilde\mu$-reset}\label{sec:random-mu}
	From Proposition \ref{prop:R convergence mu}, we see that when considering the random walk with $\mu$-reset the only feature of the law $\mu$ that determines the value of $R^\mu$ is the expectation of the inverse out-degree of a vertex sampled according to $\mu$. In this section, we show how to adapt this result to the case in which the distribution $\mu$ is replaced by the random distribution $\tilde \mu$ defined in \eqref{eq:def-tilde-mu}. The argument can be ideally divided into two parts: first, we show that the expectation of the inverse out-degree of a vertex sampled according to $\tilde\mu$ converges in probability to the value $\mathfrak{q}$ in \eqref{eq:def-mathfrak-q}; second, we use a continuity argument to conclude that $R^{\tilde{\mu}}$ converges in probability to $\mathfrak{r}=1+\Phi(\rho,\mathfrak{q})$, where $\Phi$ is defined as in \eqref{eq:def-mathfrak-r-old-new}. This concludes the proof of Proposition \ref{prop:return-diag}.
	
	Concerning the first part of the above mentioned argument, we prove a more general result, showing that actually the expectation of any bounded function of the degrees of a vertex sampled according to $\tilde\mu$ converges in probability to an explicit function of the degree sequences. This fact is formalized in Proposition \ref{prop:general-formula}. Before proceeding with the statement, it is convenient to provide the following definition.
	
	\begin{definition}\label{def:g}
		For every $n\in\N$ let $(\mathbf{d}^+,\mathbf{d}^-)$ be a degree sequence satisfying Assumption \ref{log degree assumptions}. Consider the function $g=g_n:\N_0\times\N_{\ge 2}\to \R$ defined as follows: for all $d^-\ge0$ and $d^+\ge 2$
		\begin{equation} \label{eq:def g_n}
			\begin{split}
				g(d^-,d^+)
				=&\frac{d^-|V_{d^-,d^+}|}{m^2}\left(d^- + \frac{\gamma-\rho}{1-\rho}-1\right) \ ,    
			\end{split}
		\end{equation}
		where $\rho=\rho_n$ and $\gamma=\gamma_n$ are defined in \eqref{eq:def-rho-gamma}, and  $V_{d^-,d^+}$ denotes the set of vertices having in- and out-degree equal to $d^-$ and $d^+$, respectively.
	\end{definition}
	\begin{remark}
		Notice that, thanks to Assumption \ref{log degree assumptions}
		\begin{equation}\label{eq:remark-mathfrak-p}
			n\ \sum_{d^-\ge 0}\sum_{d^+\ge 2}g(d^-,d^+)=\mathfrak{p}=\Theta(1)\ ,
		\end{equation}
		where $\mathfrak{p}$ has been defined in \eqref{eq:def-mathfrak-p}.
		Moreover, if the graph is Eulerian, then $\gamma=1$, and the function is non zero only on the diagonal, where
		\begin{equation}
			g_n(d,d)=\frac{d^2|V_d|}{m^2}\ .
		\end{equation}
		If in particular the graph is regular of degree $d$, the latter is the indicator function at $d$, divided by $n$. 
	\end{remark}
	\begin{proposition}\label{prop:general-formula}
		For every $n\in\N$ let $(\mathbf{d}^+,\mathbf{d}^-)$ be a degree sequence satisfying Assumption \ref{log degree assumptions}. Then, for every bounded function $f:\N_0\times \N_{\ge 2}\to \R$ it holds that
		\begin{equation}\label{eq:prop-gen-formula}
			\left|n\sum_{x\in[n]}\pi^2(x)f(d^-_x,d^+_x) - n\sum_{d^-\ge 0}\sum_{d^+\ge 2}g(d^-,d^+) f(d^-,d^+)\right|\overset{\P}{\longrightarrow}0\ ,
		\end{equation}
		where $g=g_n$ is as in Definition \ref{def:g}.
	\end{proposition}
	The proof of Proposition \ref{prop:general-formula} is postponed to Section \ref{suse:proof-general-formula}. We now show how to derive from it the proof of Proposition \ref{prop:return-diag}.
	\begin{proof}[Proof of Propositions \ref{prop:pi-diag} and \ref{prop:return-diag}]
		By Proposition \ref{prop:general-formula}, choosing $f(d^-,d^+)\equiv 1$ (cf. \eqref{eq:remark-mathfrak-p}) we have
		\begin{equation}
			\left|n\sum_{x\in[n]}\pi^2(x) - \mathfrak{p}\right|\overset{\P}{\longrightarrow}0\ .
		\end{equation}
		Similarly, choosing $f(d^-,d^+)=\mathbf{1}_{d^+\ge 2}/d^+$, 
		\begin{equation}
			\left|n\sum_{x\in[n]}\pi^2(x)\frac{1}{d_x^+} - n\sum_{d^-\ge 1}\sum_{d^+\ge 2}g(d^-,d^+)\frac{1}{d^+}\right|\overset{\P}{\longrightarrow}0\ .
		\end{equation}  
		In conclusion, by Definition \ref{def:g} and \eqref{eq:def-mathfrak-q} we have
		\begin{equation}\label{eq:conv-to-mathfrak-q}
			\left|\frac{1}{\pl \pi^2\pr }\sum_{x\in[n]}\pi^2(x)\frac{1}{d_x^+} - \mathfrak{q} \right|\overset{\P}{\longrightarrow}0\ .   
		\end{equation}
		Recall the definition of the sets $(\cM(q))_{q\ge 0}$ in \eqref{eq:def-cM}. Similarly, for any $\varepsilon>0$, define the set
		\begin{equation}\label{eq:M-epsilon definition}
			\begin{split}
				\cM_\varepsilon(q) &:= \left\{\mu\in\cP([n]) \mid \sum_{x\in[n]} \mu(x)\frac{1}{d^+_x} \in \left[q-\varepsilon,q+\varepsilon\right]\right\} \ .
			\end{split}
		\end{equation}
		By \eqref{eq:conv-to-mathfrak-q}, for any $\varepsilon>0$, if
		\begin{equation*}
			\cG_\varepsilon=\{\tilde{\mu}\in \cM_\varepsilon(\mathfrak{q})\}\ ,
		\end{equation*}
		then we have $\P(\cG_\varepsilon)=1-o(1)$.
		As a consequence, 
		\begin{align*}
			\P\left(\Big|R^{\tilde{\mu}} -(1+ \Phi(\rho,\mathfrak{q}))\Big|>\varepsilon\right) &= \P\left(\Big|R^{\tilde{\mu}} -(1+ \Phi(\rho,\mathfrak{q}))\Big|>\varepsilon\; ,\; \cG_\varepsilon\right) + o(1)\\
			&\le \sup_{\mu\in\cM_\varepsilon(\mathfrak{q})}\P\left(\Big|R^{\mu} -(1+ \Phi(\rho,\mathfrak{q}))\Big|>\varepsilon\right) + o(1)\ ,
		\end{align*}
		hence, the desired result follows by the continuity of the function $\Phi$ in its second variable, which can be checked immediately thanks to representation in Lemma \ref{lemma:identity}.
	\end{proof}

	\subsection{Proof of Proposition \ref{prop:general-formula}}\label{suse:proof-general-formula}
	We begin this section by explaining the strategy of proof and all the required technical ingredients. We take inspiration from \cite[Lemma 4.1]{CQ21a} and \cite[Lemma 4]{QS23}, where the authors where interested only in determining a first order bound on the quantity $\pl \pi^2\pr$ for the random walk on the DCM and on the related random Deterministic Finite Automata (DFA) model, respectively.
	
	The first ingredient in the proof is an approximation of $\pi$ with the $T$ (defined as in \eqref{eq:def-T}) step evolution\footnote{In \cite{CQ21a,QS23}, $T$ is chosen $\log^3(n)$, which is enough to guarantee that the $\ell^\infty$ distance between $\pi$ and $\lambda$ is $o(n^{-c})$ for every $c>0$. Since the bounds we are going to show are not affected by polylogarithmic factors, we prefer to use the same $T$ as in \eqref{eq:def-T}, in order to avoid the introduction of further notation.} of the simple random walk started at the uniform distribution, 
	\begin{equation}\label{eq:def-mu-T}
		\mu_T(x)=\mu_{n,T}(x)\coloneqq \frac1n \sum_{y\in[n]}P^T(y,x)\,\qquad x\in[n]\,.
	\end{equation}
	Indeed, as a corollary of Theorem \ref{th:cutoff}, we have (see also \cite[Corollary 3.7]{CCPQ21}) that the event
	\begin{equation}\label{eq:pi-close-P^T}
		\cH=\cH_n\coloneqq\left\{\max_{x\in[n]}\left|\pi(x)-\frac{1}{n}\sum_{y\in[n]}P^T(y,x) \right|\le e^{-\log(n)^{3/2}}\right\}\,,
	\end{equation}
	satisfies
	\begin{equation}\label{eq:event-E-whp}
		\P(\cH)=1-o(1)\,.
	\end{equation}
	For each $d^-\ge 0$ and $d^+\ge 2$ and any bounded function $f:\N_0\times\N_{\geq2}\to\R$ we consider the random variables
	\begin{equation}
		Z_{d^-,d^+}=\sum_{x\in V_{d^-,d^+}}\pi(x)^2 f(d^-,d^+)\, . 
	\end{equation}
	Notice that the quantity of interest in \eqref{eq:prop-gen-formula} can be rewritten as
	\begin{equation}
		\sum_{x\in[n]}\pi(x)^2 f(d_x^-,d_x^+)=\sum_{d^-\ge 0}\sum_{d^+\ge 2}Z_{d^-,d^+}\, .
	\end{equation}
	We also consider the collection of random variables
	\begin{equation}\label{eq:def-Y}
		Y_x\coloneqq\frac{1}{n^2}\sum_{y,z\in[n]}P^T(y,x)P^T(z,x)f(d^-_x,d^+_x)\,,\qquad x\in[n]\, .
	\end{equation}
	By definition of the event $\cH$ in \eqref{eq:pi-close-P^T}, for all $n$ sufficiently large,
	\begin{equation}\label{eq:under-E}
		\ind(\cH)\left| \sum_{d^-\ge 0}\sum_{d^+\ge 2}Z_{d^-,d^+}-\sum_{x\in[n]}Y_x \right|\le e^{-\log(n)^{4/3}}\ ,\qquad \P-{\rm a.s.}\ .
	\end{equation}
	Hence, the proof of \eqref{eq:prop-gen-formula} comes as a consequence of the following convergence
	\begin{equation}\label{eq:main-convergence}
		\widehat F=\widehat{F}_n\coloneqq \left|n\sum_{x\in[n]}Y_x- n\sum_{d^-\ge 0}\sum_{d^+\ge 2}g(d^-,d^+) f(d^-,d^+)\right|\overset{\P}{\longrightarrow}0\ .
	\end{equation}
	Indeed, calling
	\begin{equation}
		F=F_n\coloneqq \left|n\sum_{x\in[n]}\pi^2(x)f(d^-_x,d^+_x) - n\sum_{d^-\ge 0}\sum_{d^+\ge 2}g(d^-,d^+) f(d^-,d^+)\right|\,,
	\end{equation}
	thanks to \eqref{eq:event-E-whp} and \eqref{eq:under-E}, we have for all $\varepsilon>0$
	\begin{equation}
		\P(F>\varepsilon)\le \P(F>\varepsilon\,,\cH)+\P(\cH^c)\le\P\left(\widehat F>\varepsilon/2\,,\cH\right)+o(1)\le \P\left(\widehat F>\varepsilon/2\right)+o(1)\,.
	\end{equation}
	Therefore, the rest of the proof is devoted to establishing the convergence in \eqref{eq:main-convergence}. To this aim, we will use the second moment method for the random variable $n\sum_{x\in[n]}Y_x$, showing that
	\begin{equation}\label{eq:second-mom-method-1}
		\E\bigg[n\sum_{x\in[n]}Y_x \bigg]=(1+o(1))\,n\sum_{d^-\ge 1}\sum_{d^+\ge 2}g_n(d^-,d^+) f(d^-,d^+)\,,
	\end{equation}
	and
	\begin{equation}\label{eq:second-mom-method-2}
		\E\bigg[\big(n\sum_{x\in[n]}Y_x\big)^2 \bigg]=(1+o(1))\,\E\bigg[n\sum_{x\in[n]}Y_x \bigg]^2\,,
	\end{equation}
	from which \eqref{eq:main-convergence} follows immediately by Chebyshev inequality.
	Let us start by computing the expectation
	\begin{equation}\label{eq:expect}
		\begin{split}
			\E\bigg[\sum_{x\in[n]}Y_x \bigg]=&\frac{1}{n^2}\sum_{y,z\in[n]}\E\big[P^T(y,x)P^T(z,x)\big]f(d^-_x,d^+_x)
			=\sum_{x\in[n]}\E\big[\mu_T(x)^2\big]f(d^-_x,d^+_x)\,,
		\end{split}
	\end{equation}
	where we recall to the reader the definition of $\mu_T$ is in \eqref{eq:def-mu-T}.
	Note that, due to the symmetry of the model, if $x,x'\in V_{d^-,d^+}$ for some $d^-\ge 0$ and $d^+\ge 2$, then
	\begin{equation}
		\E\big[\mu_T(x)\big]=\E\big[\mu_T(x')\big]\, .
	\end{equation}
	Therefore, \eqref{eq:expect} can be rewritten as
	\begin{align}\label{eq:exp-Y}
		\E\bigg[\sum_{x\in[n]}Y_x \bigg]=&\sum_{d^-\ge 1}\sum_{d^+\ge2}|V_{d^-,d^+}|\,f(d^-,d^+)\,\E\bigg[\mu_T\big(x(d^-,d^+)\big)^2\bigg]\,,
	\end{align}
	where $x(d^-,d^+)$ denotes an arbitrary representative of the vertices with degrees $(d^-,d^+)$. Notice that we can neglect the case $d^-=0$, since for every $x$ such that $d_x^-=0$ it holds $\mu_T(x)=0$, $\P-{\rm a.s.}$
	
	We will show the validity of the following lemma, from which \eqref{eq:second-mom-method-1} follows immediately due to \eqref{eq:exp-Y} and the definition of $g$ in \eqref{eq:def g_n}.
	\begin{lemma}\label{lemma:exp-lambda}
		Fixed $d^-\ge 1$, $d^+\ge 2$ and $x\in V_{d^-,d^+}$ it holds
		\begin{equation}\label{eq:exp-lambda-2}
			\E\big[\mu_T(x)^2\big]=(1+o(1))\,\frac{d^-}{m^2}\left(d^- + \frac{\gamma-\rho}{1-\rho}-1\right)\,.
		\end{equation}
	\end{lemma}
	\begin{proof}
		We use again an \emph{annealing argument}, this time for the simple random walk. We start by defining the annealed walks. Notice that
		\begin{equation}\label{2-walks}
			\E\big[\mu_T(x)^2\big]=\P^{2\text{-an}}_{\rm unif}(W_t^{(1)}=W_t^{(2)}=x)\,,\qquad x\in [n]\,,
		\end{equation}
		where $\P^{2\text{-an}}_{\rm unif}$ is the law of a two random walk, $(W^{(1)}_t,W^{(2)}_t)_{t\le T}$ running one after the other, both starting (independently) at a uniform vertex, and creating the environment along with their trajectory, as explained in Remark \ref{rmk:multiple-annealed-4}. More precisely, samples according to $\P^{2\text{-an}}_{\rm unif}$ can be obtained as follows:
		\begin{itemize}
			\item[(1)] sample $W^{(1)}_0$ u.a.r. in $[n]$;
			\item[(2)] for all $s\in\{1,\dots,T\}$, choose $e\in E_{W_{s-1}}^+$ u.a.r.\,:
			\begin{itemize}
				\item [(2a)] if $e$ is already matched to some $f\in E^-$, call $v_f$ the vertex incident to the head $f$, and set $W_s^{(1)}=v_f$;
				\item [(2b)] if $e$ is unmatched, choose u.a.r. some $f\in E^-$ which is still unmatched, set $\omega(e)=f$, call $v_f$ the vertex incident to $f$ and set $W_s^{(1)}=v_f$;
			\end{itemize}
			\item[(3)] once $(W^{(1)}_t)_{t\le T}$ has been sampled, sample $(W^{(2)})_{t\le T}$ in the same way, starting from the partial environment constructed by trajectory of $W^{(1)}$.
		\end{itemize}
		Similar to what was done in Section \ref{suse:annealed-forest}, we now aim at coupling steps (1) and (2) of the previous randomized algorithm with an i.i.d. process.
		
		Let $(U_t)_{t\leq T}$ be an independent sequence of random variables in $[n]$ such that $U_0$ is uniformly distributed on $[n]$, and for all $t\ge 1$, $U_t\overset{d}{=}\mu_{\rm in}$, with $\mu_{\rm in}$ as in \eqref{eq:def-muin}.
		Such an i.i.d. process can  be coupled with the annealed walk $(W^{(1)}_t)_{t\le T}$ in a joint probability space $\Q$ 
		that can be described as follows: 
		\setlength{\leftmargini}{0.5cm}
		\setlength{\leftmarginii}{0.5cm}
		\begin{enumerate}
			\item let $W_0^{(1)}=U_0$;
			\item for $t\ge 1$, sample $U_t\sim\mu_{\rm in}$:
			\begin{itemize}
				\item[(ii a)] if $U_t\neq U_j$ for all $j<t$, sample $e\in E_{W^{(1)}_{t-1}}^+$ u.a.r. and match it with a uniform random head $f\in E_{U_t}^-$ and declare $W^{(1)}_t=U_t$;
				\item[(ii b)] if $U_t= U_s$ for some $s<t$, select a uniform random head $f\in E_{U_{t}}^-$:
				\begin{itemize}
					\item if $f$ is already matched, declare the coupling as \emph{failed}, and continue the two constructions independently;
					\item if $f$ is still unmatched, sample a uniformly random tail $e\in E^+_{W_{t-1}^{(1)}}$, match it with $f$ and set $W_{t}^{(1)}=W_{s}^{(1)}=U_t=U_s$,  declare the coupling as \emph{failed} and continue the two constructions independently.
				\end{itemize}
			\end{itemize}
		\end{enumerate}
		After the realization of the coupled process just described, the trajectory of the second annealed walk, $(W_t^{(2)})_{t\le T}$ will be sampled according to the usual procedure, conditioned on $(W_t^{(1)})_{t\le T}$. We will use the same symbol $\Q$ the refer to such an enlarged probability space, sufficiently rich to include the coupled generation of $(U_t)_{t\le T}$ and $(W_t^{(1)})_{t\le T}$ and that of $(W_t^{(2)})_{t\le T}$.
		Call
		\begin{equation}\label{eq:def-event-D}
			\cD\coloneqq\{W_T^{(1)}=x \}\,,
		\end{equation}
		and $\cF$ the event in which the coupling fails. We start by showing that
		\begin{equation}\label{eq:prob-D}
			\Q(\cD)=(1+o(1))\frac{d^-}{m}\,.
		\end{equation}
		and 
		\begin{equation}\label{eq:bound-cond-prob-coupling}
			\Q(\cF\mid \cD)=O\left(  \frac{T^3\ d^-_{\rm max}}{n} \right)\, .
		\end{equation}
		
		Call $\cE_x$ the event in which the vertex $x$ is visited by the first annealed trajectory more than once, i.e.,
		\begin{equation}
			\cE_x\coloneqq \{\exists s,t\le T \ \text{s.t.}\ W^{(1)}_s=W^{(1)}_t=x\,,\ s\neq t \}\,.
		\end{equation}
		There are two possible ways in which the latter event can happen: either $W_0^{(1)}=x$, and at time $t<T$ vertex $x$ is visited again; this happens with probability at most
		\begin{equation*}
			\frac{1}{n}\ \frac{d^-}{m-T}\ T\,.
		\end{equation*}
		On the other had, if the initial position of $W^{(1)}$ is different from $x$, in order to realize $\cE_x$, $x$ must be visited once by matching one of its heads and, subsequently, the walk has to visit again one of the vertices already visited. The probability of this event can be bounded by
		\begin{equation*}
			\frac{d^-}{m-T} \ \frac{d^-_{\max}-1}{m-T} \ T^2\,.
		\end{equation*}
		Combining the two bounds we get, for all $n$ large enough, \begin{equation}\label{eq:bound-p-1}
			\P^{\rm an}_{\rm unif}(\cE_x)=O(p_1)\,,\qquad\text{where}\qquad p_1\coloneqq\frac{T^2\ d^-\ d^-_{\max}}{n^2} \ .
		\end{equation}
		Therefore
		\begin{equation}
			\P^{2\text{-an}}_{\rm unif}(\cD)=\P^{2\text{-an}}_{\rm unif}(\cD\cap \cE_x^c)+O(p_1)\ .
		\end{equation}
		Consider also the event $\bar\cE_x$ in which there exists some other vertex that is visited more than once, i.e.,
		\begin{equation*}
			\bar\cE_x\coloneqq \cup_{y\in[n]\setminus x}\ \cE_y\,.
		\end{equation*}
		Hence,
		\begin{equation}\label{eq:bound-p-2}
			\P^{2\text{-an}}_{\rm unif}(\cD)=\P^{2\text{-an}}_{\rm unif}(\cD\cap\cE_x^c\cap \bar\cE_x^c)+\P^{2\text{-an}}_{\rm unif}(\cD\cap\cE_x^c\cap \bar\cE_x)+O(p_1)\ .
		\end{equation}
		We can further bound 
		\begin{equation}\label{eq:bound-p-3}
			\P^{2\text{-an}}_{\rm unif}(\cD\cap \cE_x^c\cap \bar\cE_x)=O(p_2)\,,\qquad\text{where}\qquad p_2\coloneqq\frac{T^3\ d^-\ d^-_{\max}}{n^2}\, .
		\end{equation}
		Indeed, 
		the probability that $y\neq x$ is visited at least twice by the annealed walk can be bounded analogously to \eqref{eq:bound-p-1}. Moreover, under this event it has to visit $x$ at time $T$ for the first time. Hence by a union bound
		\begin{equation}
			\P^{2\text{-an}}_{\rm unif}(\cD\cap\cE_x^c\cap \bar\cE_x)\le  \  \sum_{y\in[n]}T\ \frac{d_{y}^-}{m-T}\times T^2\ \frac{d_{\max}^-}{m-T}\times \frac{d^-}{m-T}\,,
		\end{equation}
		from which \eqref{eq:bound-p-3} follows.
		
		It follows that
		\begin{equation}\label{eq:bound-p-4}
			\P^{2\text{-an}}_{\rm unif}(\cD)=\P^{2\text{-an}}_{\rm unif}(\cD\,,\: \cE_x^c\cap \bar\cE_x^c)+O(p_1+p_2)\,.
		\end{equation}
		Finally, notice that
		\begin{equation}\label{eq:easy1}
			\P^{2\text{-an}}_{\rm unif}(\cD\cap \cE_x^c\cap \bar\cE^c_x)\le \frac{d^-}{m-T} =(1+o(1))\frac{d^-}{m}\,,
		\end{equation}
		and
		\begin{equation}\label{eq:easy2}
			\P^{2\text{-an}}_{\rm unif}(\cD\cap \cE_x^c\cap \bar\cE_x^c)\ge \left(1-\frac{Td_{\max}^-}{m-T} \right)^T\frac{d^-}{m-T} =(1+o(1))\frac{d^-}{m}\,.
		\end{equation}
		From \eqref{eq:easy1} and \eqref{eq:easy2} it follows that
		\begin{equation}\label{eq:from-which}
			\P^{2\text{-an}}_{\rm unif}(\cD\cap \cE_x^c\cap \cE^c)=(1+o(1))\frac{d^-}{m}\, .
		\end{equation}
		Then \eqref{eq:prob-D} follows by plugging \eqref{eq:from-which} into \eqref{eq:bound-p-4} and noting that $p_1,\, p_2=o(n^{-1})$.
		
		At this point, to prove \eqref{eq:bound-cond-prob-coupling} it suffices to note that
		\begin{equation}
			\Q(\cF\mid \cD)=\frac{\Q(\cD\cap(\cE_x\cup\bar\cE_x) )}{\Q(\cD)}= O\left(\frac{n}{d^-}\ (p_1+p_2)\right)=O\left( \frac{  T^3\  d^-_{\max}}{n} \right)\,,
		\end{equation}
		where we used \eqref{eq:prob-D} to estimate the denominator and for the numerator we bounded
		\begin{equation}
			\Q(\cD\cap(\cE_x\cup\bar\cE_x) )\le \Q(\cE_x )+\Q(\cD\cap\bar\cE_x\cap\cE_x^c )\,,
		\end{equation}
		and then used \eqref{eq:bound-p-1} and \eqref{eq:bound-p-3}.
		
		We go back to estimating the right hand side of \eqref{2-walks} using the coupling.
		Call $\cG$ the event in which the second walk enters the trajectory of the first one, follows it, and at a certain time exits it, i.e.
		\begin{equation*}
			\cG \coloneqq\bigcup_{0\le s<s'<s''\le T} \big( \cA_s\cap \cB_{s'}\cap \cC_{s''}\big)\,,
		\end{equation*}
		where 
		\begin{equation*}
			\cA_s= \{W^{(2)}_s \in \{W^{(1)}_t\}_{t\le T}\}\,, \ \cB_{s'} = \{W^{(2)}_{s'} \not\in \{W^{(1)}_t\}_{t\le T}\}\,,\ \cC_{s''} = \{W^{(2)}_{s''} \in \{W^{(1)}_t\}_{t\le T}\}\,.
		\end{equation*}
		For any $\sigma \in \cD$ we have that
		\begin{equation*}
			\Q(\cG \mid \sigma) \leq \sum_{0\le s<s'<s''\le T} \Q(\cA_s\mid \sigma)\ \Q(\cB_{s'} \mid \cA_s\,,\, \sigma)\ \Q(\cC_{s''}\mid \cB_{s'}\,,\, \cA_s\,,\, \sigma)\,,
		\end{equation*}
		where the single probabilities can be bounded as follows
		\begin{equation*}
			\begin{split}
				\Q(\cA_s\mid \sigma) &\leq  \frac{T}{n}+T^2\, \frac{d^-_{\rm max}}{m-2T}
				\,,\\
				\Q(\cB_{s'} \mid \cA_s,\: \sigma) &\le  1\,, \\
				\Q(\cC_{s''}\mid \cB_{s'}\cap \cA_s,\: \sigma) &\leq T^2\ \frac{d^-_{\rm max}}{m-2T}
				\,.
			\end{split}
		\end{equation*}
		Hence, taking the supremum over $\sigma\in\cD$,
		\begin{equation}\label{eq:sup-sigma}
			\sup_{\sigma \in \cD}\Q(\cG \mid \sigma)= O\left(\frac{T^5 \ (d_{\rm max}^-)^2}{n^2}\right)\ .
		\end{equation}
		To ease the reading,  for  $t\in[0,T]$ and $\sigma \in \cD$ let us consider the function $\sigma\mapsto q_t(\sigma)\in[0,1]$ such that
		\begin{align}
			q_t(\sigma):=&\ \Q(\cG^c,\ W^{(2)}_{t-1}\neq W^{(1)}_{t-1},\ W^{(2)}_s=W^{(1)}_s,\:s\in[t,T]\mid \sigma)\\
			=&\ \Q(W_{s'}^{(2)}\not\in\{W^{(1)}_t\}_{t\le T},\ s'\in[0,t)\text{ and }W^{(2)}_s=W^{(1)}_s,\ s\in[t,T]\mid \sigma)\,,
		\end{align}
		where, by convention, we consider $\{W^{(1)}_{-1}\neq W^{(2)}_{-1} \}$ as the sure event.
		Further, define the same probability conditioned to the event $\cD$ as 
		\begin{equation}
			q_t\coloneqq\Q(W_{s'}^{(2)}\not\in\{W^{(1)}_t\}_{t\le T},\ s'\in[0,t)\text{ and }W^{(2)}_s=X^{(1)}_s,\ s\in[t,T]\mid \cD)\,,
		\end{equation}
		so that
		\begin{equation}
			q_t=\sum_{\sigma\in\cD}\Q(\sigma\mid\cD)\ q_t(\sigma)=\hat{\E}[q_t(\sigma)\mid \cD]\,,
		\end{equation}
		where $\hat{\E}$ is the expectation in the conditional probability space $\Q(\:\cdot\mid \cD)$. 
		As a consequence of \eqref{eq:sup-sigma}, for every $\sigma\in\cD$,
		\begin{equation}\label{eq:exp-2-ub}
			\Q(W^{(2)}_T=x\mid \sigma)= \sum_{t=0}^T q_t(\sigma)+O\left(\frac{T^5\ (d_{\rm max}^-)^2}{n^2}\right)\,.
		\end{equation}
		Therefore, using Assumption \ref{degree assumptions} we get
		\begin{equation}\label{eq:cond-exp}
			\Q(W^{(1)}_T=W^{(2)}_T=x)= \Q(\cD)\left( \hat{\E}\bigg[\sum_{t=0}^Tq_t(\sigma)\mid \cD \bigg] + o(n^{-1}) \right)\,.
		\end{equation}
		Hence Lemma \ref{lemma:exp-lambda} follows at once by \eqref{eq:prob-D} and \eqref{eq:cond-exp} as soon as we prove that
		\begin{equation}\label{eq:goalll}
			\sum_{t=0}^Tq_t =(1+o(1)) \frac{1}{m}\left(d^- -1+\frac{\gamma-\rho}{1-\rho} \right)\, .
		\end{equation}
		For every $t\leq T$ and any $\sigma\in\cD$, define
		\begin{align*}
			\delta_t(\sigma) := \Q\left(\bigcap_{s=0}^{t-1} W_{s'}^{(2)}\not\in\{W^{(1)}_t\}_{t\le T}\mid\sigma\right)
			&\geq  \left(1-\frac{T}{n}\right)  \left(1 - \frac{3\,T\,d^-_{\rm max}}{m}\right)^{T} 
			= 1-o(1)\,,
		\end{align*}
		where the inequality holds for all $n$ large enough. Notice also that, calling $d^\pm_s=d_{X_s^{(1)}}^\pm$ for all $s\in\{0,\dots, T\}$, one has
		\begin{equation*}
			q_t(\sigma) = \begin{cases}
				\delta_T(\sigma)\frac{d^--1}{m-2T}=(1+o(1))\,\frac{d^--1}{m}&\text{if }t=T\,,\\
				\delta_t(\sigma)\frac{d_{t}^--1}{m-T-t}\prod_{s=t}^{T-1}\frac{1}{d^+_{s}}=(1+o(1))\,\frac{d_{t}^--1}{m}\prod_{s=t}^{T-1}\frac{1}{d^+_{s}} &\text{if }t\in\{0,\dots,T-1\}\,,
			\end{cases}
		\end{equation*}
		in particular
		\begin{equation}\label{eq:q-T}
			q_T=\hat{\E}[q_T(\sigma)\mid \cD]=(1+o(1))\ \frac{d^--1}{m}\,.
		\end{equation}
		For $t<T$, we split
		\begin{equation}\label{eq:sum-q}
			\hat{\E}[q_t(\sigma)\mid\cD]=\hat{\E}[q_t(\sigma)\ind(\cF^c)\mid\cD]+\hat{\E}[q_t(\sigma)\ind(\cF)\mid\cD]\,.
		\end{equation}
		To bound the above two terms on the rhs we use \eqref{eq:bound-cond-prob-coupling} and Assumption \ref{log degree assumptions}. The first term is bounded as follows 
		\begin{equation}\label{eq:bound1-q}
			\hat{\E}[q_t(\sigma)\ind(\cF)\mid\cD]\le\frac{d_{\max}^-}{m}\ 2^{-T+t} \ \Q(\cF\mid \cD) =O\left(T^3\ \frac{(d_{\rm max}^-)^2}{n^2} 2^{-T+t}\right)= o\big( n^{-1}\times 2^{-T+t}\big)\,.
		\end{equation}
		The second term can be written as follows
		\begin{equation}\label{eq:bound2-q}
			\begin{split}
				\hat{\E}[q_t(\sigma)&\ind(\cF^c)\mid\cD]=(1+o(1))\ \hat{\E}\left[\frac{d_{U_t}^--1}{m}\prod_{s=t}^{T-1}\frac{1}{d^+_{U_s}}\ind(\cF^c)\mid\cD\right]\\
				&=(1+o(1))\left\{ \hat{\E}\left[\frac{d_{U_t}^--1}{m}\prod_{s=t}^{T-1}\frac{1}{d^+_{U_s}}\mid\cD\right]-\hat{\E}\left[\frac{d_{U_t}^--1}{m}\prod_{s=t}^{T-1}\frac{1}{d^+_{U_s}}\ind(\cF)\mid\cD\right]\right\}\\
				&=(1+o(1))\ \hat{\E}\left[\frac{d_{U_t}^--1}{m}\prod_{s=t}^{T-1}\frac{1}{d^+_{U_s}}\right]-O\left(\frac{d_{\max}^-}{m}\ 2^{-T+t}\ \Q(\cF\mid\cD)\right)\\
				&=(1+o(1))\ \hat{\E}\left[\frac{d_{U_t}^--1}{m}\prod_{s=t}^{T-1}\frac{1}{d^+_{U_s}}\right]-o\left(n^{-1}\times 2^{-T+t}\right)\,.
			\end{split}
		\end{equation}
		It is easy to check that, if $U\sim\mu^{\rm in}$,
		\begin{equation}\label{eq:exp-for-U}
			\hat{\E}\left[\frac{1}{d_U^+} \right]=\rho\ ,\qquad \hat{\E}\left[\frac{d_U^--1}{d_U^+}\right]=\frac{1}{m}\sum_{x\in[n]}\frac{d_x^-(d_x^--1)}{d_x^+}=\gamma-\rho\ .
		\end{equation}
		Therefore, putting together \eqref{eq:q-T}, \eqref{eq:sum-q}, \eqref{eq:bound1-q}, \eqref{eq:bound2-q} and  \eqref{eq:exp-for-U}, we deduce
		\begin{equation}\label{eq:exp-sum-q}
			\begin{split}
				\hat{\E}\left[\sum_{t=0}^{T}q_t(\sigma)\mid\cD\right]=& (1+o(1))\left( \ \frac{d^--1}{m}+\hat{\E}\left[\sum_{t=0}^{T-1}\frac{d_{U_t}^--1}{m}\prod_{s=t}^{T-1}\frac{1}{d^+_{U_s}}\right] \right)+ o(n^{-1})\\
				=&(1+o(1))\left(\frac{d^--1}{m}+ \frac{\gamma-\rho}{m}\sum_{t=0}^{T-1} \rho^t\right) \sim \frac{1}{m}\left(d^- -1+\frac{\gamma-\rho}{1-\rho} \right) \,,
			\end{split}
		\end{equation}
		hence proving the validity of \eqref{eq:goalll}. 
	\end{proof}
	
	To conclude the section, we are now left to show \eqref{eq:second-mom-method-2}, which follows immediately by the following lemma.
	
	\begin{lemma}\label{prop:sec-mom-y}
		Recall the definition of $(Y_x)_{x\in[n]}$ in \eqref{eq:def-Y}. It holds \begin{equation}\label{eq:sup-2nd-mom}
			\lim_{n\to\infty}\frac{\E\left[\big(\sum_{x\in[n]}Y_x\big)^2\right]}{ \E\big[\sum_{x\in [n]}Y_x\big]^2}=1 \, .
		\end{equation}  
	\end{lemma}
	\begin{proof}
		Start by noting that it suffices to show that
		\begin{equation}\label{eq:only}
			\E\bigg[\big(\sum_{x\in[n]}Y_x\big)^2\bigg]\le (1+o(1))\ \E\bigg[\sum_{x\in [n]}Y_x\bigg]^2\,,
		\end{equation}
		since the other inequality is trivially true. Rewrite
		\begin{equation}\label{eq:760}
			\E\bigg[\big(\sum_{x\in[n]}Y_x\big)^2\bigg]=\sum_{x\in[n]}\sum_{y\in[n]}\E[Y_x\ Y_y]=\sum_{x\in[n]}\sum_{y\in[n]}\E[\mu_T(x)^2\mu_T(y)^2]\ f(d_x^-,d_x^+)\ f(d_y^-,d_y^+)\,,
		\end{equation}
		and consider the annealed process described in the proof of Lemma \ref{lemma:exp-lambda}, but this time, consider also the two additional annealed trajectories $\{W_t^{(3)}\}_{t\le T}$ and $\{W_t^{(4)}\}_{t\le T}$, to be generated, sequentially, after the generation of $\{W_t^{(1)}\}_{t\le T}$ and $\{W_t^{(2)}\}_{t\le T}$. We have
		\begin{equation}\label{eq:761}
			\E[Y_xY_y]=\P^{4\text{-an}}_{\rm unif}\left(W_T^{(1)}=W_T^{(2)}=x\,,\,W_T^{(3)}=W_T^{(4)}=y \right)\ f(d_x^{-},d_x^+)\ f(d_y^{-},d_y^+)\,.
		\end{equation}
		Let
		\begin{equation}
			\cN_x\coloneqq \left\{W_T^{(1)}=W_T^{(2)}=x \right\}\,,\quad\text{and}\quad \cI_y\coloneqq \left\{W_T^{(3)}=W_T^{(4)}=y \right\}\,.
		\end{equation}
		Start by considering the case $x\neq y$, and let $\cJ$  be the event where the trajectories of $(W_t^{(1)}, W_t^{(2)})_{t\le T}$ hit the set of trajectories of $(W_t^{(3)}, W_t^{(4)})_{t\le T}$, which we formally write as follows:
		\begin{equation}
			\cJ\coloneqq\left\{\exists s,t\le T\text{ and }\exists i\in\{1,2\},\ j\in\{3,4\}\text{ s.t. }W^{(i)}_s=W^{(j)}_t \right\}\,.
		\end{equation}
		
		We claim that, uniformly in $x\neq y$
		\begin{equation}\label{eq:assu1}
			\P^{4\text{-an}}_{\rm unif}(\cN_x\cap\cI_y\cap\cJ)=o\big(\P^{4\text{-an}}_{\rm unif}(\cN_x)\ \P^{4\text{-an}}_{\rm unif}(\cI_y) \big)\,,
		\end{equation}
		and
		\begin{equation}\label{eq:assu2}
			\P^{4\text{-an}}_{\rm unif}(\cN_x\cap\cI_y\cap\cJ^c)\le \ \P^{4\text{-an}}_{\rm unif}(\cN_x)\ \P^{4\text{-an}}_{\rm unif}(\cI_y)\,.
		\end{equation}
		Assuming the validity of \eqref{eq:assu1} and \eqref{eq:assu2}, from \eqref{eq:760} and \eqref{eq:761} we deduce
		\begin{equation}\label{eq:dedu}
			\begin{split}
				\E\bigg[\big(\sum_{x\in[n]}Y_x\big)^2\bigg]
				&\le(1+o(1))\sum_{x\in[n]}\sum_{\substack{y\in[n]\\y\neq x}} \P^{4\text{-an}}_{\rm unif}(\cN_x)\ \P^{4\text{-an}}_{\rm unif}(\cI_y)f(d_x^-,d_x^+)f(d_y^-,d_y^+)\\
				&\qquad\qquad\qquad\qquad+\sum_{x}\P^{4\text{-an}}_{\rm unif}(\cN_x\cap\cI_x)f(d_x^-,d_x^+)^2\\
				&\le(1+o(1))\E\bigg[\sum_{x\in[n]}Y_x\bigg]^2+O\left(\sum_{x}\P^{4\text{-an}}_{\rm unif}(\cN_x\cap\cI_x)\right)\,.
			\end{split}
		\end{equation}
		To bound the error term in \eqref{eq:dedu}, notice that for all $x\in[n]$,
		\begin{equation} \label{eq:N_x-I_x bound}
			\P^{4\text{-an}}_{\rm unif}(\cN_x\cap\cI_x)\le \frac{T\ d_x^-}{m-T}\times\frac{T^2\ d_{\max}^-}{m-2T}\times \frac{2\ T^2\ d_{\max}^-}{m-3T}\times \frac{3\ T^2\ d_{\max}^-}{m-4T}=O\left(\frac{T^7\ d_x^-\ (d_{\max}^-)^3}{n^4} \right)\,.
		\end{equation}
		Thus the validity of \eqref{eq:only} follows, since, thanks to the fact that $f$ is bounded above,
		\begin{equation}\label{eq:dedu2}
			\begin{split}
				\E\bigg[\big(\sum_{x\in[n]}Y_x\big)^2\bigg]&\le(1+o(1))\E\bigg[\sum_{x\in[n]}Y_x\bigg]^2+O\left(\frac{T^7\ (d_{\rm max}^-)^3}{n^3}\right)\le(1+o(1))\E\bigg[\sum_{x\in[n]}Y_x\bigg]^2\,,
			\end{split}
		\end{equation}
		where in the last step we used Assumption \ref{log degree assumptions} and $\E\big[\sum_{x\in[n]}Y_x\big]=O(n^{-1})$.
		
		We are left to show \eqref{eq:assu1} and \eqref{eq:assu2}.
		To prove \eqref{eq:assu1}, call
		\begin{equation}
			\cQ_y=\left\{\exists s\le T\text{ and }i\in\{1,2\}\text{ s.t. }X^{(i)}_s=y \right\}\,,
		\end{equation}
		let $\sigma$ be a realization of $\{W^{(1)}_t \}_{t\le T}$ and $\{W^{(2)}_t \}_{t\le T}$ satisfying $\cN_x$, and rewrite
		\begin{equation}\label{eq:assu1-a}
			\P^{4\text{-an}}_{\rm unif}(\cN_x\cap\cI_y\cap\cJ)= \P^{4\text{-an}}_{\rm unif}(\cN_x\cap\cQ_y\cap\cI_y\cap\cJ)+\P^{4\text{-an}}_{\rm unif}(\cN_x\cap\cQ_y^c\cap\cI_y\cap\cJ)\,.
		\end{equation}
		The first term on the right-hand side of \eqref{eq:assu1-a} can be bounded as follows
		\begin{equation}\label{eq:split-1}
			\P^{4\text{-an}}_{\rm unif}(\cN_x\cap\cQ_y\cap\cI_y\cap\cJ)\le  \P^{4\text{-an}}_{\rm unif}(\cN_x\cap\cQ_y)\sup_{\sigma\in\cN_x\cap\cQ_y} \P^{4\text{-an}}_{\rm unif}(\cI_y\mid \sigma)\,.
		\end{equation}
		Notice that the event on the first probability on the right-hand side of \eqref{eq:split-1} can be bounded by
		\begin{equation}\label{eq:split-1-a}
			\P^{4\text{-an}}_{\rm unif}(\cN_x\cap\cQ_y)=O\left(\frac{T^4\ d_x^-\ d_y^-\ d_{\rm max}^-}{n^3} \right)\,.
		\end{equation}
		Similarly, the conditional probability on the right-hand side of \eqref{eq:split-1} can be bounded uniformly in $\sigma\in\cN_x\cap \cQ_y$ by
		\begin{equation}\label{eq:split-1-b}
			\sup_{\sigma\in\cN_x\cap\cQ_y} \P^{4\text{-an}}_{\rm unif}(\cI_y\mid \sigma)=O\left(\frac{T^4\ (d_{\max}^-)^2}{m^2} \right)\,.
		\end{equation}
		Plugging \eqref{eq:split-1-a} and \eqref{eq:split-1-b} into \eqref{eq:split-1} we deduce
		\begin{equation}\label{eq:assu-1-c}
			\P^{4\text{-an}}_{\rm unif}(\cN_x\cap\cQ_y\cap\cI_y\cap\cJ)=O\left( \frac{T^8\ d_x^-\ d_y^-\ (d_{\max}^-)^3}{n^5}\right)=o\big(\P^{4\text{-an}}_{\rm unif}(\cN_x)\ \P^{4\text{-an}}_{\rm unif}(\cI_y) \big)\,,
		\end{equation}
		where in the last estimate we used Assumption \ref{log degree assumptions}  and Lemma \ref{lemma:exp-lambda}.
		
		As for the second term on the right-hand side of \eqref{eq:assu1-a}, we can start by splitting
		\begin{equation}\label{eq:split-1-d}
			\P^{4\text{-an}}_{\rm unif}(\cN_x\cap\cQ_y^c\cap\cI_y\cap\cJ)\le \P^{4\text{-an}}_{\rm unif}(\cN_x)\ \sup_{\sigma\in\cN_x\cap\cQ_y^c}\P^{4\text{-an}}_{\rm unif}(\cI_y\cap\cJ\mid \sigma)\,.
		\end{equation}
		Then, to bound the latter conditional probability we argue as follow:
		either $W^{(3)}$ hits the trajectory of $W^{(1)}\cup W^{(2)}$ and $y$, and $W^{(4)}$ hits $W^{(1)}\cup W^{(2)}\cup W^{(3)}$; or $W^{(3)}$ hits $y$ but not  $W^{(1)}\cup W^{(2)}$ and $y$, and then $W^{(4)}$ hits $W^{(1)}\cup W^{(2)}$ and $W^{(3)}$. In conclusion,
		\begin{equation}\label{eq:split-1-c}
			\sup_{\sigma\in\cN_x\cap\cQ_y^c}\P^{4\text{-an}}_{\rm unif}(\cI_y\cap\cJ\mid \sigma)=O\left(\frac{T^4\ d_y^-\ (d_{\max}^-)^2}{n^3} \right)\,.
		\end{equation}
		Plugging \eqref{eq:split-1-c} into \eqref{eq:split-1-d} and using Lemma \ref{lemma:exp-lambda} we conclude
		\begin{equation}\label{eq:assu-1-b}
			\P^{4\text{-an}}_{\rm unif}(\cN_x\cap\cQ_y^c\cap\cI_y\cap\cJ)=O\left(\frac{T^4\ d_x^-\ d_y^-\ (d_{\max}^-)^2}{n^5} \right)=o\big(\P^{4\text{-an}}_{\rm unif}(\cN_x)\ \P^{4\text{-an}}_{\rm unif}(\cI_y) \big)\,.
		\end{equation}
		Thus, \eqref{eq:assu1} follows from \eqref{eq:assu1-a}, \eqref{eq:assu-1-c} and \eqref{eq:assu-1-b}.
		
		To show \eqref{eq:assu2}, it is enough to realize that, for each $\sigma\in\cN_x$, called $(\cI_y\cap\cJ^c)(\sigma)$ the set of trajectories for $W^{(3)}$ and $W^{(4)}$ that do not intersect $\sigma$ and satisfy $W^{(3)}_T=W_T^{(4)}=y$, one has
		\begin{equation}\label{eq:split-2}
			\begin{split}
				\P^{4\text{-an}}_{\rm unif}(\cN_x\cap\cI_x\cap\cJ^c)=&\sum_{\sigma\in\cN_x}\P^{4\text{-an}}_{\rm unif}(\sigma)\sum_{\eta\in(\cI_y\cap\cJ^c)(\sigma)}\P^{4\text{-an}}_{\rm unif}(\eta \mid \sigma)\\
				=&\sum_{\sigma\in\cN_x}\P^{4\text{-an}}_{\rm unif}(\sigma)\sum_{\eta\in(\cI_y\cap\cJ^c)(\sigma)}\P^{4\text{-an}}_{\rm unif}(\eta )\\
				\le&\sum_{\sigma\in\cN_x}\P^{4\text{-an}}_{\rm unif}(\sigma)\sum_{\eta\in\cI_y}\P^{4\text{-an}}_{\rm unif}(\eta )= \P^{4\text{-an}}_{\rm unif}(\cN_x)\  \P^{4\text{-an}}_{\rm unif}(\cI_y)\,,
			\end{split}
		\end{equation}
		where in the second equality in \eqref{eq:split-2} we removed the conditioning on $\sigma$ since, under $\cJ^c$, the evolution of $W^{(3)}$ and $W^{(4)}$ is not affected by the exact set of vertices that outline the trajectories of $W^{(1)}$ and $W^{(2)}$. This concludes the proof. 
	\end{proof}

	\section{Mixing time of random walks with $\tilde\mu$-reset} \label{sec: Mixing time}
	In this section we prove that the mixing time of the auxiliary chain $\tilde{\mathbf{P}}^{\tilde{\mu}}$ on $\tilde{V}$ has polylogarithmic order with respect to the size $n$ of the original network $G$. The proof technique adapts the one used in \cite[Section 4.2]{QS23}, where the authors investigate the mixing time of two independent walks moving synchronously on DFA. Recall that $\tilde P=\tilde{P}^{\tilde \mu}$ refers to the transition matrix of $(\tilde X)_{t\ge 0}$, the two walks process on $G$ with $\tilde \mu$-reset, as defined in \eqref{eq:def-tilde-mu}, and $\tilde \pi$ to its unique stationary distribution on $\tilde V$. Throughout the whole section we set
	\begin{equation}\label{eq:def-S}
		S=S_n\coloneqq \ceil{\log^3(n)}\,.
	\end{equation}
	\begin{proposition}\label{prop:mixing}
		Assume that the degree sequence satisfies Assumption \ref{degree assumptions}. Then,
		\begin{equation}
			\max_{\x\in \tilde{V}} \|\tilde{P}^S(\x,\cdot)-\tilde{\pi}(\cdot)\|_{\rm{TV}} \overset{\P}{\longrightarrow}0\,,
			\label{mixing time}
		\end{equation}
		where $S$ is defined as in \eqref{eq:def-S}.
	\end{proposition}

	Recall the definition of the random rooted tree $\cT_x^+$, $x\in[n]$, defined in Section \ref{sec: Local structures}, of  $\hbar$ in \eqref{eq:def-h-bar} and of the set of locally-tree-like vertices $V_\star$ in \eqref{eq:def Vstar}. By using the same approach as in Proposition \ref{prop:general-formula}, we start by showing the following property of the distribution $\tilde\mu$.
	\begin{proposition}\label{prop:mu-Vstar}
		Assume that the degree sequence satisfies Assumption \ref{log degree assumptions}, and let $V_\star$ as in \eqref{eq:def Vstar}. It holds 
		$$\tilde\mu(V_\star)=\frac{\sum_{x\in V_\star}\pi^2(x)}{\pl \pi^2\pr}\overset{\P}{\longrightarrow}1\ .$$
	\end{proposition}
	
	\begin{proof}
		It is enough to show the following
		\begin{equation}\label{eq:V-star}
			n \left|\sum_{x\in[n]} \pi^2(x) - \sum_{x\in[n]}\mathds{1}_{\{x\in V_\star\}} \pi^2(x) \right| \overset{\P}{\longrightarrow} 0,
		\end{equation}
		as 
		\begin{equation*}
			\tilde \mu(V_\star)= 1 - \frac{
				n \sum_{x\in[n]} \left[ \pi^2(x) - \mathds{1}_{\{x\in V^\star\}} \pi^2(x)\right]}{n \sum_{x\in[n]} \pi^2(x)},
		\end{equation*}
		and $n \sum_{x\in[n]} \pi^2(x) = \Theta_\P(1)$, as proved in Proposition \ref{prop:pi-diag}.  Fix $x\in[n]$, and consider the approximation $\mu_T$ of $\pi(x)$ introduced in \eqref{eq:def-mu-T}. Recall the definition of the event $\cH$ in \eqref{eq:pi-close-P^T} and that, by \eqref{eq:event-E-whp}, $\cH$ holds with high probability. Fixed $\delta>0$, we get
		\begin{align*}
			&\P\left( n \sum_{x\not\in V_\star} \pi^2(x) >\delta\right)  
			= \P\left( n \left|\sum_{x\in[n]} \pi^2(x) - \sum_{x\in[n]}\mathds{1}_{\{x\in V_\star\}} \pi^2(x) \right|>\delta\,,\,\cH\right) + o(1)\ .
		\end{align*}
		Consider the events
		\begin{align*}
			\cA=\bigg\{ n \sum_{x\in[n]}|\pi^2(x) -\mu_S^2(x) |>\frac{\delta}{3}\bigg\}&\,,\quad  \cB=\bigg\{n \sum_{x\in[n]} \mathds{1}_{x\in V_\star}|\pi^2(x)-\mu_S^2(x) |>\frac{\delta}{3}\bigg\}\,, \\
			\cC=\bigg\{n \sum_{x\in[n]}\left[\mu_S^2(x)-\mathds{1}_{x\in V_\star}\mu_S^2(x) \right]>\frac{\delta}{3}\bigg\}\,&,\quad \cK=\bigg\{ n \sum_{x\in[n]}\left[ \pi^2(x) - \mathds{1}_{\{x\in V_\star\}} \pi^2(x) \right]>\delta\bigg\} \ .
		\end{align*}
		By the triangular inequality we have that
		\begin{equation} \label{eq:split-delta}
			\begin{split}
				\P(\cK\,,\,\cH)
				\leq \P\left( \cA \;,\; \cH\right) + \P\left(\cB \;,\; \cH\right) + \P\left( \cC\;,\;\cH\right) \leq 2\ \P\left( \cA \;,\; \cH\right) + \P\left( \cC\;,\;\cH\right) \, . 
			\end{split}
		\end{equation}
		The first term on the right hand side of \eqref{eq:split-delta} can be easily bounded by Markov inequality, indeed
		\begin{align*}
			\P\left( \cA\,, \cH\right) &\leq \frac{3\ n}{\delta}\ \sum_{x\in[n]}\E\left[\ind_{\cH}|\pi(x)-\mu_S(x)||\pi(x)+\mu_S(x)|\right]\leq \frac{6}{\delta}\ n^2\ e^{-\log(n)^{3/2}}=o(1)\,.
		\end{align*}
		Therefore we are left to prove that, for any $\delta>0$,
		\begin{equation}\label{eq:lambda difference}
			\P\left( \cC\right)=o(1)\ .
		\end{equation}
		In order to show \eqref{eq:lambda difference}, 
		suppose that 
		\begin{equation}\label{eq:V-star-first-moment}
			\E[\mathds{1}_{\{x\in V_\star\}} \mu_S^2(x)] = (1-o(1))\E[\mu_S^2(x)]\,
		\end{equation}
		for any $x\in [n]$. Then, summing over $x$ and using Lemma \ref{lemma:exp-lambda}, we deduce that $\E[\mu_S^2([n]\setminus V_\star)]=o(n^{-1})$, so that \eqref{eq:lambda difference} follows by Markov inequality.
		
		Therefore, we are left to prove \eqref{eq:V-star-first-moment}. To this aim, note that it is enough to show the lower bound, as the upper bound is trivially satisfied.\\
		Let $\gamma_x :=\cB_x^+(\hbar)$ denote any realization of the complete out-neighbourhood of $x$ up to distance $\hbar$, defined in \eqref{eq:def-h-bar}, and notice that, under $\P$, the event $\{x\in V_\star\}$ is measurable with respect to $\gamma_x$. Thus
		\begin{equation}\label{eq:out-neighbourhood-conditioning}
			\E\big[\mathds{1}_{\{x\in V_\star\}} \mu_S^2(x)\big] = \E\left[\mathds{1}_{\{x\in V_\star\}} \E[ \mu_S^2(x)\mid \gamma_x]\right].
		\end{equation}
		In order to study $\E[ \mu_S^2(x)\mid \gamma_x]$, we need to employ another annealing argument as in proof of Lemma \ref{lemma:exp-lambda}, but this time the two annealed walks evolve in an environment that starts as $\gamma_x$.  
		
		Let $\P_{\rm unif}^{2\text{-an}|\gamma_x}$ denote the law of such a process, as introduced in Remark \ref{rmk:multiple-annealed-4}. In particular,
		\begin{equation*}
			\E[ \mu_S^2(x)\mid \gamma_x] = \P_{\rm unif}^{2\text{-an}|\gamma_x}(W^{(1)}_S=W^{(2)}_S=x)\,.
		\end{equation*}
		We aim to show that 
		\begin{equation}\label{eq:lower bound coupling}
			\P_{\rm unif}^{2\text{-an}|\gamma_x}(W^{(1)}_S=W^{(2)}_S=x) \geq \P_{\rm unif}^{2\text{-an}}(W^{(1)}_S=W^{(2)}_S=x)(1-o(1))\,,
		\end{equation}
		where the annealing law on the right-hand side corresponds to the one in which the initial environment is given by the empty matching of the edges, i.e., that of the process described in the proof of Lemma \ref{lemma:exp-lambda}. If \eqref{eq:lower bound coupling} holds, then \eqref{eq:out-neighbourhood-conditioning} reads
		\begin{equation*}
			\E\left[\mathds{1}_{\{x\in V_\star\}} \E[ \mu_S^2(x)\mid \gamma_x]\right] \geq \P(x\in V_\star)\E[\mu_S^2(x)](1-o(1))=(1-o(1))\E[\mu_S^2(x)]\,,
		\end{equation*}
		since $\P(x\in V_\star)=1-o(1)$, as shown in Lemma \ref{lemma:LTL-structure}. To prove \eqref{eq:lower bound coupling} we construct a coupling of the two annealed processes as follows:
		\renewcommand{\theenumi}{\alph{enumi}}
		\begin{enumerate}
			\item Let $(\bar W^{(1)}_t,\bar W^{(2)}_t)_{t\leq S}$ denote the annealed walks having as initial environment $\gamma_x$. On the other hand, let   $({W}^{(1)}_t,{W}^{(2)}_t)_{t\leq S}$ denote the annealed random walks having as initial environment the empty matching of the edges. Henceforth,  set $\bar\omega_0=\gamma_x$ and $\omega_0=\emptyset$.
			\item For the first walk, at time $t=0$, sample $W_0^{(1)}=\Bar{W}_0^{(1)}$ according to the uniform distribution on $[n]$. If $W_0^{(1)}\in\gamma_x$  declare the coupling as \emph{failed} and let the two processes evolve independently.
			\item For the first walk and $t\in(0,S)$, construct $W_t^{(1)}$ as  for the annealed walk in the proof of Lemma \ref{lemma:exp-lambda}, i.e., sample a tail $e\in E^+_{W^{(1)}_{t-1}}$ u.a.r.:
			\begin{itemize}
				\item[(c-1)]  If $\omega_{t-1}(e)=f\in E^-$ set $W_t^{(1)}=\bar W_t^{(1)}=v_f$, where $v_f$ is the vertex incident to $f$;
				\item[(c-2)] otherwise, select a head $f$ u.a.r. among those that are not matched in $\omega_{t-1}$:
				\begin{itemize}
					\item set $\omega_t(e)=f$ and $W_t^{(1)}=v_f$;
					\item if $\omega^{-1}_{t-1}(f)=\emptyset$ but $\bar\omega^{-1}_{t-1}(f)=e'\in E^+$ or if $v_f\in\gamma_x$, declare the coupling as \emph{failed} and let the processes evolve independently;
					\item otherwise let $\Bar{W}_{t}=W_{t}$ and $\bar\omega_t(e)=f$.
				\end{itemize}
			\end{itemize}
			\item Follow the same procedure for the walks $W^{(2)}$ and $\bar W^{(2)}$.
		\end{enumerate}
		Call $\hat{\P}$ the joint probability space just described, and note that the marginals of the coupled process corresponds to the laws $\P^{2\text{-an}|\gamma_x}_{\rm unif}$ and $\P^{2\text{-an}}_{\rm unif}$, respectively. Let $\cF$ be the event that the coupling fails. It holds that
		\begin{equation*}
			\hat{\P} \left(\bigcap_{t=0}^S\bigcap_{i\in\{1,2\}}\{W_t^{(i)}=\bar W_t^{(i)}\}\mid \cF^c\right) =1.
		\end{equation*}
		Notice also that, at each step, in order for the coupling to fail, it must be the case that the sample at step (c-2) is a head of a vertex in $\gamma_x$. Hence, for each step of the construction, the failure probability of the coupling can be bounded uniformly by
		\begin{equation}
			\max_{x\in[n]}\max_{\gamma_x}\max_{t\in[0,S]}\max_{i=1,2}\hat{\P}(\text{\small Coupling fails at step $t$ of the $i$-th walk})\le \frac{(d_{\max}^+)^{\hbar}}{n}+\frac{d_{\rm max}^-\,(d_{\rm max}^+)^\hslash}{m-2S-(d_{\rm max}^+)^{\hslash}}\,,
		\end{equation}
		hence, by a union bound,
		\begin{align*}
			\hat{\P}(\cF) &\leq   2S  \left( \frac{(d_{\max}^+)^{\hbar}}{n}+\frac{d_{\rm max}^-\,(d_{\rm max}^+)^\hslash}{m-2S-(d_{\rm max}^+)^{\hslash}} \right)=o(1)\,,
		\end{align*}
		as $S=\log^3(n)$, $d^-_{\rm max}=o(n^{1/2 -\varepsilon})$, for some $\varepsilon>0$ and $(d_{\max}^+)^{\hslash}=O(n^{1/5})$. As a consequence
		\begin{align*}
			\P_{\rm unif}^{{2\text{-an}}|\gamma_x}(X_S=Y_S=x) &\geq  \hat{\P}(W^{(1)}_S=W^{(2)}_S=x\;,\; \cF^c)=  (1-o(1))\P_{\rm unif}^{2\text{-an}}(X_S=Y_S=x)\,,
		\end{align*}
		from which \eqref{eq:V-star-first-moment}, and in turn \eqref{eq:lambda difference}, follow.
	\end{proof}
	
	\subsection{Proof of Proposition \ref{prop:mixing}}
	The general idea behind the proof is based on a comparison between the chain with $\tilde\mu$-reset,  $\tilde{\mathbf{P}}^{\tilde \mu}$, and the product chain, $\mathbf{P}^{\otimes2}$. Assuming that the two processes start at some $(x,y)$ with $y\neq x$, they can be perfectly coupled up to the first meeting of the two walks, corresponding to the first hitting time of $\partial$ for the process with reset. Recall that the two stationary distributions, $\tilde\pi$ and $\pi^{\otimes2}$, coincide for every couple $(x,y)\not\in\Delta$ and $\pi^{\otimes2}(\Delta)=\tilde\pi(\partial)=o_{\P}(1)$. Therefore, if the meeting happens after the mixing horizon $T$, we can rely on Theorem \ref{th:cutoff} and the fact that the mixing time of the product chain is at most a constant multiple of the one of a single walk. At that point, we are left to prove that the collapsed chain  $\tilde{P}$, started at $\partial$ makes logarithmic number of steps, within time $T$, without visiting $\partial$ has probability $1-o_\P(1)$.
	
	We now start to make the above mentioned heuristic rigorous. Let us start by defining the surjective map $ \varphi:V^2\rightarrow\tilde{V}$ by
	\begin{equation}\label{eq:def-varphi}
		\varphi(\x)=
		\begin{cases}
			(x_1,x_2)  &\x=(x_1,x_2)\text{ with }x_1,x_2\in[n]\,,\,x_1\neq x_2\,,\\ 
			\partial  &\x=(x,x)\text{ with }x\in[n]\,.
		\end{cases}
	\end{equation}
	Moreover, recall that in our notation $\tilde{\mathbf{P}}^{\tilde \mu}_{\x}$ describes the quenched law of $\tilde X$, with $\tilde X_0=\x\in\tilde V$ and $\tilde \mu$-reset, while $\mathbf{P}^{\otimes 2}_{(x,y)}$ the quenched law of two independent walks with $(x,y)\in [n]^2$ as initial position. By the triangular inequality and observing that $\|\tilde{\pi}(\cdot)-\pi^{\otimes2}(\varphi^{-1}(\cdot))\|_{\rm{TV}}=0$ almost surely, it holds that
	\begin{equation}\label{eq:split-tv}
		\begin{split}
			\max_{\x\in \tilde{V}} \|\tilde{P}^S(\x,\cdot)-\tilde{\pi}(\cdot)\|_{\rm{TV}} &\leq   \max_{(x,y)\in [n]^2} \|\left(P^{\otimes2}\right)^S((x,y),\cdot)-\pi^{\otimes2}(\cdot)\|_{\rm{TV}} \\ 
			&\quad+  \max_{(x,y)\in [n]^2} \sup_{A\subset\tilde{V}} \left|\tilde{\mathbf{P}}^{\tilde \mu}_{\varphi((x,y))}\left(\tilde{X}_S\in A\right) - \mathbf{P}^{\otimes2}_{(x,y)}\left(X^{\otimes2}_S\in\varphi^{-1}(A)\right)\right|.
		\end{split}
	\end{equation}
	The first term on the right-hand side vanishes in probability thanks to Theorem \ref{th:cutoff} and the fact that $S\gg \log(n)$, while the second term can be rewritten as
	\begin{equation}
		\max_{(x,y)\in [n]^2} \sup_{A\subset\tilde{V}} \left|\tilde{\mathbf{P}}^{\tilde \mu}_{\varphi((x,y))}\left(\tilde{X}_S\in A, \tau_\partial<S\right) - \mathbf{P}^{\otimes2}_{(x,y)}\left(X^{\otimes2}_S\in\varphi^{-1}(A), \tau_{\rm{meet}}<S\right)\right|,
		\label{eq:mixing 1}    
	\end{equation}
	since the product chain and the collapsed chain can be perfectly coupled until the first hitting time of the diagonal. In particular
	\begin{equation}
		\tilde{\mathbf{P}}^{\tilde \mu}_{\varphi((x,y))}\left(\tau_\partial=t\right) = \mathbf{P}^{\otimes2}_{(x,y)}\left(\tau_{\rm{meet}}=t\right), \quad \forall (x,y)\in [n]^2\,,\,x\neq y\,,\, t\in\mathbb{N}\,.
		\label{eq:tauDelta equals tauMeet}
	\end{equation}
	By the strong Markov property and \eqref{eq:tauDelta equals tauMeet} we get, uniformly in $(x,y)\in V^2$ and $A\subset \tilde{V}$,
	\begin{equation}\label{eq:split-tv2}
		\begin{split}
			&\left| \tilde{\mathbf{P}}^{\tilde \mu}_{\varphi((x,y))}\left(\tilde{X}_S\in A, \tau_\partial<S\right) - \mathbf{P}^{\otimes2}_{(x,y)}\left(X^{\otimes2}_S\in\varphi^{-1}(A), \tau_{\rm{meet}}<S\right)\right| \\
			&  =\bigg|\sum_{t=0}^S  \tilde{\mathbf{P}}^{\tilde \mu}_{\varphi((x,y))}\left(\tau_\partial=t\right) \tilde{\mathbf{P}}^{\tilde \mu} \left(\tilde{X}_{S-t}\in A\right) \\
			& \qquad\qquad  \qquad\qquad- \sum_{t=0}^S \sum_{z\in[n]}\mathbf{P}^{\otimes2}_{(x,y)}\left(\tau_{\rm{meet}}=t, X_t^{\otimes2}=(z,z) \right) \mathbf{P}^{\otimes2}_{(z,z)} \left(X^{\otimes2}_{S-t}\in\varphi^{-1}(A)\right)\bigg| \\
			&=  \bigg|\sum_{t=0}^S \sum_{z\in[n]}\mathbf{P}^{\otimes2}_{(x,y)}\left(\tau_{\rm{meet}}=t,X^{\otimes2}_t =(z,z)\right)\left[\tilde{\mathbf{P}}^{\tilde \mu} \left(\tilde{X}_{S-t}\in A\right) - \mathbf{P}^{\otimes2}_{(z,z)} \left(X^{\otimes2}_{S-t}\in\varphi^{-1}(A)\right) \right]\bigg|\,.
		\end{split}
	\end{equation}
	Let us fix a constant $\alpha\in(0,1)$ and partition the above sum over $t$ into two parts centered at $\alpha S$. By the triangle inequality and \eqref{eq:mixing 1} we get
	\begin{equation}
		\begin{split}
			& \max_{(x,y)\in [n]^2} \sup_{A\subset\tilde{V}} \left|\tilde{\mathbf{P}}^{\tilde \mu}_{\varphi((x,y))}\left(\tilde{X}_S\in A, \tau_\partial<S\right) - \mathbf{P}^{\otimes2}_{(x,y)}\left(X^{\otimes2}_S\in\varphi^{-1}(A), \tau_{\rm{meet}}<S\right)\right| \\
			&\qquad\qquad\leq  \max_{(x,y)\in [n]^2} \sup_{A\subset\tilde{V}}  \bigg|\sum_{t=0}^{\alpha S} \sum_{z\in[n]}\mathbf{P}^{\otimes2}_{(x,y)}\left(\tau_{\rm{meet}}=t,X^{\otimes2}_t =(z,z)\right)\times \\
			&\qquad\qquad\qquad\qquad\times \sup_{t\leq \alpha S} \max_{z\in[n]}\left[\tilde{\mathbf{P}}^{\tilde \mu}_\partial \left(\tilde{X}_{S-t}\in A\right) - \mathbf{P}^{\otimes2}_{(z,z)} \left(X^{\otimes2}_{S-t}\in\varphi^{-1}(A)\right) \right]+ \\
			&\qquad\qquad\qquad\qquad\qquad\qquad\qquad\qquad+ \sum_{t=\alpha S}^S\sum_{z\in[n]} \mathbf{P}^{\otimes2}_{(x,y)}\left(\tau_{\rm{meet}}=t, X^{\otimes2}_t=(z,z)\right) \bigg|\\
			&\qquad\qquad\leq\max_{(x,y)\in [n]^2} \sup_{A\subset\tilde{V}}  \bigg|\sup_{t\leq \alpha S} \max_{z\in[n]}\left[\tilde{\mathbf{P}}^{\tilde \mu}_\partial \left(\tilde{X}_{S-t}\in A\right) - \mathbf{P}^{\otimes2}_{(z,z)} \left(X^{\otimes2}_{S-t}\in\varphi^{-1}(A)\right) \right]\\
			&\qquad\qquad\qquad\qquad\qquad\qquad\qquad\qquad+ \sum_{t=\alpha S}^S\sum_{z\in[n]} \mathbf{P}^{\otimes2}_{(x,y)}\left(\tau_{\rm{meet}}=t, X^{\otimes2}_t=(z,z)\right) \bigg|\,,
		\end{split}      
		\label{eq:mixing 2}
	\end{equation}
	where  we used the rough bounds
	$$\left| \tilde{\mathbf{P}}^{\tilde \mu}_{\partial} \left(\tilde{X}_{S-t}\in A\right) - \mathbf{P}^{\otimes2}_{(z,z)} \left(X^{\otimes2}_{S-t}\in\varphi^{-1}(A)\right) \right| \leq 1\,,\qquad t\in[\alpha S,S]\,,z\in[n]\,,$$
	and
	\begin{equation*}
		\sum_{t=0}^{\alpha S} \sum_{z\in[n]}\mathbf{P}^{\otimes2}_{(x,y)}\left(\tau_{\rm{meet}}=t,X^{\otimes2}_t =(z,z)\right)\leq 1\,.
	\end{equation*}
	Putting together \eqref{eq:split-tv}, \eqref{eq:split-tv2} and \eqref{eq:mixing 2}, using the triangular inequality and the fact ,that for any $A\subset \tilde{V}$ it holds that $\tilde{\pi}(A)=\pi^{\otimes2}(\varphi^{-1}(A))$, we get
	\begin{equation}\label{eq:mixing 3}
		\begin{split}
			\max_{\x\in \tilde{V}} \|\tilde{P}^S(\x,\cdot)-\tilde{\pi}(\cdot)\|_{\rm{TV}}  
			&\leq \sup_{t\leq \alpha S} \|\tilde{P}^{S-t}\left(\partial,\cdot\right) - \tilde{\pi}\|_{\rm{TV}}\\
			&\qquad\qquad+
			\sup_{t\leq \alpha S} \max_{z\in[n]}\|\left(P^{\otimes2}\right)^{S-t}\left((z,z),\cdot\right) - \pi^{\otimes2}\|_{\rm{TV}} \\
			&\qquad\qquad\qquad\qquad+ \max_{(x,y)\in V^2}\sum_{t=\alpha S}^S \mathbf{P}^{\otimes2}_{(x,y)}\left(\tau_{\rm{meet}}=t\right)\,.
		\end{split}
	\end{equation}
	The second term on the right-hand side of \eqref{eq:mixing 3} goes to zero in probability, thanks to Theorem \ref{th:cutoff} and the fact that the total-variation distance is non-increasing in $t$.
	As for the third term in \eqref{eq:mixing 3}, we have that
	\begin{align*}
		\max_{(x,y)\in V^2} \sum_{t=\alpha S}^S \mathbf{P}^{\otimes2}_{(x,y)}\left(\tau_{\rm{meet}}=t\right) &\leq S\, \sup_{t>\alpha S} \max_{(x,y)\in V^2} \left[ \mathbf{P}^{\otimes2}_{(x,y)}\left(X^{\otimes2}_t\in\varphi^{-1}(\partial)\right) - \tilde{\pi}(\partial)\right] + S\ \tilde{\pi}(\partial) \\
		&\leq S\, \sup_{t>\alpha S} \max_{(x,y)\in V^2} \|\left(P^{\otimes2}\right)^t\left((x,y),\cdot\right) - \pi^{\otimes2}\|_{\rm{TV}} + S\,\tilde{\pi}(\partial)=o_\P(1)\,,
	\end{align*}
	where the last asymptotic estimate follows from
	\begin{equation*}
		S\,\tilde{\pi}(\partial)= \log^3(n)\,\sum_{x\in[n]}\pi^2(x) \leq n\,\log^3(n)\,\left[\max_{y\in[n]}\pi(y)\right]^2\,,
	\end{equation*}
	and Theorem \ref{th:extremal-pi}, while
	\begin{align*}
		S\, \sup_{t>\alpha S} \max_{(x,y)\in V^2} \|\left(P^{\otimes2}\right)^t\left((x,y),\cdot\right) - \pi^{\otimes2}\|_{\rm{TV}} \leq S \max_{(x,y)\in V^2} \|\left(P^{\otimes2}\right)^{\alpha S}\left((x,y),\cdot\right) - \pi^{\otimes2}\|_{\rm{TV}} 
		=o_\P(1)\,,
	\end{align*}
	thanks to Theorem \ref{th:cutoff} and monotonicity of the total-variation distance. In conclusion, for any constant $\alpha\in(0,1)$
	\begin{equation}
		\max_{\x\in \tilde{V}} \|\tilde{P}^T(\x,\cdot)-\tilde{\pi}(\cdot)\|_{\rm{TV}}  \le  \sup_{t\leq \alpha S} \|\tilde{P}^{S-t}\left(\partial,\cdot\right) - \tilde{\pi}\|_{\rm{TV}}+o_\P(1)\,.
	\end{equation}
	Choosing, e.g, $\alpha=\frac12$, to conclude the proof of Proposition \ref{prop:mixing} we are left to show that
	\begin{equation}\label{eq:1st-rhs-mixing3}
		\|\tilde{P}^{\frac{S}2}\left(\partial,\cdot\right) - \tilde{\pi}\|_{\rm{TV}}\overset{\P}{\longrightarrow}0\,.
	\end{equation}
	The proof of the latter convergence is provided in next subsection.
	
	\subsubsection{Proof of \eqref{eq:1st-rhs-mixing3}}
	Recall the definition of $\hbar$ and $V_\star$ in \eqref{eq:def-h-bar} and \eqref{eq:def Vstar}, respectively. Consider the process $(\tilde{X}_t)_{t\ge 0}$ with law $\tilde{\mathbf{P}}^{\tilde{\mu}}$ and, if $\tilde{X}_t\neq\partial$, call $X_t$ and $Y_t$ the projections of the two coordinates of $\tilde X_t$. Define
	\begin{equation}
		\label{deviation time}
		\tau_{\rm{dev}} \coloneqq \inf\{t>0\ :\ \tilde{X}_t\neq\partial\,, \cB^+_{X_t}(\hbar)\cap\cB^+_{Y_t}(\hbar)=\emptyset\text{ and }X_t,Y_t\in V_\star  \}\ ,
	\end{equation}
	to be the first time such that the walks are visiting vertices the $\hbar$ out-neighborhood of which are not intersecting trees.  Moreover, let $\nu_{\rm dev}\in\mathcal{P}(\tilde V\setminus \partial)$ be the distribution of the process $\tilde{X}$ started at $\partial$,
	at the occurrence of the stopping time $\tau_{\rm dev}$. Namely,
	\begin{equation}\label{eq:nu_dev}
		\nu_{\rm{dev}}(\x) \coloneqq\mathbf{\tilde{P}}_\partial^{\tilde \mu} \left(\tilde{X}_{ \tau_{\rm{dev}}}=\x\right)\,, \qquad \x\in \tilde{V}\setminus \partial.
	\end{equation}
	The proof is articulated into three main lemmas, Lemma \ref{lemma:tau-dev-tau-meet}, Lemma \ref{lemma:tau-dev} and Lemma \ref{lemma:meeting-distant-particles}. First, in Lemma \ref{lemma:tau-dev-tau-meet}
	we show that for a typical realization of the graph, the process $\tilde X$ started at $\partial$ is such that the probability that both $\tau_\partial^+$ and $\tau_{\rm dev}$ are ``large'' is arbitrarily small. Then, in Lemma \ref{lemma:tau-dev}, we use that result as a bootstrap to show that the probability that $\tau_{\rm dev}$ itself is ``large'' is arbitrarily small. Finally, in Lemma  \ref{lemma:meeting-distant-particles} (and in Corollary \ref{coro:mu-dev}) we show that starting at $\nu_{\rm dev}$ the probability to hit $\partial$ before time $S$ is arbitrarily small for a typical realization of graph. To conclude the proof it suffices to collect the pieces: if both $\tilde{\mathbf{P}}^{\tilde{\mu}}_\partial(\tau_{\rm dev}<S/4)$ and $\tilde{\mathbf{P}}^{\tilde{\mu}}_{\nu_{\rm dev}}(\tau_\partial>S)$ are $1-o_\P(1)$, then with the same probability the processes $\tilde{\mathbf{P}}^{\tilde{\mu}}$ and  $\mathbf{P}^{\otimes 2}$ can be perfectly coupled within time $S/2$ for a consecutive interval of time having size $S/4$. At this point, the desired result is a consequence of Theorem \ref{th:cutoff}. We now present the main three lemmas and the respective proofs, and in the last part of this section we spell-out in detail the above-mentioned concluding argument.
	\begin{lemma}\label{lemma:tau-dev-tau-meet}
		Let $h_\star=h_{\star,n}$ be any sequence such that $h_\star\to\infty$ and $h_\star=o(\log(n))$, and call 
		\begin{equation}
			\tau_\partial^+\coloneqq\inf\{t\ge 1\mid \tilde{X}_t=\partial \}\,.
		\end{equation}
		Then, 
		\begin{equation}
			\mathbf{\tilde P}_\partial^{\tilde \mu}(\tau_{\rm dev}\wedge\tau_\partial^+>h_\star)\overset{\P}{\longrightarrow}0.
		\end{equation}
	\end{lemma}
	\begin{proof}
		Let $V_\star\subset[n]$ as in \eqref{eq:def Vstar} and $\hbar$ as in \eqref{eq:def-h-bar}. Recall that, thanks to Proposition \ref{prop:mu-Vstar}, $\tilde\mu(V_\star)=1-o_{\P}(1)$.  Hence, fixed any
		$\varepsilon_1,\varepsilon_2,\varepsilon_3>0$ we can bound for all $n$ large enough
		\begin{equation}\label{eq:eps0123}
			\begin{split}
				&\P\left(\mathbf{\tilde P}_\partial^{\tilde \mu}(\tau_{\rm dev}\wedge\tau_\partial^+>h_\star)\le \varepsilon_1\right)\\
				&\qquad\qquad\ge\ \P(\tilde\mu(V_\star)>1-\varepsilon_2)-\P\left(\mathbf{\tilde P}^{\tilde \mu}_\partial(\tau_{\rm dev}\wedge\tau_\partial^+>h_\star)> \varepsilon_1\,,\,\tilde\mu(V_\star)>1-\varepsilon_2\right)\\
				&\qquad\qquad\ge\ 1-\varepsilon_3-\P\left(\mathbf{\tilde P}_\partial^{\tilde \mu}(\tau_{\rm dev}\wedge\tau_\partial^+>h_\star)>\varepsilon_1\,,\,\tilde\mu(V_\star)>1-\varepsilon_2\right)\ .
			\end{split}
		\end{equation}
		Consider the (random) subset of states 
		\begin{equation}\label{eq:def-M}
			M\coloneqq \left\{\x=(x_1,x_2)\in\tilde{V}\setminus\{\partial \}\mid x_1\in V_\star\,,\,x_1\to x_2\text{ or } x_2\in V_\star\,,\,x_2\to x_1 \right\}\,,
		\end{equation}
		and notice that for every $\x=(x_1,x_2)$ it always possible to identify a \emph{chasing} and a \emph{escaping} walk. In that case, we write $x_c$ and $x_e$ to avoid confusion on the roles of $x_1$ and $x_2$.
		As already pointed out in \eqref{eq:tauDelta equals tauMeet}, starting the process $\tilde{X}$ at any $\x=(x_1,x_2)\in M$, the first hitting time of $\partial$ equals the first meeting time of two independent walks started at $(x_1,x_2)\in[n]^2\setminus\Delta$. We now show that, starting at any  $\x\in M$, the probability of the event $\{\tau_{\rm dev}>h_\star\}\cap\{\tau_{\rm meet}>h_\star\}$ is $o_\P(1)$. 
		Without loss of generality, assume that $\x=(x_c,x_e)$. Under the event  $\{\tau_{\rm meet}>s\}$, the only way to realize $\{\tau_{\rm dev}>s\}$ is to have the chasing walk (the one starting at $x_c$) following the steps of the escaping one (the one starting at $x_e$) without reaching it.
		Hence, calling $A_{x_c}(s)$ the random number of jumps made by the walk starting at $x_c$ before $s>0$, for all $s\in[0,\hslash)$, $x_c\in V_\star$, $x_e\in [n]$ such that $x_c\to x_e$ we have
		\begin{equation}\label{eq:a.s.-bound}
			\begin{split}
				\mathbf{P}^{\otimes 2}_{(x_c,x_e)}(\tau_{\rm dev}\wedge\tau_{\rm meet}>s)&
				\le (d_{\rm min}^+)^{-s/3}+\mathbf{P}_{(x_c,x_e)}^{\otimes 2}\big(A_{x_c}(s)\le \tfrac13 s\big)\,.
			\end{split}
		\end{equation}
		In particular, taking $s\in(h_\star,\hslash)$ in the last display and using the fact that $$\mathbf{P}^{\otimes 2}_{(x_c,x_e)}\big(A_{x_c}(s)\le\tfrac13 s\big)=\mathbf{P}\bigg({\rm Bin}\big(s,\tfrac12\big)\le\tfrac13 s\bigg)=o(1)\,,$$ we conclude that, for every $\varepsilon_1>0$ and all $n$ large enough
		\begin{equation}\label{eq:a.s.-bound2}
			\P\left( \max_{\x\in M} \mathbf{P}^{\otimes 2}_{\x}(\tau_{\rm dev}\wedge\tau_{\rm meet}>s)\le \varepsilon_1\right)=1\,,
		\end{equation}
		where we used that $d_{\min}^+\ge2$. Exploiting \eqref{eq:a.s.-bound2}, we finally obtain, for all sufficiently large $n$,
		\begin{equation}\label{eq:eps-fin}
			\begin{split}
				&\P\left(\mathbf{\tilde P}_\partial^{\tilde \mu}(\tau_{\rm dev}\wedge\tau_\partial^+>h_\star)>\varepsilon_1\,,\,\tilde \mu(V_\star)>1-\varepsilon_2\right) \\
				&\qquad\le \P\left(\sum_{\substack{(x_1,x_2)\in[n]^2 : \\ \x=(x_1,x_2)\in M}} \frac{\tilde\mu(x_c)}{d_{x_c}^+}\mathbf{P}^{\otimes2}_{(x_c,x_e)}(\tau_{\rm dev}\wedge\tau_{\rm meet}>h_\star-1)>\varepsilon_1\,,\,\tilde\mu(V_\star)>1-\varepsilon_2\right) + \varepsilon_2 \\ 
				&\qquad\le  \P\left(\max_{(x_c,x_e)\text{ s.t. }x_c\in V_\star,x_c\to x_e}\mathbf{P}^{\otimes 2}_{(x_c,x_e)}(\tau_{\rm dev}\wedge\tau_{\rm meet}>h_\star-1)>\varepsilon_1\right) + \varepsilon_2= \varepsilon_2\,,
			\end{split}
		\end{equation}
		and the desired result follows from \eqref{eq:eps0123} and \eqref{eq:eps-fin}, by letting $\varepsilon_1$, $\varepsilon_2$ and $\varepsilon_3$ going to zero.
	\end{proof}
	
	\begin{lemma}\label{lemma:tau-dev}
		Let $h_\star=h_{\star,n}$ be any sequence such that $h_\star\to\infty$ and $h_\star=o(\log(n))$. Then
		\begin{equation}
			\mathbf{\tilde P}^{\tilde \mu}_\partial(\tau_{\rm dev}>h_\star)\overset{\P}{\longrightarrow}0\ .
		\end{equation}
	\end{lemma}
	
	\begin{proof}
		Recall the definition of the set $M$ in \eqref{eq:def-M}. For any positive $r$, define
		\begin{equation*}
			\cR_r = \bigcup_{\x\in M} \bigcap_{\ell=0}^{r-1} \{\tilde X_{\tau_\partial^{(\ell)}+1}=\x\}\ ,
		\end{equation*}
		where we define $$\tau^{(j)}_\partial\coloneqq\inf\{t>\tau^{(j-1)}_\partial \mid \tilde{X}_t\in\partial \}\,,\qquad j\in\mathbb{N}\,. $$
		In words, $\cR_r$ reads as follows: the first $r$ times in which the process visits $\partial$ (including time zero), it exits $\partial$ by reaching (in a single step) some $\x=(x_1,x_2)\in M$. Recall that if $\x=(x_1,x_2)\in M$ it is always possible to distinguish a \emph{chasing} and an \emph{escaping} starting point between $x_1$ and $x_2$, which we refer to as $x_c$ and $x_e$. 
		Notice that, $\P-{\rm a.s.}$,
		\begin{equation}
			\mathbf{\tilde{P}}^{\tilde \mu}_\partial(\cR_r)=\tilde \mu(V_\star)^r\,.
		\end{equation}
		Therefore,
		\begin{align}\label{eq:dev-meet0}
			\mathbf{\tilde{P}}^{\tilde \mu}_\partial(\tau_{\rm dev}>h_\star)\le (1-\tilde\mu(V_\star)^r)+\mathbf{\tilde{P}}^{\tilde \mu}_\partial(\tau_{\rm dev}>h_\star\,,\, \cR_r)\,.
		\end{align}
		Notice that under $\cR_r$ the probability that $\partial$ is visited more than $\ell$ times before $\tau_{\rm dev}$ is exponentially small in $\ell$. Indeed, for any $\ell< r$,
		\begin{equation}
			\begin{split}
				\mathbf{\tilde{P}}^{\tilde \mu}_\partial\left(\tau_\partial^{(\ell)} < \tau_{\rm dev},\, \cR_r\right)&= \mathbf{\tilde{P}}^{\tilde \mu}_\partial\left(\tau_\partial^{(1)} < \tau_{\rm dev},\, \cR_2 \right)^\ell \le \mathbf{\tilde{P}}^{\tilde \mu}_\partial\left(\tau_{\rm dev}>2,\, \cR_2 \right)^\ell \\
				&\qquad\qquad\qquad\qquad\quad\quad\:\:\:\leq \left(\frac12 \left(\frac{1}{d_{\rm min}^+} +1 \right)\right)^\ell\le\left(\frac34 \right)^\ell\,,
				\label{eq:tauDelta-K(r)-1}
			\end{split}
		\end{equation}
		as, under $\cR_r$, the event $\tau_{\rm dev}>2$ starting at $\partial$ requires that: at the step right after the exit from $\partial$, either the chasing walk reaches the escaping one, or the escaping one moves forward in the tree. Moreover, by definition of the process $\tilde X$, $\tilde{\mathbf{P}}^{\tilde\mu}_\partial(\tilde{X}_1=\partial,\,\cR_2)=0$, since in order to see a transition $\partial\to\partial$ in the $\tilde\mu$-reset case one has to sample a vertex having a self-loop, hence not a member of $V_\star$.
		
		Therefore, for every $\ell< r$,
		\begin{equation}\label{eq:dev-meet1}
			\mathbf{\tilde{P}}^{\tilde \mu}_\partial\left(\tau_{\rm dev}>h_\star,\, \cR_r\right)\le \mathbf{\tilde{P}}^{\tilde \mu}_\partial\left(\tau_\partial^{(\ell)} > \tau_{\rm dev}>h_\star,\, \cR_r\right)+ \left(\frac{3}{4}\right)^\ell\,.
		\end{equation}
		Notice that, under the event $\{\tau_\partial^{(\ell)} > \tau_{\rm dev}>h_\star,\, \cR_r \}$ it must exist some $j\le \ell$ such that the $j$-th excursion from $\partial$ to $\partial$ lasted at least $h_\star/\ell$ steps and $\tau_{\rm dev}$ did not  realize in such an interval of time. Hence, thanks to Lemma \ref{lemma:tau-dev-tau-meet}, if $\ell=o(h_\star)$,
		\begin{equation}\label{eq:dev-meet2}
			\tilde{\mathbf{P}}^{\tilde \mu}_\partial\left(\tau_\partial^{(\ell)} > \tau_{\rm dev}>h_\star,\, \cR_r\right)\le \tilde{\mathbf{P}}^{\tilde\mu}_{\partial}\left(\tau_{\rm dev}\wedge \tau_{\partial}^+\ge \frac{h_\star}{\ell}\right)=o_\P(1)\,.
		\end{equation}
		Plugging \eqref{eq:dev-meet1} and \eqref{eq:dev-meet1} into \eqref{eq:dev-meet0}, choosing, e.g., $\ell=\sqrt{h_\star}$, we finally obtain
		\begin{equation}
			\tilde{\mathbf{P}}^{\tilde\mu}_\partial(\tau_{\rm dev}>h_\star)\le \left(\frac{3}{4}\right)^{\sqrt{h_\star}}+o_\P(1)\,,
		\end{equation}
		from which the desired result follows.
	\end{proof}
	\begin{lemma}\label{lemma:meeting-distant-particles}
		Consider the (random) set
		$$D\coloneqq\{(x,y)\in V_\star^2\mid \cB^+_x(\hbar)\cap\cB^+_y(\hbar)=\emptyset\}\,.$$ 
		For all $n$ large enough it holds
		\begin{equation}
			\max_{(x,y)}\,\P\left(\ind_{(x,y)\in D}\mathbf{P}^{\otimes 2}_{(x,y)}(\tau_{\rm meet}<S)>\varepsilon \right)\le n^{-3} \,,
		\end{equation}
		where $S$ is defined in \eqref{eq:def-S}.
	\end{lemma}
	\begin{corollary}\label{coro:mu-dev}
		It holds
		\begin{equation}
			\max_{\x\in {\rm supp}(\nu_{\rm dev})}\tilde{\mathbf{P}}^{\tilde\mu}_{\x}(\tau_{\partial}<S)\overset{\P}{\longrightarrow}0\,,
		\end{equation}
		where $\nu_{\rm dev}\in\cP(\tilde V\setminus\{\partial\})$ is defined in \eqref{eq:nu_dev}.
	\end{corollary}
	\begin{proof}[Proof of Corollary \ref{coro:mu-dev}]
		The result follows by coupling $\tilde{\mathbf{P}}$ and $\mathbf{P}^{\otimes2}$ up to $\tau_\partial$, using Lemma \ref{lemma:meeting-distant-particles}, a union bound on $(x,y)\in D$, and the fact that the support of $\nu_{\rm dev}$ is contained in $D$ by definition.
	\end{proof}
	\begin{proof}[Proof of Lemma \ref{lemma:meeting-distant-particles}]
		Fix some constant $\kappa\in\N$. By the generalized Markov inequality we have
		\begin{equation}
			\max_{(x,y)}\P(\ind_{(x,y)\in D}\mathbf{P}^{\otimes 2}_{(x,y)}(\tau_{\rm meet}<S)>\varepsilon )\le\frac{\max_{(x,y)}\E[\ind_{(x,y)\in D}\mathbf{P}^{\otimes 2}_{(x,y)}(\tau_{\rm meet}<S)^\kappa]}{\varepsilon^\kappa} \ . 
		\end{equation}
		Call $\sigma=\sigma_{x,y}$ a realization of the complete out-neighborhood of $x$ and $y$ up to length $\hbar$, Notice that the event $(x,y)\in D$ is $\sigma$-measurable. With a slight abuse of notation we say that $\sigma\in D$ if $(x,y)\in D$ under $\sigma$. Then, it is enough to show that there exists some constant $
		\kappa\in\N$ such that
		\begin{equation}
			\max_{(x,y)}\max_{\sigma\in D}\E[\mathbf{P}^{\otimes 2}_{(x,y)}(\tau_{\rm meet}<S)^\kappa\mid \sigma]\le n^{-3}\ ,
		\end{equation}
		for all $n$ large enough. To this aim, we use another (multiple) \emph{annealing argument}. In particular, let the law $\P^{\otimes2,\,\rm \kappa\text{-an}|\sigma}_{(x,y)}$ refer to the non-Markovian process introduced in Remark \ref{rmk:multiple-annealed-3} with initial environment $\sigma$.
		In such a probability space we consider the events
		\begin{equation}
			\cE_i=\{ \exists t\le S \text{ s.t. }W^{(i,1)}_t=W^{(i,2)}_t\}\ ,\qquad i\le \kappa\ ,
		\end{equation}
		and we notice that
		\begin{equation}\label{eq:annealing-K}
			\max_{(x,y)}\max_{\sigma\in D}\E[\mathbf{P}^{\otimes 2}_{(x,y)}(\tau_{\rm meet}<S)^\kappa\mid \sigma]=\max_{(x,y)}\max_{\sigma\in D}\P^{\otimes2,\,\kappa\text{-an}|\sigma}_{(x,y)}( \cap_{i\le \kappa}\cE_i)\ .
		\end{equation}
		We now show that there exists some $\delta>0$ such that
		\begin{equation}\label{eq:cond-ann-prob}
			\max_{j\le\kappa}\max_{(x,y)}\max_{\sigma\in D}\P^{\otimes2,\,\kappa\text{-an}|\sigma}_{(x,y)}( \cE_{j}\mid \cap_{i< j}\cE_i)\le n^{-\delta}\,,
		\end{equation}
		for all $n$ large enough. Then, from \eqref{eq:annealing-K} and \eqref{eq:cond-ann-prob} follows that
		\begin{equation}\label{eq:annealing-K-final}
			\max_{(x,y)}\max_{\sigma\in D}\E[\mathbf{P}^{\otimes 2}_{(x,y)}(\tau_{\rm meet}<S)^\kappa\mid \sigma]=\max_{(x,y)}\max_{\sigma\in D}\prod_{j=1}^\kappa\P^{\otimes2,\,\kappa\text{-an}|\sigma}_{(x,y)}( \cE_{j}\mid \cap_{i< j}\cE_i)\le n^{-\kappa\,\delta}\ .
		\end{equation}
		and the desired result follows by taking $\kappa$ large enough, e.g., such that $\kappa\,\delta> 3$. 
		
		We are left with proving \eqref{eq:cond-ann-prob}. Fix any $(x,y)\in[n]^2\setminus\Delta$, and fix any $\sigma=\sigma_{x,y}$ such that $\sigma\in D$. Then, when $i=1$ and $t=0$ we have a subgraph of $G$ consisting of two out-going trees of length $\hbar$, rooted at $x$ and $y$, that do not intersect. We let the first couple of walks evolve for time $S$ and bound from above the probability of the event $\cE_1$. Notice that, by the fact that $\sigma\in D$ we have
		\begin{equation}\label{eq:not-too-soon}
			W^{(i,1)}_s\neq W^{(i,2)}_s\qquad  \forall s< \hbar,\:i\le \kappa\qquad \P^{\otimes2,\,\kappa\text{-an}|\sigma}_{(x,y)}-{\rm a.s.}\,.
		\end{equation}
		In order for the event $\cE_1$ to occur, one of the two walks has to exit its $\sigma$-tree and hit one of the vertices that are already discovered, that is, that have at least a head or a tailed matched. Since there are at most $S$ steps available and, uniformly on the past, the chance that such a hitting takes place at a fixed time $t\leq S$ is bounded by
		\begin{equation}\label{eq:p1}
			w_1\coloneqq\frac{d_{\rm max}^-[2(d_{\rm max}^+)^{2\hbar}+2S]}{m-2(d_{\max}^+)^\hslash-2S}\,,
		\end{equation}
		as at any time $t$, for each of the two walks, there are at most $2(d_{\rm max}^+)^{2\hbar}+2S$  vertices that have been explored already. We conclude that
		\begin{equation}
			\P^{\otimes2,\,\kappa\text{-an}|\si}_{(x,y)}(\cE_1)\le S\,w_1\,.
		\end{equation}
		Now condition on an arbitrary realization of the paths of the first $i-1$ couples of walks in which $\cap_{j<i}\cE_j$ is realized. Notice that, for $z=x,y$, within the leafs of $\cB_z^+(\hbar)$ there are at most $(2S)^{i-1}\le  (2\log(n))^{3i}$ having at least a matched tail. Hence, the probability that the walk that starts at $z$ exits $\cB^+_z(\hbar)$ at a leaf that has already been discovered can be bounded (uniformly in $(x,y)$ on $\si$, and on the behavior of the previous couples of walks) for all $n$ large enough, by
		\begin{equation}
			\bar w\coloneqq (d_{\min}^+)^{-\hbar} (2\log(n))^{3\kappa}\le \frac12 n^{-2/5}\,,
		\end{equation}
		thanks to the definition of $\hslash$ in \eqref{eq:def-h-bar} and the fact that $d_{\min}^+\ge 2$.
		Indeed, being $\cB_z^+(\hslash)$ a tree, the probability that the $i$-th walk ends up in specific leaf is given by the probability to follow the unique path from $z$ to that leaf.
		
		If the walk that start at $z$ exists $\cB^+_z(\hbar)$ from a leaf that has no matched tails, then the argument before \eqref{eq:p1} applies, and this time the probability of a meeting time before $S$ can be bounded by
		\begin{equation}
			w_i\coloneqq\frac{d_{\max}^-[2(d^+_{\rm max})^{2\hbar}+2iS]}{m-2(d_{\max}^+)^{\hslash}-2iS}\le n^{-2/5}\ ,\qquad i\le \kappa\ ,
		\end{equation}
		for all $n$ large enough.
		In conclusion, for all $n$ large enough,
		\begin{equation}
			\max_{j\le \kappa}\max_{(x,y)}\max_{\sigma\in D}\P^{\otimes2,\,\kappa\text{-an}|\sigma}_{(x,y)}( \cE_{j}\mid \cap_{i< j}\cE_i)\le \bar w+S\,w_\kappa\le 2n^{-2/5}\,,
		\end{equation}
		and \eqref{eq:cond-ann-prob} readily follows.
	\end{proof}

	We are now in shape to prove \eqref{eq:1st-rhs-mixing3}, and in turn conclude the proof of Proposition \ref{prop:mixing}.
	\begin{proof}[Conclusion of the Proof of \eqref{eq:1st-rhs-mixing3}]
		Recall that the function $\varphi$ in \eqref{eq:def-varphi} and notice that $\varphi$ is injective on ${\rm supp}(\nu_{\rm dev})$, hence there is no ambiguity in identifying $\nu_{\rm dev}$ with its lifting in $[n]^2$. 
		We can rewrite the first term on the right-hand side of \eqref{eq:mixing 3} as follows
		\begin{align*}
			\|\tilde{P}^{\frac{S}2}\left(\partial,\cdot\right) - \tilde{\pi}\|_{\rm{TV}} 
			& = \sup_{A\subset \tilde{V}} \left|\tilde{\mathbf{P}}^{\tilde \mu}_\partial\left(\tilde{X}_{S/2}\in A, \tau_{\rm{dev}} >\frac{S}{4} \right) + \tilde{\mathbf{P}}^{\tilde \mu}_\partial\left(\tilde{X}_{S/2}\in A, \tau_{\rm{dev}} \leq \frac{S}{4}\right) - \tilde{\pi}(A) \right| \\
			&\leq \tilde{\mathbf{P}}^{\tilde \mu}_\partial\left( \tau_{\rm{dev}} > \frac{S}4\right) +  \sup_{\ell\le \frac{S}4} \sup_{A\subset \tilde{V}} \left|\tilde{\mathbf{P}}^{\tilde \mu}_{\nu_{\rm{dev}}}\left(\tilde{X}_{\frac{S}{2}-\ell}\in A\right) - \tilde{\pi}(A)\right| \\
			&= \tilde{\mathbf{P}}^{\tilde \mu}_\partial\left( \tau_{\rm{dev}} > \frac{S}4\right) +   \left\|\nu_{\rm dev}\tilde{P}^{\frac{S}{4}} - \tilde{\pi}\right\|_{\rm TV} \\
			&\leq \tilde{\mathbf{P}}^{\tilde \mu}_\partial\left( \tau_{\rm{dev}} >\sqrt{\log(n)}\right) +    \bigg\| \nu_{\rm{dev}}\left(P^{\otimes 2} \right)^{\frac{S}4}- \pi^{\otimes2}\bigg\|_{\rm{TV}} + \left\|\nu_{\rm{dev}}\left(P^{\otimes2}\right)^{\frac{S}4} - \nu_{\rm{dev}}\tilde{P}^{\frac{S}{4}}\right\|_{\rm{TV}} \,,
		\end{align*}
		where the last step follows by the triangular inequality.
		
		The first two terms on the right-hand side of the last display are both $o_\P(1)$, thanks to Lemma \ref{lemma:tau-dev} and Theorem \ref{th:cutoff}, respectively.
		Moreover, the last TV-distance appearing on the right-hand side can be bounded by a coupling argument: using the same source of randomness to generate the initial position according to $\nu_{\rm dev}$, and the two processes $P^{\otimes 2}$ and $\tilde{P}$ up to time $\tau_{\rm meet}$ (and then letting the chain evolving independently), and call $\mathbf{Q}$ such a (random) coupled probability law. In particular, consider the random variable $\mathbf{Q}(\tau_{\rm fail}\le s)$, that correspond to the (random) probability that the coupling fails before $s$, and therefore, by the definition of the TV-distance in terms of optimal coupling, it follows that for every $s\ge 0$,
		$$\mathbf{Q}(\tau_{\rm fail}\le s)\ge \left\|\nu_{\rm{dev}}\left(P^{\otimes2}\right)^{s} - \nu_{\rm{dev}}\tilde{P}^{s}\right\|_{\rm{TV}},\qquad \P-{\rm a.s.}.$$
		Thanks to Corollary \ref{coro:mu-dev} we have,
		\begin{equation}
			\mathbf{Q}(\tau_{\rm fail}\le S/4)\le\mathbf{Q}(\tau_{\rm fail}\le S)\le\max_{\x\in{\rm supp}(\nu_{\rm dev})}\tilde{\mathbf{P}}^{\tilde{\mu}}_{\x}(\tau_{\partial}\le S)=o_{\P}(1)\,, 
		\end{equation}
		from which \eqref{eq:1st-rhs-mixing3} follows.
	\end{proof}
	
	\section{Conclusions and open problems}\label{sec:open}
	With this work we provide the first-order asymptotic of the expected meeting, consensus and coalescence time on a typical sparse random digraphs from the DCM ensemble. This result adds to the list of examples, presented in the Introduction, for which such a precise result can be obtained. Moreover, to the best of our knowledge, it is the only degree-inhomogeneous class of graphs in which such results has been obtained so far. We conclude the paper with a short list of open problems and with an outlook on possible future research. 
	
	As pointed out in Section \ref{suse:degreeassumptions}, in this paper we did not aim at understanding the most general conditions on the degree sequence up to which our technique could be stretched. Nevertheless, it is natural to ask this question. More importantly, it looks clear from our results that a change in the scaling of the consensus time should be attained as soon as the in-degree sequence does not admit a bounded second moment. Unfortunately, the understanding of the random walk on the DCM with such extreme heavy-tailed degrees is still completely open, and some new ideas are needed to carry out this investigation.
	
	The question of determining the preconstant for the consensus time of the undirected configuration model remains open. The representation proposed in \cite[Lemma 6.12]{HSDL22} looks promising, but an explicit characterization is still missing. In this respect, as mentioned in Section \ref{sec:examples}, non-rigorous results in \cite{SAR08} seems to suggest a phenomenology similar to that in the Eulerian case. In fact, in view of the mean field conditions and the First Visit Time Lemma, it is not hard to realize that the value of $\vartheta$ in Eulerian case should provide a lower bound for its undirected counterpart. Indeed, being the stationary distribution unaltered, the only thing to check is that the quantity $R_T(\partial)$ in Proposition \ref{prop:return-diag} should be larger in the latter case, due to the backtracking feature of the random walk on sparse undirected graphs.
	
	\subsection*{Acknowledgements}
	M.Q. thanks German Research Foundation for financial support (project number 444084038, priority program SPP2265).
	The work of L.A., R.S.H. and F.C. is supported in part by the Netherlands Organisation for Scientific Research (NWO) through the Gravitation {\sc Networks} grant 024.002.003. The work of F.C. is further supported by the European Union's Horizon 2020 research and innovation programme under the Marie Sk\l odowska-Curie grant agreement no.\ 945045. \hfill
	\parbox{0.1\textwidth}
	{~~~~\includegraphics[width=0.05\textwidth]{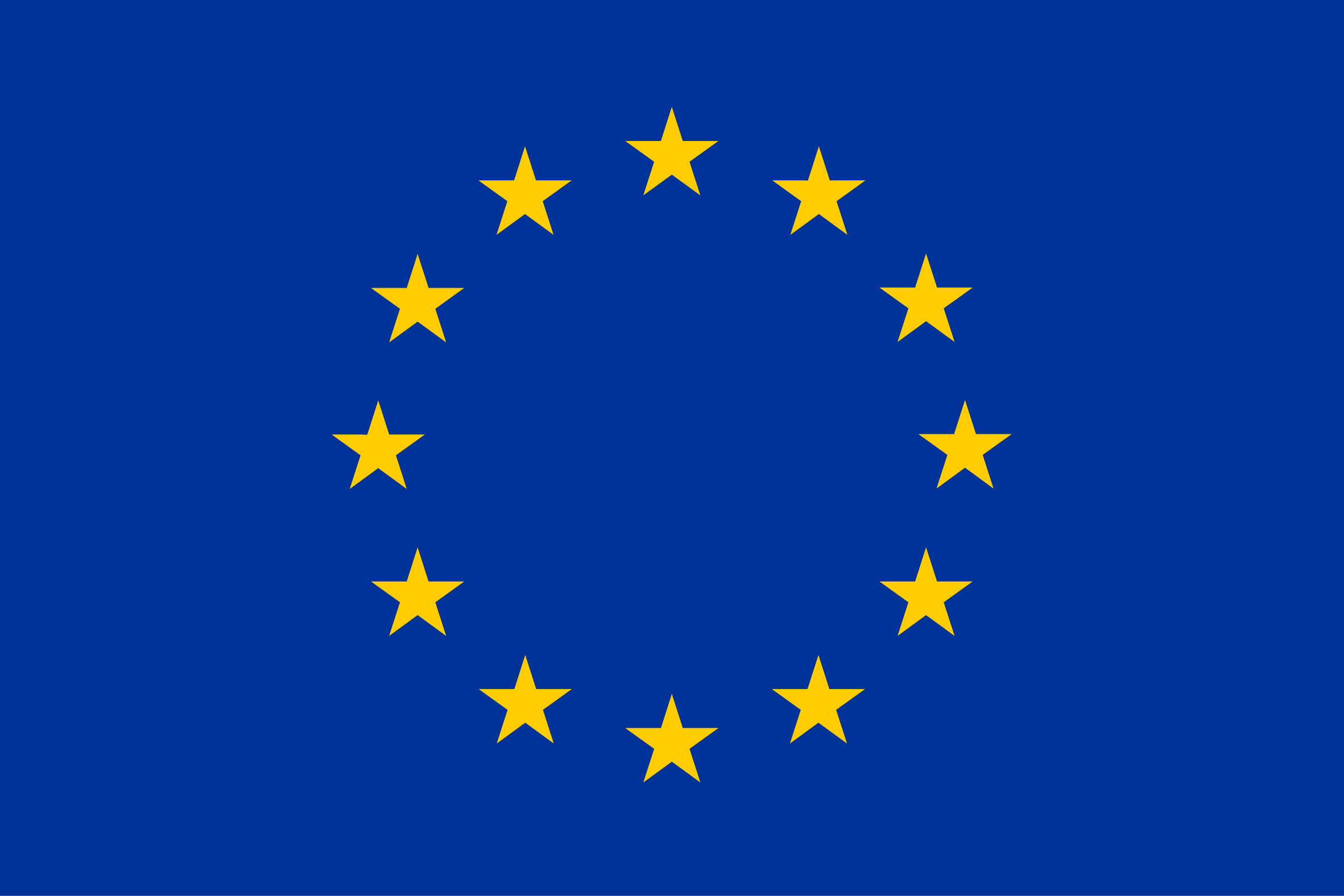}}
	
\bibliographystyle{imsart-number} 
\bibliography{reference}       
\end{document}